\documentclass[11pt]{article}
\usepackage{amsmath,amsfonts,amssymb}
\usepackage{amsthm}
\usepackage{xcolor}
\usepackage{enumitem}
\usepackage[english]{babel}

\usepackage{datetime2}
\usepackage{dsfont}
\usepackage{caption}
\usepackage{booktabs}
\usepackage{mathtools}
\usepackage{float}
\usepackage{graphicx}
\usepackage{subfig}
\usepackage{comment}
\setitemize{itemsep=0.1pt}
\setenumerate{itemsep=0.1pt}

\usepackage{datetime2}

\renewcommand{\baselinestretch}{1.236}

\RequirePackage[colorlinks,citecolor=blue,urlcolor=blue,backref = page]{hyperref}

\newtheorem{proposition}{Proposition}[section]
\newtheorem{theorem}[proposition]{Theorem}
\newtheorem{corollary}[proposition]{Corollary}
\newtheorem{lemma}[proposition]{Lemma}
\newtheorem{remark}[proposition]{Remark}

\newtheorem{example}[proposition]{Example}

\newcommand{\nc}{\newcommand}
\nc{\md}{\mathrm{d}}
\nc{\I}{{\mathbf 1}}
\nc{\bU}{\mathbb{U}}
\nc{\E}{\mathbb{E}}
\nc{\bN}{{\mathbf N}}
\nc{\bfM}{{\mathbf M}}
\nc{\cB}{{\mathcal B}}
\nc{\cC}{{\mathcal C}}
\nc{\cK}{{\mathcal K}}
\nc{\cL}{{\mathcal L}}
\nc{\R}{{\mathbb R}}
\nc{\M}{{\mathcal M}}
\nc{\N}{{\mathbb N}}
\nc{\cN}{{\mathcal N}}
\nc{\Z}{{\mathbb Z}}
\nc{\bF}{{\mathbf F}}
\nc{\bS}{{\mathbf S}}
\nc{\bbS}{{\mathbb S}}
\nc{\tc}{\tilde{c}}
\nc{\hC}{\tilde{C}}
\nc{\hc}{\tilde{c}}
\nc{\tphi}{\tilde{\phi}}
\nc{\tPhi}{\tilde{\Phi}}
\nc{\Psif}{\Psi^{!}}
\nc{\psif}{\psi^{!}}
\nc{\tbN}{\tilde{\mathbf{N}}}
\nc{\tx}{\tilde{x}}
\nc{\ty}{\tilde{y}}
\nc{\talpha}{\tilde{\alpha}}

\DeclareMathOperator{\conv}{conv}
\DeclareMathOperator{\reach}{reach}
\DeclareMathOperator{\tr}{tr}
\nc{\BP}{\mathbb{P}}
\nc{\BE}{\mathbb{E}}
\nc{\BQ}{\mathbb{Q}}
\nc{\RR}{{\sf R}}

\DeclareMathOperator*{\argmin}{arg\,min}
\DeclareMathOperator{\Nor}{Nor}

\DeclareMathOperator{\Nsf}{{\sf N}}

\renewcommand{\d}[1]{ \mathrm{d}{#1} }
\newcommand{\1}{\mathds{1}}
\newcommand{\ten}[3]{$\Phi_#3^{#1,#2}$}
\newcommand{\EP}[4]{$\mathbb{E}\left[ V_{#1}\left(P_{{#3},{#4}}^{#2}\right) \right]$}

\newcommand{\nn}{\ell}

\numberwithin{equation}{section}

\nc{\ch}[1]{\textcolor{red}{#1}}
\usepackage{soul,xcolor}
\setstcolor{red}

\begin{document} 

\renewcommand{\thefootnote}{\fnsymbol{footnote}}
\author{D. Hug\footnotemark[1], M.A. Klatt\footnotemark[2], and D. Pabst\footnotemark[1]\hspace{5pt}\footnotemark[3]}
\footnotetext[1]{Karlsruhe Institute of Technology, Institute of Stochastics, 76131 Karlsruhe, Germany}
\footnotetext[2]{German Aerospace Center (DLR), Institute for AI Safety and Security, Wilhelm-Runge-Str. 10, 89081 Ulm, Germany; German Aerospace Center (DLR), Institute for Material Physics in Space, 51170 Köln, Germany; Department of Physics, Ludwig-Maximilians-Universität München, Schellingstr. 4, 80799 Munich, Germany}
\footnotetext[3]{Friedrich Alexander University Erlangen-Nuremberg, Institute of Theoretical Physics, Staudtstr. 7, 91058 Erlangen, Germany}

\title{Minkowski tensors for point clouds and voxelized data:\\ robust, asymptotically unbiased estimators}
\date{}
\maketitle

\begin{abstract} \noindent %
Minkowski tensors, also known as tensor valuations, provide robust
$n$-point information for a wide range of random spatial structures.
Local estimators for point clouds, e.g., representing voxelized data, however,
are unavoidably biased even in the limit of infinitely high resolution.
Here, we substantially improve a recently proposed, asymptotically
unbiased algorithm to estimate Minkowski tensors from point clouds.
Our improved algorithm is more robust and efficient. Moreover we
generalize the theoretical foundations for an asymptotically bias-free
estimation of the interfacial tensors, among others, to the case of finite unions of
compact sets with positive reach, which is relevant for many
applications like rough surfaces or composite materials. As a realistic
test case of random spatial structures, we consider random (beta) polytopes. We first
derive explicit expressions of the expected Minkowski tensors, which we
then compare to our simulation results. We obtain precise estimates with
relative errors of a few percent for practically relevant resolutions.
Finally, we apply our methods to real data of metallic grains and
nanorough surfaces, and we provide an open-source python package, which
works in any dimension.
\end{abstract}

\noindent
{\em Keywords:} tensor valuations, microstructure characterization, anisotropy, digitized image analysis

\vspace{0.2cm}
\noindent
2020 MSC: 94A08 · 68U10 · 60D05 · 53C65 · 28A75 · 62H35 · 52A22 


\newpage

\renewcommand{\baselinestretch}{0.8}\normalsize
	     {\hypersetup{linkcolor=black}\small
               \tableofcontents
}	\renewcommand{\baselinestretch}{1.1}\normalsize

\section{Introduction}
\label{sec:intro}

Random spatial structures appear ubiquitously in nature and technology.
Examples (and their corresponding mathematical models) are heterogeneous
materials and porous media (random sets)~\cite{ ohser_statistical_2000,
  torquato_random_2002, ohser_3d_2009, adler_fractured_2013,
armstrong_porous_2019}, cellular tissues and foam-like structures
(tessellations)~\cite{ okabe_spatial_2000-1, rath_strength_2008,
klatt_cell_2017, stinville_multi-modal_2022}, rough surfaces (random
fields)~\cite{adler_random_2007, vanmarcke_random_2010,
spengler_strength_2019, rottger_contactengineeringcreate_2022}, and
bacterial colonies (particle processes)~\cite{chiu_stochastic_2013,
hansen_theory_2013, ziegel_estimating_2015} as well as error-correcting
code (high-dimensional sphere packings)~\cite{ zong_sphere_1999,
cohn_new_2003}. This variety of examples corresponds to a similar
diversity in random shapes.

In all of these cases, Minkowski tensors (also known as tensor
valuations)~\cite{schneider_stochastic_2008,JK2017} provide a
comprehensive, unified shape analysis. The Minkowski
tensors~\cite{schroder-turk_minkowski_2011} are generalizations of the
well-known Minkowski functionals (or intrinsic volumes) from integral
geometry~\cite{schneider_stochastic_2008}; see Section~\ref{sec:prelim}
for more details and the theoretical background. The intrinsic volumes
represent intuitive geometric information, e.g., on the volume, surface
area, or curvature; the Minkowski tensors additionally distinguish
different orientations with respect to these geometrical
properties~\cite{klatt_mean-intercept_2017}; see also
Fig.~\ref{fig:intro_schematic}.

\begin{figure}[t]
  \centering
  \includegraphics[width=\textwidth]{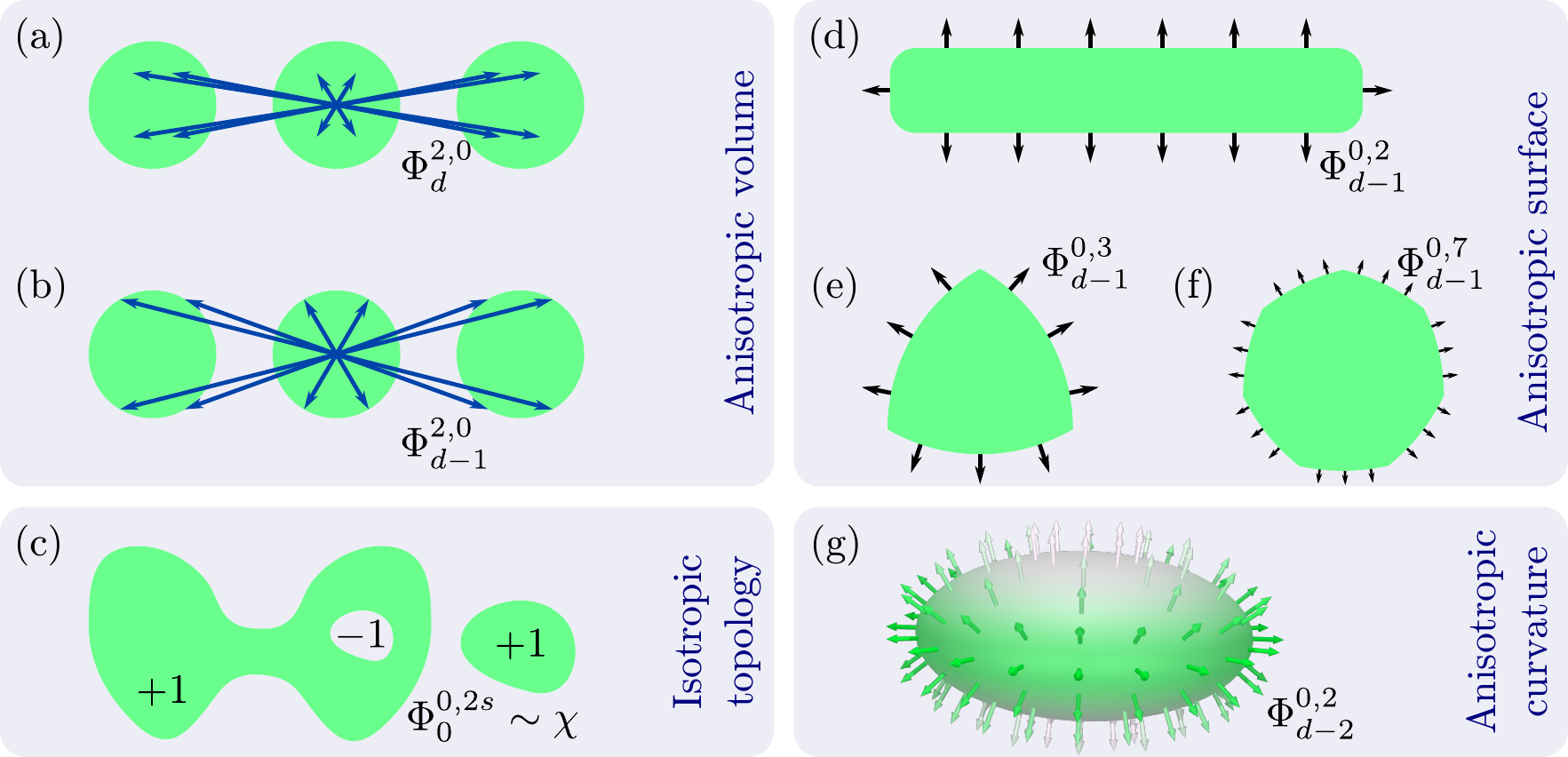}
  \caption{Distinct Minkowski tensors quantify different types of
    anisotropy: (a--b) moment tensors of the volume and boundary
    distributions capture the anisotropy of disk arrangements; (c) the
    Euler characteristic $\chi$ is a topological constant,
    generalizations to surface tensors $\Phi_0^{0,2s}$ are proportional
    to $\chi$ and the metric tensor and hence always isotropic, see
    \eqref{phi002s}; (d--f) anisotropy of the surface‑normal
    distribution is measured by surface tensors, different symmetries
    are captured by tensors of a corresponding rank; (g) curvature
    anisotropy is incorporated by an additional curvature
    function to the surface tensors.}
  \label{fig:intro_schematic}
\end{figure}

Minkowski tensors of different ranks capture symmetries of all orders.
They quantify the degree of anisotropy as well as the preferred
directions. They are comprehensive in the sense of Hadwiger's and
Alesker's characterization theorems~\cite{schneider_stochastic_2008},
intuitively speaking, they contain all motion-covariant, continuous, and
additive shape information. Since they are additive, Minkowski
functionals and tensors provide robust access to $n$-point
information~\cite{mecke_robust_1994}. Moreover, the definition in real
space allows for a convenient treatment of boundary
conditions~\cite{schroder-turk_tensorial_2010}.

The versatility of Minkowski functionals and tensors is demonstrated by
their ample successful applications in a wide range of fields, including
statistical physics~\cite{mecke_integral_1998, mecke_statistical_2000,
klatt_characterization_2022}, biology~\cite{beisbart_extended_2006,
barbosa_integral-geometry_2014, barbosa_novel_2019}, spatial
statistics~\cite{ziegel_estimating_2015, ebner2018}, image
analysis~\cite{ohser_statistical_2000, ohser_3d_2009}, as well as
astronomy and cosmology~\cite{gott_topology_1990,
schmalzing_minkowski_1998, joby_search_2019, klatt_detecting_2020,
collischon_tracking_2021, collischon_morphometry_2024}. Explicit
examples of random spatial structures characterized via Minkowski
functionals and tensors are porous media~\cite{arns_reconstructing_2003,
  Arns2010, klatt_anisotropy_2017, armstrong_porous_2019,
jiang_fast_2020}, composite
materials~\cite{ernesti_characterizing_2020}, nanorough
surfaces~\cite{spengler_strength_2019}, fluid
demixing~\cite{bobel_kinetics_2016}, and trabecular
bone~\cite{rath_strength_2008, klatt_mean-intercept_2017,
callens_local_2021}.

A challenge, however, is the application of Minkowski tensors to
point clouds, e.g., representing voxelized data, where
a systematic bias can often not be avoided even in the
limit of infinitely high resolution. In contrast, for voxelized
gray-scale data, marching-cube algorithms can create triangulated meshes with
continuous orientations, which results in asymptotically bias-free
algorithms~\cite{MantzJacobsMecke2008, legland_computation_2011,
svane_estimation_2014-1, svane_estimation_2014}. If no gray-scale data
is available, any local algorithm for Minkowski tensors using
black-and-white voxels will be asymptotically biased, i.e., the estimator will
not converge to the true value in the limit of infinite
resolution~\cite{jensen_valuations_2017}.

An algorithmic solution to the problem was recently proposed in
\cite{hug_voronoi-based_2017}, which had been inspired by
\cite{MR2594445,5669298}. A more general framework, admitting to work with
general distance-like functions (such as power distances and,  
correspondingly, power cells) to improve the robustness of estimation
procedures against outliers, while preserving the resilience with
respect to Hausdorff noise, was explored in \cite{CLMT15}; for a discussion of and
references to related methods, see \cite{CLMT15}. The key idea
of the algorithmic solution in \cite{hug_voronoi-based_2017} is to
construct Voronoi cells that can encode global information and thus
guarantee asymptotically bias-free estimators of all Minkowski tensors
for sets of positive  reach. This requirement of a positive reach,
however, excludes many interesting applications, like rough surfaces and
composite materials. More importantly, the proposed algorithm suffered
from numerical instabilities since it relies on a matrix inversion,
which is slow and sensitive to rounding errors, especially for
ill‑conditioned matrices.

Here, we substantially advance the robustness of the algorithm by
replacing the matrix inversion by a least squares fit. The latter avoids
numerical instabilities and allows for more data from parallel sets than
needed, which further enhances robustness against statistical and
systematic errors. Furthermore, we increase its efficiency by using an
unbiased estimator of the Voronoi tensors (also called Voronoi volume
integrals). This revised algorithm, here called Voronoi-LSQ algorithm,
is also asymptotically unbiased for sets of positive reach, as we
rigorously prove; see Remark \ref{rem:citeintro}. We demonstrate its
robustness and accuracy for exemplary cases of geometric shapes,
stochastic models, and real data.

Even for a set without positive reach, our algorithm achieves
surprisingly good results with relative deviations of less than 3\% in
the eigenvalues of the interfacial tensors, see Table
\ref{table:cuttedRect}, at least for the simple example of a rectangular
shell. Since such precision is not theoretically substantiated, we also
introduce an alternative method, here called Voronoi-FD algorithm, to
estimate interfacial tensors, for which we are able to prove that its
corresponding estimator is asymptotically unbiased for a large class of
sets, including parallel sets of compact sets or finite unions of
compact sets with positive reach (see Theorem \ref{cor4.5:convergence}).

Together with the paper, we publish an open-source python package \cite{VorominkCode}
including both methods. It is the first implementation of Minkowski
tensors that we know of that can be applied to any dimension. It can
also be applied to any discrete representation of the set, i.e., not
only for pixelated images but representations using non-cubic lattice or
even random point patterns. The last generalization is especially
helpful in high dimensions, where,  similarly to numerical integration,
lattice representations may converge too slowly. 

In Section~\ref{sec:prelim}, we provide an introduction to integral
geometry and specifically to Minkowski tensors. In
Section~\ref{sec:asymptotics}, we first further explain and illustrate
the main theoretical foundations from geometric measure theory and then
establish two results required for the algorithms presented in Section
\ref{sec:Algorithms}. Details on the algorithm, its theoretical
foundation,  implementation, and suitable choices of parameters are
discussed in Sections~\ref{sec:VoronoiTensorEstimation}--\ref{sec:choice}.

We demonstrate the reliability of our methods for simple geometric test
cases in Sections~\ref{sec:ConvexTests} and \ref{sec:NonconvexTests}
(see also Table \ref{table:Cor4.1}). For a stochastic test case, we
apply our estimator to a classic example of random convex sets,
isotropic random polytopes (see Section~\ref{sec:beta}), both in theory
(see Section~\ref{sec:betatheo}) and simulations (see
Section~\ref{sec:betasim}).

In Section~\ref{sec:exp}, we turn to experimental data. First, we
analyze metallic grains in a polycrystalline nickel-based superalloy
using data from \cite{stinville_multi-modal_2022}. The Minkowski tensors
characterize the cells as being distinctly more anisotropic with respect
to curvature than the surface area  (see Section~\ref{sec:grains}).
Then, we determine the interfacial tensors of nanorough surfaces from
\cite{spengler_strength_2019} (see Section~\ref{sec:nano}).

\section{Minkowski tensors in integral geometry}
\label{sec:prelim}

In this section, we provide a brief introduction to Minkowski tensors (or tensor valuations) and summarize their basic properties. For basic concepts from convex and integral geometry not introduced here, we refer to \cite{HugWeil2020,S14}. Let $\cK^d$ denote the set of all compact convex subsets (convex bodies) of  $\R^d$. We denote by $\langle\cdot\,,\cdot\rangle$ and $\|\cdot\|$ a Euclidean scalar product and the induced norm. We write $B^d:=\{x\in\R^d:\|x\|\le 1\}$ for the Euclidean unit ball centered at the origin $o$ and $\mathbb{S}^{d-1}=\partial B^d=\{x\in\R^d:\|x\|= 1\}$ for its boundary. We set $B^d(x,r):=x+rB^d$ for $x\in\R^d$ and $r\ge 0$. The volume of $B^d$ is denoted by $\kappa_d=\pi^{d/2}/\Gamma(1+d/2)$ and $\omega_d=d\kappa_d$ is its surface area, i.e., the $(d-1)$-dimensional volume of $\mathbb{S}^{d-1}$. 
The intrinsic volumes $V_i:\cK^d\to [0,\infty)$, $i\in\{0,\ldots,d\}$, are a collection of $d+1$ basic functionals on $\cK^d$ that are distinguished by their properties:
\begin{itemize}
\item They form a basis of the vector space of continuous, isometry  invariant, additive, functions on $\cK^d$; $V_i$ is positively homogeneous of degree $i$ \cite[Theorem 4.20]{HugWeil2020}.
\item They arise as coefficient functionals of a Steiner formula \cite[Theorem 3.10]{HugWeil2020}.
\item They satisfy various integral geometric formulas such as Crofton formulas, projection formulas and (intersectional as well as additive) kinematic formulas \cite[Chap.~5]{HugWeil2020}.
\item They can be additively extended to polyconvex sets (finite unions of compact, convex sets) and play a key role in the classical theory of geometric valuations (see \cite[Chap.~1]{SVJK} and \cite[Chap.~4.5]{HugWeil2020}). 
\end{itemize}
In particular, $V_d$ is the volume functional, $2V_{d-1}$ is the surface area, $V_1$ is proportional to the mean width functional and $V_0$ is the Euler characteristic (see \cite[Section 3.3]{HugWeil2020}).  
The intrinsic volumes have also been called Minkowski functionals (or quermassintegrals, though with a different normalization and notation). Due to their properties, the Minkowski functionals are useful descriptors that can be applied for the analysis of complex spatial structure. 

Since the Minkowski functionals are isometry invariant, their usefulness in characterizing anisotropic features of objects under investigation is limited. The Minkowski tensors (see \cite[Chap.~2]{HSVJK}), which we introduce next, provide a more general set of tensor-valued geometric descriptors that are sensitive to position and orientation of geometric objects in Euclidean space (see Fig.~\ref{fig:intro_schematic} for an illustration). For detailed information on the determination and reconstruction of (classes of) convex bodies from volume or surface tensors we refer to \cite{Kous2017,KK2016,KS2021} and to \cite[Chap.~7]{LNP+}. We start by providing a brief introduction to symmetric tensors and tensor-valued functionals as needed for the present purpose.

In the following, we use the 
scalar product of $\R^d$ to identify $\R^d$ with its dual space, hence a  vector $a\in \R^d$ will be identified with the linear functional
$x\mapsto \langle a, x\rangle$ from $\R^d$ to $\R$. For $r\in \N_0$, an
\emph{$r$-tensor}, or tensor of rank~$r$, on $\R^d$ is  an
$r$-linear mapping from $(\R^d)^r$ to $\R$. 
The vector space of all $r$-tensors of $\R^d$ has dimension $d^r$. If $e_1,\ldots,e_d$ is the standard basis of $\R^d$, then the $r$-tensors  $e_{i_1}\otimes \cdots\otimes e_{i_r}$, $1\le i_1,\ldots,i_r\le d$, are a basis of the vector space of all tensors of rank $r$ in $\R^d$. Here we have $e_{i_1}\otimes \cdots\otimes e_{i_r}(x_1,\ldots,x_r)=\prod_{j=1}^r\langle e_{i_j},x_j\rangle$ for $x_1,\ldots,x_r\in\R^d$.

A tensor is \emph{symmetric} if it is invariant under permutations of its arguments. By
$\mathbb{T}^{r}(\R^d)$, or simply by $\mathbb{T}^r$, we denote the real vector space (with its
standard topology) of symmetric $r$-tensors on $\R^d$. We define 
$\mathbb{T}^{0}={\mathbb R}$, and, by the identification made
above, we have $\mathbb{T}^1=\R^d$. In any case,  $\dim \mathbb{T}^r(\R^d)=\binom{d+r-1}{r}$.

The \emph{symmetric tensor product} $a_1\odot \dots\odot a_k\in\mathbb{T}^{r_1+\dots+r_k}$ of $a_i\in
\mathbb{T}^{r_i}$, $i=1,\dots,k$, is a symmetric tensor of rank $r_1+\cdots+ r_k$. Denoting by ${\mathcal S}(m)$ the group of bijections of $\{1,\ldots,m\}$,   
$s_0:=0$ and $s_i:=r_1+\dots+r_i$, for $i=1,\dots,k$, it is defined by
$$
(a_1\odot \dots\odot a_k)(x_1,\dots,x_{s_k})
:= \frac{1}{s_k!} \sum_{\sigma\in {\mathcal S}(s_k)} \prod_{i=1}^k
a_i(x_{\sigma(s_{i-1}+1)},\dots,x_{\sigma (s_i)})
$$
for $x_1,\dots,x_{s_k}\in {\mathbb R}^d$. 
In this way, the space of symmetric tensors (of arbitrary rank)
becomes an associative, commutative algebra with unit. For symmetric tensors $a,b,a_i$ we will use the abbreviations $a\odot b=:ab$ and  $a_1\odot \dots\odot a_k=: a_1\cdots a_k$, and we write $a^r$ for the $r$-fold symmetric tensor product of $a$, if $r\in\N$, and $a^0:= 1$. 
If $r_1=\ldots=r_k=1$ and $x_1,\ldots,x_k\in\R^d$, then
$$
a_1\cdots a_k(x_1,\ldots,x_k)
=\frac{1}{k!}\sum_{\sigma\in\mathcal{S}(k)} \prod_{i=1}^k \langle a_i,x_{\sigma(i)}\rangle,
$$
and hence, for  $a\in\R^d$ and $r\ge
1$,  the $r$-fold symmetric tensor product of $a$ satisfies
\begin{equation*}
  a^r(x_1,\dots,x_r)=\langle a, x_1\rangle  \cdots \langle a, x_r\rangle. 
\end{equation*}
The scalar product, $Q(x,y) = \langle x, y\rangle$, $x,y\in \R^d$, 
is a symmetric tensor of rank two; we call $Q$ the \emph{metric
  tensor}. Clearly, if $e_1,\ldots,e_d$ is an orthonormal basis of $\R^d$, then $Q=e_1^2+\cdots+e_d^2$. 

Let $(e_1,\dots,e_d)$ be an orthonormal basis of ${\mathbb R}^d$.
Then the symmetric tensors $e_{i_1}\cdots e_{i_r}$ with
$1\le i_1\le\dots\le i_r\le d$ form a basis of $\mathbb{T}^r(\R^d)$. 
The corresponding coordinate representation of $T\in\mathbb{T}^r(\R^d)$ is 
\begin{equation}\label{02-1.2a}
T=\sum_{1\le i_1\le\dots\le i_r\le d} t_{i_1\dots i_r} e_{i_1}\cdots e_{i_r}\quad \text{with} \quad 
t_{i_1\dots i_r} =\binom{r}{m_1\dots m_d}T(e_{i_1},\dots,e_{i_r}),
\end{equation}
where 
$m_k:=|\{\ell\in\{1,\ldots,r\}:i_\ell=k\}|$, $k=1,\ldots,d$, 
counts how often the number $k$ appears among the indices
$i_1,\dots,i_r$.  
In particular, for $d=r=2$, we have $T(e_1,e_1)=t_{11}$, $T(e_2,e_2)=t_{22}$ and $T(e_1,e_2)=T(e_2,e_1)=\frac{1}{2}t_{12}$, hence
$T=T(e_1,e_1)e_1^2+2T(e_1,e_2)e_1e_2+T(e_2,e_2)e_2^2$ 
with
$2T(e_1,e_2)e_1e_2=T(e_1,e_2)e_1e_2+T(e_2,e_1)e_2e_1$. 

For $T\in \mathbb{T}^r(\R^d)$ we introduce a norm by
$$
|T|:=\sup\{|T(x_1,\ldots,x_r)|:x_i\in\R^d,\|x_i\|\le 1,i=1,\ldots,r\}.
$$
If $a_i\in
\mathbb{T}^{r_i}$, $i=1,\dots,k$, then
\begin{equation}\label{eq:neua1}
|a_1\odot \dots\odot a_k|\le |a_1|\cdots |a_k|.
\end{equation}
In particular, if $x\in\R^d$ and $u\in \mathbb{S}^{d-1}$, then $|x^ru^s|\le \|x\|^r$. Moreover, if $x_1,\ldots,x_m\in\R^d$ and $y_1,\ldots,y_m\in\R^d$, then 
\begin{equation}\label{eq:neua2}
\left|\bigodot_{i=1}^m x_i-\bigodot_{i=1}^m y_i\right|\le 
\sum_{j=1}^m\|x_j-y_j\|\prod_{i=1}^{j-1}\|x_i\|\prod_{i=j+1}^{m}\|y_i\|.
\end{equation}

For $T\in\mathbb{T}^r(\R^d)$ and a rotation $\vartheta\in O(d)$ of $\R^d$, we define the operation of $\vartheta$ on $T$ in such a way that the resulting tensor $\vartheta T\in \mathbb{T}^r(\R^d)$ is given by  $(\vartheta T) (x_1,\ldots,x_r)=T(\vartheta^{-1}x_1,\ldots,\vartheta^{-1} x_r)$ for $x_1,\ldots,x_r\in\R^d$.

In order to introduce the basic Minkowski tensors of (nonempty) convex bodies, we use the support measures which can be considered as local versions of the intrinsic volumes. For a convex body $K\subset\R^d$, the support measure $\Lambda_j(K,\cdot)$, for $j\in\{0,\ldots,d-1\}$, is a Borel measure on $\R^d\times \mathbb{S}^{d-1}$ that arises as a coefficient in a local Steiner formula (compare \eqref{eq1} for a much more general relation) and satisfies $\Lambda_j(K,\R^d\times\mathbb{S}^{d-1})=V_j(K)$. 

For instance, if $P\subset\R^d$ is a polytope, $\mathcal{F}_j(P)$ is the finite set of $j$-dimensional faces of $P$, and $N(P,F)$ is the $(d-j)$-dimensional normal cone of $P$ at $F\in \mathcal{F}_j(P)$, then
\begin{equation}\label{eq:interpretlambda1}
\Lambda_j(P,\cdot)=\frac{1}{\omega_{d-j}}\int_F\int_{N(P,F)\cap\mathbb{S}^{d-1}}\1\{(x,u)\in\cdot\}\, \mathcal{H}^{d-1-j}(\md u)\, \mathcal{H}^{j}(\md x),
\end{equation}
where $\mathcal{H}^k$ denotes the $k$-dimensional Hausdorff measure,  for $k\in\{0,1,\ldots,d\}$, on the appropriate domain, respectively 
 (see Fig.~\ref{fig:PolytopeNormalcones} for an illustration of faces and normal cones of a polytope). For a face $F$ of a polytope $P$, the normal cone $N(P,F)$ is defined as the set of all $u \in \R^d$ such that $\langle z,u\rangle\le \langle x,u\rangle$ for all $z\in P$, where $x$ is an arbitrary point in the relative interior of $F$ (see \cite{S14}, Section 2.2). An alternative description is
provided in \cite{HugWeil2020} (Solution for Exercise 3.3.9). The pairs $(x,u)\in F\times N(P,F)$ are called support elements of $P$.  Since  support measures are concentrated on the support elements of the underlying set (also in the case of general sets), the terminology is justified.

\begin{figure}[t]
  \captionsetup{ labelfont = {bf}, format = plain }
  \centering
  \includegraphics[height=0.48\linewidth]{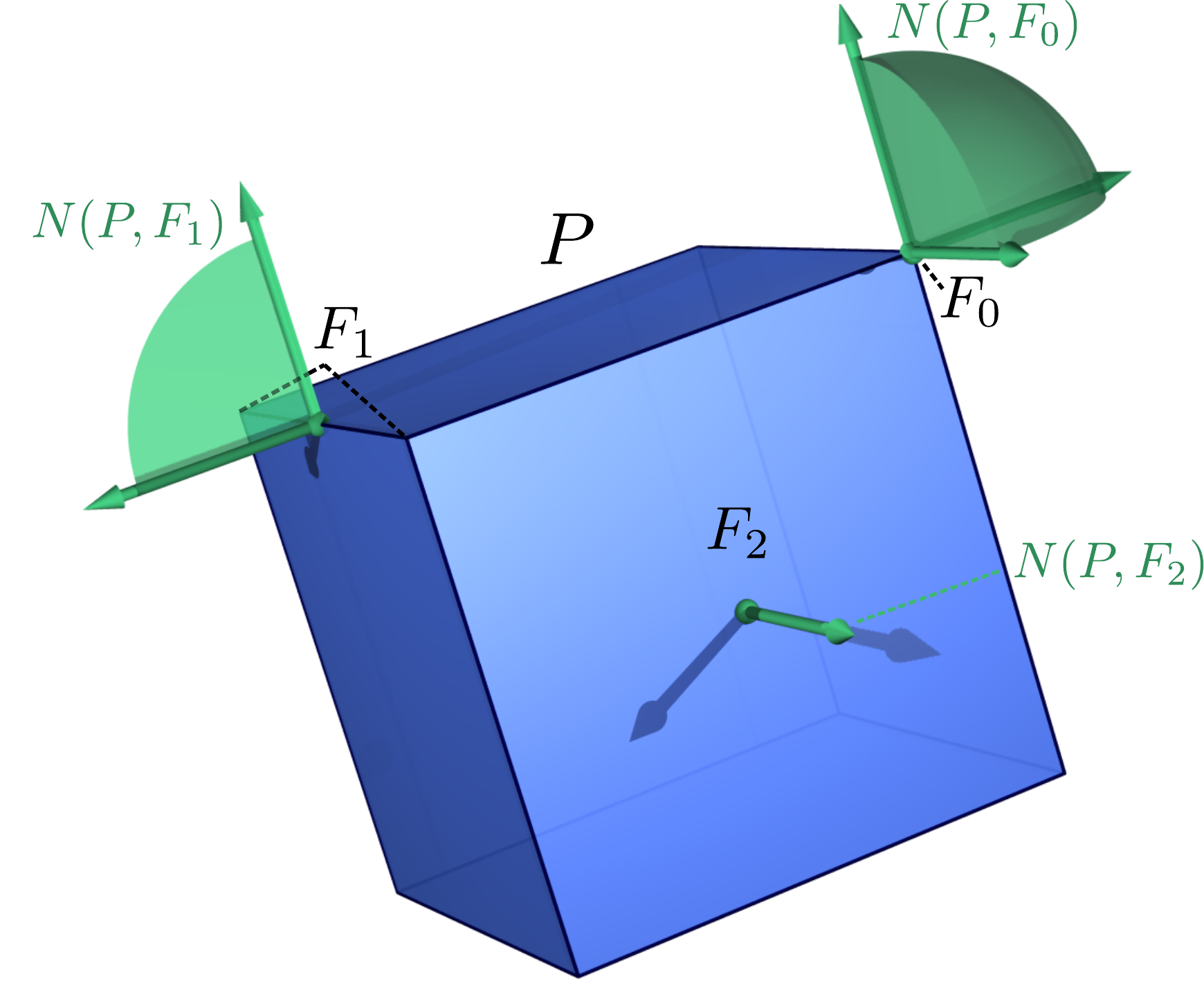}
  \caption{Cube $P$ with normal cones $N(P,F_i)$ at faces $F_i\in\mathcal{F}_i(P)$ for $i=0,1,2$.}
  \label{fig:PolytopeNormalcones}
\end{figure}

Another basic description of the support measures is available for convex bodies (or more general domains)  $K\subset \R^d$ with nonempty interior whose boundary $\partial K$  is of class $C^2$. If $\nu(K,x)$ denotes the unique exterior unit  normal vector of $K$ at $x\in\partial K$, then 
\begin{equation}\label{eq:interpretlambda2}
\Lambda_j(K,\cdot)=\frac{1}{\omega_{d-j}}\int_{\partial K}\1\{(x,\nu(K,x))\in\cdot\} \sum_{|I|=d-1-j}\prod_{i\in I}k_i(K,x)\, \mathcal{H}^{d-1}(\d x),
\end{equation}
where $k_1(K,x),\ldots,k_{d-1}(K,x)\ge 0$ are the principal curvatures of $K$ at $x\in\partial K$ and the summation extends over all subsets $I\subseteq\{1,\ldots,d-1\}$ of cardinality $|I|=d-1-j$ (if $j=d-1$ the empty product is interpreted as $1$).  
 
 For basic properties of the support measures, their connection to a local Steiner formula and a representation of $\Lambda_j(K,\cdot)$ for general convex bodies, which produces \eqref{eq:interpretlambda1} and \eqref{eq:interpretlambda2} as special cases, we refer to \cite[(2.78) and (4.28)]{S14} and Section \ref{sec:asymptotics}. The maps $K\mapsto \Lambda_j(K,\cdot)$ are measure-valued, additive (valuations), weakly continuous and for Borel sets $\alpha\subseteq\R^d$ and $\beta\subseteq \mathbb{S}^{d-1}$ they satisfy the covariance condition $\Lambda_j(gK,g\alpha\times g_0\beta)=\Lambda_j(K,\alpha\times\beta)$, where $g:\R^d\to\R^d$ is a rigid motion and $g_0$ is its rotational part (that is, $g(x)=g_0(x)+t$ for $x\in\R^d$, where $t\in\R^d$ is a translation vector and $g_0$ is a  rotation). More generally, the support measures have been defined (as signed measures) for the more general class of sets of positive reach (see Section \ref{sec:asymptotics} for a definition and  \cite[Chap. 4]{RZ2019}  for a detailed introduction) and for finite unions of compact sets with positive reach all of whose finite intersections have again positive reach (see \cite[Chap. 5]{RZ2019} and the literature cited there); the corresponding class of sets in $\R^d$ is denoted by $\mathcal{U}^d$. In fact, extensions of these measures to arbitrary closed subsets of $\R^d$ have been studied in \cite{HLW04,HS22} and will be considered in Section~\ref{sec:asymptotics}.   

If $K$ is a convex body or a compact set with positive reach and  $r,s\in\N_0$, then the basic \emph{Minkowski tensors} of $K$ are defined by
  \begin{equation}\label{02-1.5}
    \Phi_k^{r,s}(K) := \frac{1}{r!s!}\frac{\omega_{d-k}}{\omega_{d-k+s}} \int_{\R^d\times\mathbb{S}^{d-1}} x^ru^s\,\Lambda_k(K,\md(x,u)),
  \end{equation}
  for $k\in\{0,\dots,d-1\}$, 
  \begin{equation}\label{02-1.6}
    \Phi_d^{r,0}(K):=\frac{1}{r!}\int_K x^r\, \md x,
  \end{equation}
  and by
  \begin{equation*}
    \Phi_k^{r,s}(K):=0\qquad
    \text{if $k\notin \{0,\dots,d\}$ or $r\notin\N_0$  or
      $s\notin\N_0$  or
      $ k=d$, $s\not=0$}.
  \end{equation*}
  Note that the tensor product  $x^ru^s$ in the integrand  is the symmetric tensor product of the symmetric tensor powers  $x^r$ and $u^s$. 
  
Support measures and thus Minkowski tensors can be defined not only for convex bodies or compact sets with positive reach. By additivity their domain can be extended to the class of polyconvex sets (finite unions of convex bodies). Moreover they can be defined also for sets from the class $\mathcal{U}^d$, since such an extension is possible for the support measures. Furthermore, Minkowski tensors can be introduced for even more general classes of sets provided that the support measures are defined in a natural way and the integrals exist (see Section~\ref{sec:asymptotics} and the references given there).  An explicit description of the Minkowski tensors of polytopes is based on \eqref{eq:interpretlambda1} and used in Section~\ref{sec:ConvexTests},  in Section~\ref{sec:NonconvexTests} we evaluate formulas for nonconvex sets with smooth boundaries (compare \eqref{eq:interpretlambda2}). In the case of sets with positive reach or for even more general classes of compact sets $K$ (such as arbitrary finite unions of sets with positive reach), the support measures $\Lambda_j(K,\cdot)$ are consistently replaced by the reach measures $\mu_j(K;\cdot)$ in definition \eqref{02-1.5}; the relationship between these measures will be discussed in Section~\ref{sec:asymptotics}.

It is a straightforward consequence of the properties of the support measures of convex bodies that the Minkowski tensors of convex bodies are continuous, isometry (rigid motion)  covariant (for a detailed definition, especially of the underlying notion of polynomial behavior with respect to translations, we refer to \cite[Section~2.2]{HSVJK}) and additive tensor functionals (valuations) on the space of convex bodies in $\R^d$. By a fundamental result due to Alesker~\cite{Alesker_annalsofmathematics,Alesker_geomdedicata}, the vector space $T_p$ of all mappings $\Gamma:\mathcal{K}^d\to\mathbb{T}^p$ having these properties is spanned by the tensor valuations  $Q^m\Phi_k^{r,s}$, 
  where $m,r,s\in\N_0$ satisfy $2m+r+s=p$, where $k\in\{0,\dots,d\}$,
  and where $s=0$ if $k=d$. In contrast to the real-valued and vector-valued case, these tensor valuations are no longer linearly independent. A study of linear dependencies was initiated by McMullen and completed by Hug, Schneider, Schuster (for precise references and detailed statements, see, e.g., \cite[Theorems 2.6 and 2.7]{HSVJK}).   
  These general investigations imply that if $T_{2,k}$ denotes the subspace of all maps in $T_2$ that are homogeneous of degree $k$, then  $T_2=T_{2,0}\oplus\cdots \oplus T_{2,d+2}$ (direct sum decomposition), $\dim T_2=3d+1$, and a basis of $T_{2,k}$ is displayed in Table~\ref{tabbasesgencase}. Denoting by $T_2^\ast$ the subspace of translation invariant maps in $T_2$ and by $T_{2,k}^\ast$ the subspace of translation invariant maps in $T_{2,k}$, we have $T_2^\ast=T_{2,0}^\ast\oplus\cdots\oplus T_{2,d}^\ast$, $\dim T_2^\ast=2d$, and  a basis of $T_{2,k}^\ast$ is displayed in Table~\ref{tabbasesgencase}. Moreover, as a special consequence of McMullen's linear dependencies, we obtain 
  \begin{equation}\label{phi002s}
    \Phi_0^{0,2s} = \chi\cdot  (4\pi)^{-s} (s!)^{-1} Q^s,
  \end{equation}
for $s\in\N_0$, as indicated in Figure \ref{fig:intro_schematic}(c).

\begin{table}
  \centering
\captionsetup{
  labelfont = {bf},
  format = plain,
  belowskip = 1ex,
  width = \textwidth
}
\caption{Bases for the subspaces $T_{2,k}$, $k\in\{0,\ldots,d+2\}$, and $T_{2,k}^\ast$, $k\in\{0,\ldots,d\}$}
  
\begin{tabular}{c||c|c|c|c|c}
  \toprule
& $k=0$ & $k=1$ & $k\in\{2,\ldots,d-1\}$ & $k=d$ & $k\in\{d+1,d+2\}$\\[1ex]
  \midrule
$T_{2,k}$ & $\{Q\Phi_{0}^{0,0}\}$ & $\{\Phi_{1}^{0,2}, Q\Phi_{1}^{0,0}\}$ & $\{\Phi_{k}^{0,2},Q\Phi_{k}^{0,0},\Phi_{k-2}^{2,0}\}$ & $\{\Phi_{d-2}^{2,0},Q\Phi_{d}^{0,0}\}$ & $\{Q\Phi_{k-2}^{2,0}\}$\\[1ex]
$T_{2,k}^\ast$ & $\{Q\Phi_{0}^{0,0}\}$ & $\{\Phi_{1}^{0,2},Q\Phi_{1}^{0,0}\}$ & $\{\Phi_{k}^{0,2},Q\Phi_{k}^{0,0}\}$ & $\{Q\Phi_{d}^{0,0}\}$ & -- \\[1ex]
  \bottomrule
  \end{tabular}
  \label{tabbasesgencase}
  \end{table}

\section{Asymptotics for estimators: a  theoretical foundation}
\label{sec:asymptotics}

In the current section, we first recall and illustrate a local Steiner formula for general compact sets together with the required concepts from geometric measure theory. For special classes of sets, we explain how the reach measures involved are related to the better known support measures. The main results of this section are Theorems \ref{Thm1} and \ref{thm:unions}. They relate differentials of local volume integrals to surface integrals for general classes of compact sets, including finite unions of sets with positive reach. These results provide the foundation for the Voronoi-FD algorithm introduced in Section \ref{sec:Voronoi-FD}.

Let $A\subset\R^d$ be a nonempty closed set.  The topological boundary of $A$ is denoted by $\partial A$.  For $z\in\R^d$, the distance of $z$ from $A$ is denoted by 
$\text{dist}(A,z):= \min\{\|a-z\|: a\in A\} $. The $r$-neighborhood of $A$, for $r\ge 0$, is the set
$$
A^r:=\{y\in \R^d:\text{dist}(A,y)\le r\}.
$$
The (positive) normal bundle of $A$ is the Borel subset of $\R^d\times \mathbb{S}^{d-1}$ that is defined by
$$
\Nor(A):=\{(x,u)\in\partial A\times \mathbb{S}^{d-1}:{\rm dist}(A,x+ru)=r\text{ for some }r>0\},
$$
and the (Borel measurable) reach function of $A$ is given by
$$
r_A(x,u):=\sup\{s>0:{\rm dist}(A,x+su)=s\},\quad (x,u)\in \Nor(A)
$$
(see \cite{HS22} and note the slight correction in comparison with \cite{HLW04}). For the preceding assertions regarding measurability, see the references in \cite[Section 2.4]{HS22}.

\begin{figure}[t]
  \captionsetup{ labelfont = {bf}, format = plain }
  \centering
  \includegraphics[height=0.48\linewidth]{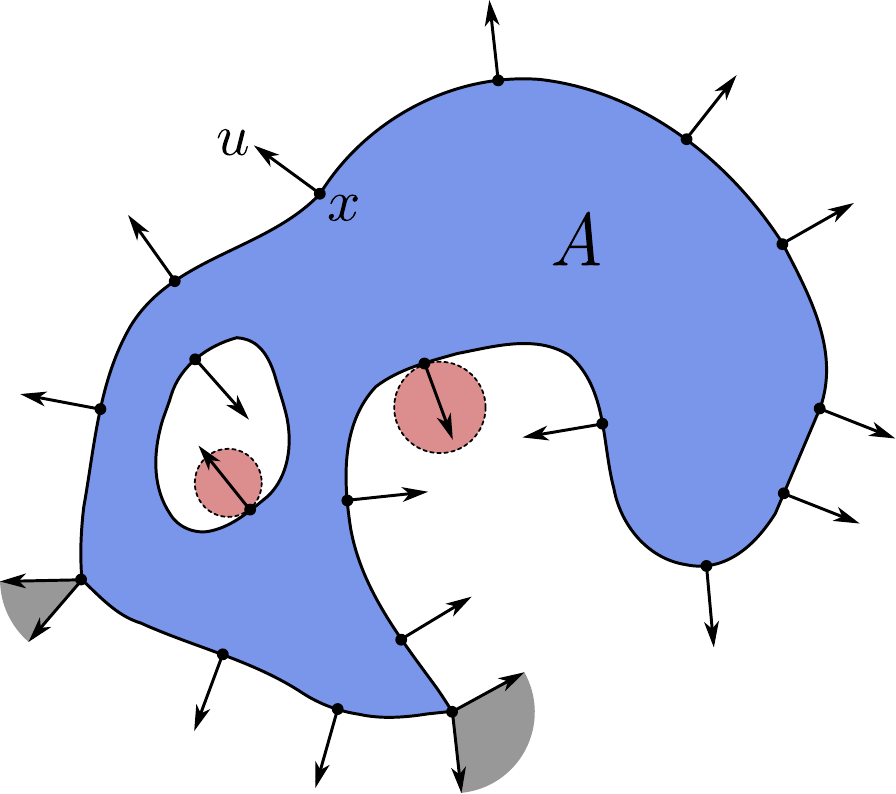}
  \caption{Set $A$ with positive reach and various support elements  $(x,u)\in \Nor(A)$.}
  \label{fig:posreachandsupportelements}
\end{figure}

The reach of $A$, denoted by $\reach(A)$, is defined as the supremum of all $r\ge 0$ such that each $z\in A^r$ has a unique nearest point in $A$. We say that $A$ has positive reach, if $\reach(A)>0$. It can be shown that $\reach(A)\ge r$, for some $r>0$, if $r_A(x,u)\ge r$ for $\mathcal{H}^{d-1}$-almost all $(x,u)\in \Nor(A)$; see \cite[Lemma 3.19]{HS22} for the proof of a more general fact. An illustration of elements of the  normal bundle of a set with positive reach is given in Fig.~\ref{fig:posreachandsupportelements}, Fig.~\ref{fig:reachfunction1} illustrates the reach function for a polygonal arc which is a polyconvex set (but not a set with positive reach). Recall that $\mathcal{U}^d$ denotes the class of all finite unions of compact sets with positive reach all of whose finite intersections have again positive reach. The polygonal arc shown in Fig.~\ref{fig:reachfunction1} lies in $\mathcal{U}^d$.

It was shown in \cite[Theorem 2.1]{HLW04}, \cite[Theorem 3.16]{HS22} that for a general closed set $A\subset\R^d$ there exist uniquely determined signed ``reach measures'' $\mu_i(A;\cdot)$  on the Borel subsets of $\R^d\times\mathbb{S}^{d-1}$, for $i\in \{0,\ldots,d-1\}$,  concentrated on $\Nor(A)$ and such that 
\begin{align}\label{eq1}
&\int_{\R^d\setminus A}f(z)\, \mathcal{H}^d(\md z)\notag
\\
&\qquad 
=\sum_{i=0}^{d-1}\omega_{d-i}\int_0^\infty\int_{\Nor(A)} t^{d-1-i}\1\{r_A(x,u)>t\}f(x+tu)\, \mu_i(A;\md(x,u))\, \md t,
\end{align}
where $f:\R^d\to\R$ is an arbitrary measurable bounded function with compact support. Clearly, \eqref{eq1} remains true if $f$ takes values in the (finite-dimensional) vector space $\mathbb{T
}^r(\R^d)$, for any fixed $r\in\N_0$, where  $f:\R^d\to \mathbb{T
}^r(\R^d)$ is said to be bounded if the map $x\mapsto |f(x)|$ is bounded ($f$ is bounded on a subset $B$ if the restriction of this map to $B$ is bounded). It is shown in \cite[(2.27)]{HLW04} that on the right side of \eqref{eq1} the order of integration can be interchanged. Hence, for $(x,u)\in\Nor(A)$ the indicator $\1\{r_A(x,u)>t\}$ restricts the domain of integration for $t\ge 0$ so that the triples $(x,u,t)$ uniquely parametrize almost all points of $\R^d$ via $(x,u,t)\mapsto x+tu\in \R^d$.  

\begin{remark}\label{rem:sec3null}
If $A\neq\emptyset$ is a compact convex set, then $r_A(x,u)=\infty$ for $(x,u)\in \Nor(A)$ and \eqref{eq1} boils down to the local Steiner formula for convex bodies. If $\text{reach}(A)>r$ and the function $f$ is supported in $A^r$, then the indicator function in \eqref{eq1} can be omitted, and again we get a local Steiner formula in the $r$-neighborhood $A^r$ of $A$.    
\end{remark}

\begin{figure}[t]
  \captionsetup{ labelfont = {bf}, format = plain }
  \centering
  \includegraphics[height=0.3\linewidth]{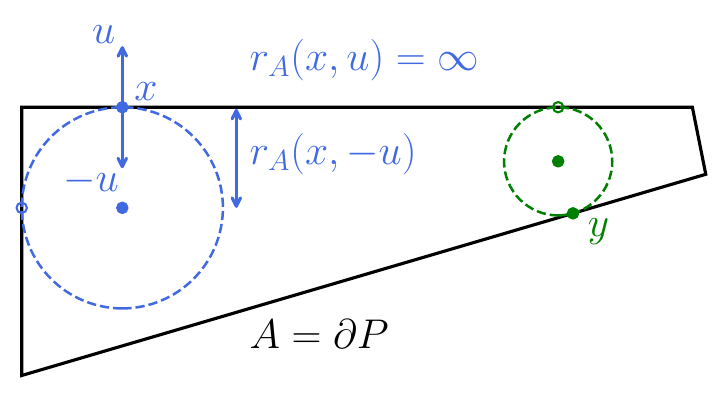}
  \caption{Boundary $A=\partial P$ of a polytope $P$ and particular values of the reach function with touching balls from outside.}
  \label{fig:reachfunction1}
\end{figure}

\begin{remark}\label{rem:sec3} {\rm 
If $A\in\mathcal{U}^d$, then the support measures $\Lambda_j(A,\cdot)$, $j\in\{0,\ldots,d-1\}$, of $A$ are defined by additive extension. For this generalization, it is crucial that $A=A_1\cup\ldots\cup A_m$ can be obtained as a finite union of compact sets $A_l$ with positive reach, $l\in\{1,\ldots,m\}$, such that arbitrary finite intersections of $A_1,\ldots,A_m$ have positive reach. In general, the intersection of two sets of positive reach does not have positive reach. A basic example is provided in \cite[Example 4.16]{RZ2019}, an even more striking example is constructed in \cite[Section 5, Example 1]{ACV2008}, another example is displayed in Fig.~\ref{fig:reachthree3}. The set $A$ is the epigraph of the function $t\mapsto f\left(\frac{\pi}{16}-|t|\right)$, $t\in [-\frac{\pi}{16},\frac{\pi}{16}]$, where $f(t):= t^6\cdot\sin^2\left(\frac{1}{t}\right)$, for $t\neq 0$, and $f(0):=0$. 

\begin{remark}\label{rem:sec3c} {\rm 
For a generic (in the sense of Baire categories \cite[Section 2.6]{S14}) convex body $K\subset\R^d$ (even for $d=2$) the boundary $A=\partial K$ of $K$ is strictly convex, of class $C^1$, for $\mathcal{H}^{d-1}$-a.e.~$x\in A$, $N(A,x):=\{u\in\mathbb{S}^{d-1}:(x,u)\in \Nor(A)\}$ consists of two antipodal unit vectors with $r_A(x,u)=\infty$ and $r_A(x,-u)\in (0,\infty)$ (say) and the set of points $x\in A$ where $r_A(x,-u)=0$ may be dense in $A$ (see \cite{McMullen1974} and Corollary 2.7.2, Theorem 2.7.4 and Note 2 for Section 2.7 in \cite{S14}). Clearly, $A$ is not a finite union of sets with positive reach.}
\end{remark}

\begin{figure}[t]
  \captionsetup{ labelfont = {bf}, format = plain }
  \centering
  \includegraphics[height=0.48\linewidth]{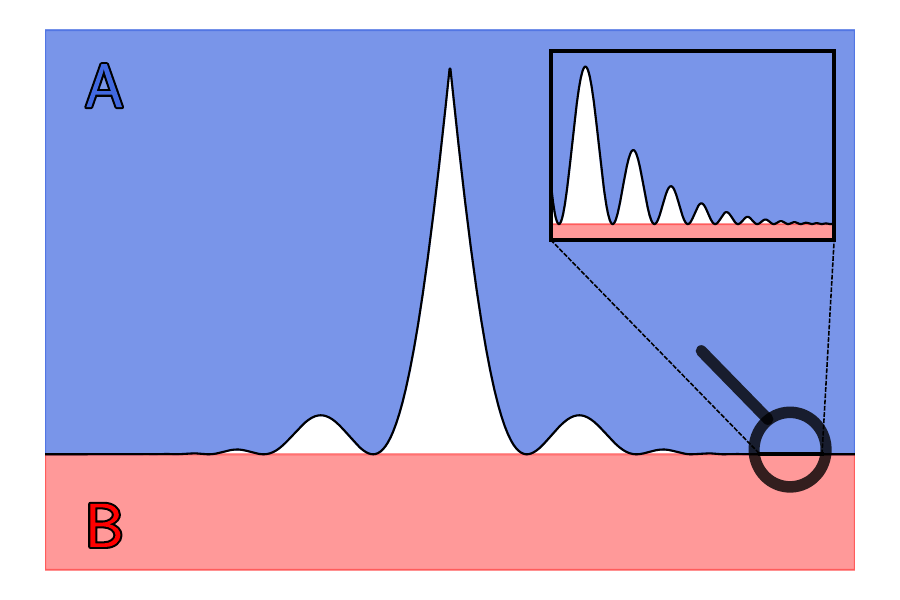}
  \caption{The sets $A, B$ have positive reach, but neither $A \cap B$ nor $A\cup B$ has positive reach.}
  \label{fig:reachthree3}
\end{figure}

 It was shown in \cite[Section 3]{HLW04} that if $A\in \mathcal{U}^d$, then 
$$
\mu_j(A;\cdot)=\Lambda_j(A,\Nor(A)\cap\cdot),\quad \text{for } j=0,\ldots,d-1.
$$
For $A\in \mathcal{U}^d$, the additive extension $\Lambda_j(A, \cdot)$ of the support measures is concentrated on $N^*(A)$, another notion of normal bundle that contains $\Nor(A)$ as a subset and is defined by means of an index function.  
The arguments in \cite[Sections 3 and 4]{HLW04} show that indeed 
\begin{equation}\label{eq:3consistent}
\mu_{d-1}(A;\cdot)=\Lambda_{d-1}(A,\cdot)
\end{equation}
as Borel measures on $\R^d\times\mathbb{S}^{d-1}$; 
see also \cite[Corollary 2]{Rat05}. In particular, $A\mapsto \mu_{d-1}(A;\cdot)$ is additive on $\mathcal{U}^d$. On the other hand, the support measures are in general not defined on the larger domain of finite unions of sets with positive reach. For this reason, we work with the reach measures whenever the support measures are not available. The relation \eqref{eq:3consistent} shows that this approach is consistent with the previous literature if we restrict ourselves to the top-order measures with index $j=d-1$, as in the definition \eqref{eq:4consistent} below.}
\end{remark}

The integrals on the right-hand side of \eqref{eq1} are well defined, but in general the reach measures $\mu_i(A;\cdot)$ are not Radon measures. In the following, we consider compact sets $A\subset\R^d$ and assume that the (nonnegative) total variation measure $\|\mu_i(A)\|(\cdot)$ of $\mu_i(A;\cdot)$ is finite for $i\in\{0,\ldots,d-1\}$. 
Sufficient conditions are provided in Lemma~\ref{Lem:3.1}. 

\begin{lemma}\label{Lem:3.1}
    If $A=\bigcup_{l=1}^mA_l$ with compact sets $A_l\subset\R^d$ such that $\mathcal{H}^{d-1}(\Nor(A_l))<\infty$, then the total variation measures of the reach measures of $A$ are concentrated on $\Nor(A)$ and satisfy 
    \begin{equation}\label{eq2}
    \|\mu_i(A)\|(\R^d\times\mathbb{S}^{d-1})<\infty\quad \text{ for }i\in\{0,\ldots,d-1\}.
    \end{equation}
    In particular, \eqref{eq2} holds if $A$ is a finite union of compact sets of positive reach.
\end{lemma}

\begin{proof}
Recall that $\mu_i(A;\cdot)$ is concentrated on $\Nor(A)$ which implies the first assertion of the lemma. 

    It follows from \cite[Corollary 2.5]{HLW04} that 
    $$
    \|\mu_i(A)\|(\Nor(A))\le c(d,i)\int_{\Nor(A)}|H_{d-1-i}(A,x,u)|\, \mathcal{H}^{d-1}(\md(x,u)),
    $$
    where $c(d,i)$ is a finite constant and 
    $|H_{d-1-i}(A,x,u)|\le \binom{d-1}{i}$, 
    as can be seen from \cite[(2.13)]{HLW04}. Let $A=\bigcup_{l=1}^mA_l$ with compact sets $A_l\subset\R^d$ such that $\mathcal{H}^{d-1}(\Nor(A_l))<\infty$. Since 
    $$
    \Nor(A)\subset \bigcup_{l=1}^m \Nor(A_l),
    $$
    the second assertion follows. 

    If $A_l$ is a compact set with positive reach, then 
    $\Nor(A_l)$ is bilipschitz homeomorphic to the compact $(d-1)$-dimensional submanifold ${\rm dist}(A_l,\cdot)^{-1}(\{r_0\})$ of $\R^d$, if $r_0<\reach(A_l)$; see also Remark 4.4, Lemma 4.21 and Corollary 4.22 in \cite{RZ2019}. Hence $\mathcal{H}^{d-1}(\Nor(A_l))<\infty$, which implies the remaining assertion.
\end{proof}

\begin{remark}\label{rem:erg1}
    If $A\neq\emptyset$ is compact and $\varepsilon>0$, then $A^\varepsilon$ satisfies condition \eqref{eq2}; see \cite[Corollary 4.4]{HLW04}. 
\end{remark}

Let $A\subset\R^d$ be a nonempty closed set. For $\mathcal{H}^d$-almost all $z\in\R^d\setminus A$, the metric projection $p_A(z)\in A$ of $z$ to $A$ and the vector $u_A(z)\in\mathbb{S}^{d-1}$ are 
uniquely determined (see \cite[(2.1)]{HLW04}  and the references given there) by 
$$
\|p_A(z)-z\|={\rm dist}(A,z), \quad u_A(z)=\frac{z-p_A(z)}{{\rm dist}(A,z)}.
$$
Note that $p_A(z)=z$ for $z\in A$ (whereas $u_A(z)$ is defined only for $z\notin A$), and if $(x,u)\in\Nor(A)$ and $0<s<r_A(x,u)$, then $p_A(x+su)=x$ and $u_A(x+su)=u$. 
For a nonempty compact set $A\subset\R^d$, we consider
$$
\partial^+A:=\{x\in\partial A: (x,u)\in \Nor(A)\text{ for some }u\in\mathbb{S}^{d-1}\},
$$
and for $x\in\partial^+A$ we define the set of  unit normal vectors of $A$ at $x$ by
$$
N(A,x):=\{u\in\mathbb{S}^{d-1}:(x,u)\in \Nor(A)\}.
$$
Note that the Borel set $N(A,x)$ is either a singleton, consists of two unit vectors $\{\pm u\}$  or is an infinite spherically convex set. We set
$$
\partial^{i}A:=\{x\in\partial^+A: |N(A,x)|=i\},\quad i\in\{1,2\},
$$
where $|N(A,x)|$ denotes the cardinality of the set $N(A,x)$. 
We write $\nu_A(x)$ for the unique unit vector in $N(A,x)$, if $x\in \partial^{1} A$, and $\pm \nu_A(x)$ for the two unit vectors in $N(A,x)$, if $x\in \partial^{2} A$.

\begin{remark}\label{rem:neureferee}
The boundary $A=\partial P=\partial A$ of a polytope $P\subset\R^2$ is not a set with positive reach, but a finite union of (convex) segments. Note that $\partial^+A=\partial A$, $\partial^1 A$ is the set of vertices and $\partial^2 A$ is $\partial A$ minus the set of vertices. While here $A$ is still in the convex ring, this is no longer the case for the boundary of a general (planar, generic) convex body $K$, where still $\partial^2 K$ has full measure in $\partial K$ and $\partial^1 K$ can be a dense subset of $\partial K$ (see Remark \ref{rem:sec3c}). 

An example of a set with positive reach that is the closure of its interior (but not a topological manifold) is provided in \cite[Section 5, Example 1]{ACV2008} (see the set between $A$ and $B$ in Fig.~\ref{fig:reachthree3} for a simple illustration and compare with Remark \ref{rem:sec3c}). Clearly, these examples show the relevance of the sets $\partial^1 A$ and $\partial ^2 A$. 
\end{remark}

The subsets $\partial^1 A$ and $\partial ^2 A$ of $\partial A$ are crucial for the description of the limit in the following theorem. If $s$ is odd, the integral over $\partial^2 A$ does not contribute, since $1+(-1)^s=0$. If $s$ is even, the integral is multiplied by $2$, which corresponds to the two unit normal vectors $\pm\nu_A(x)$ at boundary points $x\in\partial^2 A$.  

\begin{theorem}\label{Thm1}
    Let $r,s\in\mathbb{N}_0$. Let $A\subset\R^d$ be a nonempty compact set such that \eqref{eq2} holds. If $f:\R^d\to\mathbb{T}^r(\R^d)$ is measurable and bounded on $\partial A$, then 
    \begin{align*}
    &\lim_{\varepsilon\to 0_+}\frac{1}{\varepsilon}\int_{A^\varepsilon\setminus A}f(p_A(z))u_A(z)^s\, \mathcal{H}^d(\md z)\\
    &=
    \int_{\partial^{1}A}f(x)\nu_A(x)^s\, \mathcal{H}^{d-1}(\md x)+\left(1+(-1)^s\right)\int_{\partial^{2}A}f(x)\nu_A(x)^s\, \mathcal{H}^{d-1}(\md x).
    \end{align*}  
In particular, the assertion holds for finite unions of compact sets with positive reach.
\end{theorem}

\begin{proof}
We apply \eqref{eq1} with the map $\tilde{f}:\R^d\to\mathbb{T}^{r+s}$, $\tilde{f}(z)=f(p_A(z))u_A(z)^s\1\{z\in A^\varepsilon\setminus A\}$ (with the understanding that $\tilde{f}(z)=0$ if $z\in A$), which yields
\begin{align*}
&\int_{A^{\varepsilon}\setminus A} f(p_A(z))
u_A(z)^s \, \mathcal{H}^d(\md z)\\
&=\sum_{i=0}^{d-1}\omega_{d-i}\int_0^\varepsilon \int_{\Nor(A)}t^{d-1-i}f(x)u^s \1\{r_A(x,u)>t\}\, \mu_i(A;\md(x,u))\, \md t.
\end{align*}
We consider
$$
h_i(t):=\omega_{d-i}t^{d-1-i}\int_{\Nor(A)}f(x)u^s\1\{r_A(x,u)>t\}\, \mu_i(A;\md (x,u)).
$$
If $i\in\{0,\ldots,d-2\}$, then
\begin{align}
   \varepsilon^{-1} \left|\int_0^\varepsilon h_i(t)\, \md t\right|
    &\le \frac{\omega_{d-i}}{\varepsilon}\int_0^\varepsilon t^{d-1-i}\int_{\Nor(A)}|f(x)|\1\{r_A(x,u)>t\}\, \|\mu_i(A)\|(\d(x,u))\, \d t\nonumber\\
    &\le \frac{\omega_{d-i}}{d-i}\cdot |f|_{\partial A}\cdot \|\mu_i(A)\|(\Nor(A))\cdot\varepsilon^{d-i-1}
    \to 0\quad \text{ as } \varepsilon \to 0_+,\label{eq:higherorder}
\end{align}
where $|f|_{\partial A}:=\sup\{|f(x)|:x\in\partial A\}<\infty$ and basic properties of the norm $|\cdot|$ on (symmetric) tensors were used, in particular \eqref{eq:neua1}. 
Moreover, by the dominated convergence theorem, 
\begin{align*}
    h_{d-1}(t)&=2\int_{\Nor(A)}f(x)u^s\1\{r_A(x,u)>t\}\, \mu_{d-1}(A;\md(x,u))\\
    &\to  2\int_{\Nor(A)}f(x)u^s \, \mu_{d-1}(A;\md(x,u))=:h_{d-1}(0_+)\quad \text{ as }t \to 0_+,
\end{align*}
since $r_A(x,u)>0$ for $(x,u)\in\Nor(A)$. 
Note that by \cite[Proposition 4.1]{HLW04}, $\mu_{d-1}(A;\cdot)$ is a nonnegative $\sigma$-finite Borel measure and 
\begin{align}\label{eq:reachd-1measure}
2\mu_{d-1}(A;\cdot)&=\int_{\partial^1 A}\1\{(x,\nu_A(x))\in\cdot\}\,\mathcal{H}^{d-1}(\md x)\\
&\quad + \int_{\partial^{2}A}\1\{(x,\nu_A(x))\in\cdot\}+
\1\{(x,-\nu_A(x))\in\cdot\}
 \, \mathcal{H}^{d-1}(\md x).\nonumber
\end{align}
Hence, if $t\to 0_+$, then 
\begin{align*}
    h_{d-1}(t) 
    &\to  \int_{\partial^{1}A}f(x)\nu_A(x)^s\, \mathcal{H}^{d-1}(\md x)+\left(1+(-1)^s\right)\int_{\partial^{2}A}f(x)\nu_A(x)^s\, \mathcal{H}^{d-1}(\md x).
\end{align*}
Since 
$$
\varepsilon^{-1}\int_0^\varepsilon h_{d-1}(t)\, \md t\to h_{d-1}(0_+)\quad \text{ as }\varepsilon \to 0_+,
$$
the assertion follows.
\end{proof}

\begin{remark}\label{rem:slowlimit}
Relation \eqref{eq:higherorder} in the proof of Theorem \ref{Thm1} shows that  
$$
  \varepsilon^{-1} \left|\int_0^\varepsilon h_i(t)\, \md t\right|\le C(d,A,f)\cdot \varepsilon, \quad i\in\{0,\ldots,d-2\},
$$
where $C(d,A,f)$ is a constant that depends only on the parameters in brackets. 
If $\reach(A)>0$, then $h_{d-1}(t)=h_{d-1}(0_+)$ for $0<t<\reach(A)$, hence in this case the deviation 
\begin{equation}\label{eq:deviation}\left|\frac{1}{\varepsilon}\int_{A^\varepsilon\setminus A}f(p_A(z))u_A(z)^s\, \mathcal{H}^d(\md z)-2\int_{\Nor(A)}f(x)u^s \, \mu_{d-1}(A;\md(x,u))\right|
\end{equation}
is at most of the order $\varepsilon$ as $\varepsilon\to 0_+$. 

If $A=B^2(o,1)\cup B^2(2\Delta e,1)$, for some fixed $\Delta\in (0,1)$ and $e\in\mathbb{S}^1$, then \eqref{eq:deviation} is still of the   order $\varepsilon$. However, if $\Delta=1$, then \eqref{eq:deviation} is of the order $\varepsilon^{1/2}$. Considering the epigraph of the function $x\mapsto \frac{1}{\beta}|x|^\beta$, $\beta>1$ fixed, and its reflection across the $e_1$-axis, one can construct a set $A_\beta$ for which the deviation \eqref{eq:deviation} is of the order $\varepsilon^{1/\beta}$. 
\end{remark}

The limit in Theorem \ref{Thm1} involves only integrations over (parts of) the boundary of the compact set $A$, which corresponds to the contribution of the top order support measure $\mu_{d-1}(A;\cdot)$, which can be thought of as the surface area measure of $A$. However, in the proof we have to control the contributions of all support measures $\mu_i(A;\cdot)$, $i\in \{0,\ldots,d-2\}$. 
By essentially the same argument, we  obtain the following variant of Theorem~\ref{Thm1}. The case $s=0$ of the next theorem is already covered by Theorem~\ref{Thm1} with the domain of integration restricted to the complement of the set $A$.

\begin{theorem}\label{thm:unions}
     Let $s\in\mathbb{N}$. Let $A\subset\R^d$ be a nonemtpy compact set such that \eqref{eq2} holds. If $f:\R^d\to\mathbb{T}^r(\R^d)$ is measurable and bounded on $\partial A$, then 
    \begin{align}\label{eq:Theorem3.3}
    &\lim_{\varepsilon\to 0_+}\frac{1+s}{\varepsilon^{1+s}}\int_{A^\varepsilon }f(p_A(z))(z-p_A(z))^s\, \mathcal{H}^d(\md z)\\
    &=
    \int_{\partial^{1}A}f(x)\nu_A(x)^s\, \mathcal{H}^{d-1}(\md x)+\left(1+(-1)^s\right)\int_{\partial^{2}A}f(x)\nu_A(x)^s\, \mathcal{H}^{d-1}(\md x).\nonumber
    \end{align}  
In particular, the assertion holds for finite unions of compact sets with positive reach.
\end{theorem}

\begin{proof} First, note that $z-p_A(z)=o$ if $z\in A$ and $s\neq 0$. Hence the domain of integration on the left-hand side can be restricted to $A^\varepsilon\setminus A$ without changing the value of the integral. We apply \eqref{eq1} with the map $\hat{f}:\R^d\to\mathbb{T}^{r+s}$, $\hat{f}(z)=f(p_A(z))(z-p_A(z))^s\1\{z\in A^\varepsilon\setminus A\}$ and use that $(x+tu)-p_A(x+tu)=tu$ for $(x,u)\in\Nor(A)$ and $0<t<r_A(x,u)$. Thus we obtain
\begin{align*}
&\int_{A^{\varepsilon}} f(p_A(z))
(z-p_A(z))^s \, \mathcal{H}^d(\md z)\\
&=\sum_{i=0}^{d-1}\omega_{d-i}\int_0^\varepsilon \int_{\Nor(A)}t^{d-1-i+s}f(x)u^s \1\{r_A(x,u)>t\}\, \mu_i(A;\md(x,u))\, \md t.
\end{align*}
Note that if $g:[0,\infty)\to\R$ is measurable and bounded, and $g(t)\to g(0_+)$ as $t\to 0_+$, then 
$$
\varepsilon^{-(1+s)}\int_0^\varepsilon t^s g(t)\, \md t\to \frac{g(0_+)}{1+s} \quad\text{ as }\varepsilon\to 0_+.
$$
Otherwise, the proof follows the lines of the proof of Theorem~\ref{Thm1}.
\end{proof}

\section{Algorithms to estimate Minkowski tensors}
\label{sec:Algorithms}

Based on our theoretical results from the previous section, we now define
two explicit algorithms to estimate the Minkowski tensors of a set $K$ that is represented by a point cloud $K_0$, e.g., by a pixelated image.
For both algorithms, it is necessary to estimate the so-called
Voronoi tensors, which is the subject of Subsection
\ref{sec:VoronoiTensorEstimation}. In the next two Subsections \ref{sec:Voronoi-FD} and \ref{sec:Voronoi-LSQ}, we
introduce and describe the two algorithms in detail.
Finally, in Subsections~\ref{sec:implementation} and \ref{sec:choice},
we discuss important aspects of our open-source implementation
\cite{VorominkCode} and provide some advice on the choice
of parameters.

\subsection{Voronoi tensor estimation}\label{sec:VoronoiTensorEstimation}

\begin{figure}[t]
  \captionsetup{ labelfont = {bf}, format = plain }
  \centering
  \subfloat[][]{\includegraphics[height=0.48\linewidth,angle=270]{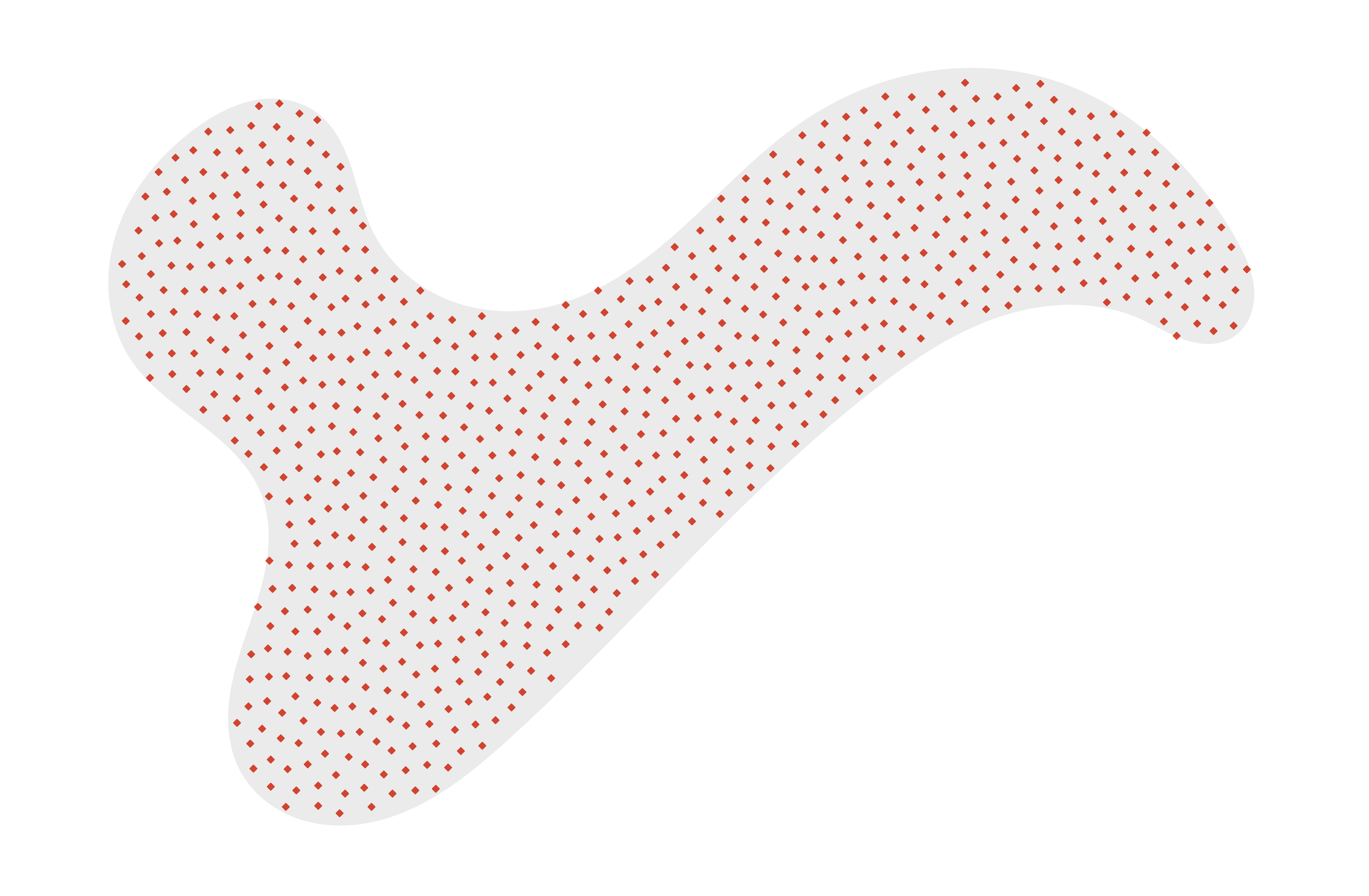}}
  \quad
  \subfloat[][]{\includegraphics[height=0.48\linewidth,angle=270]{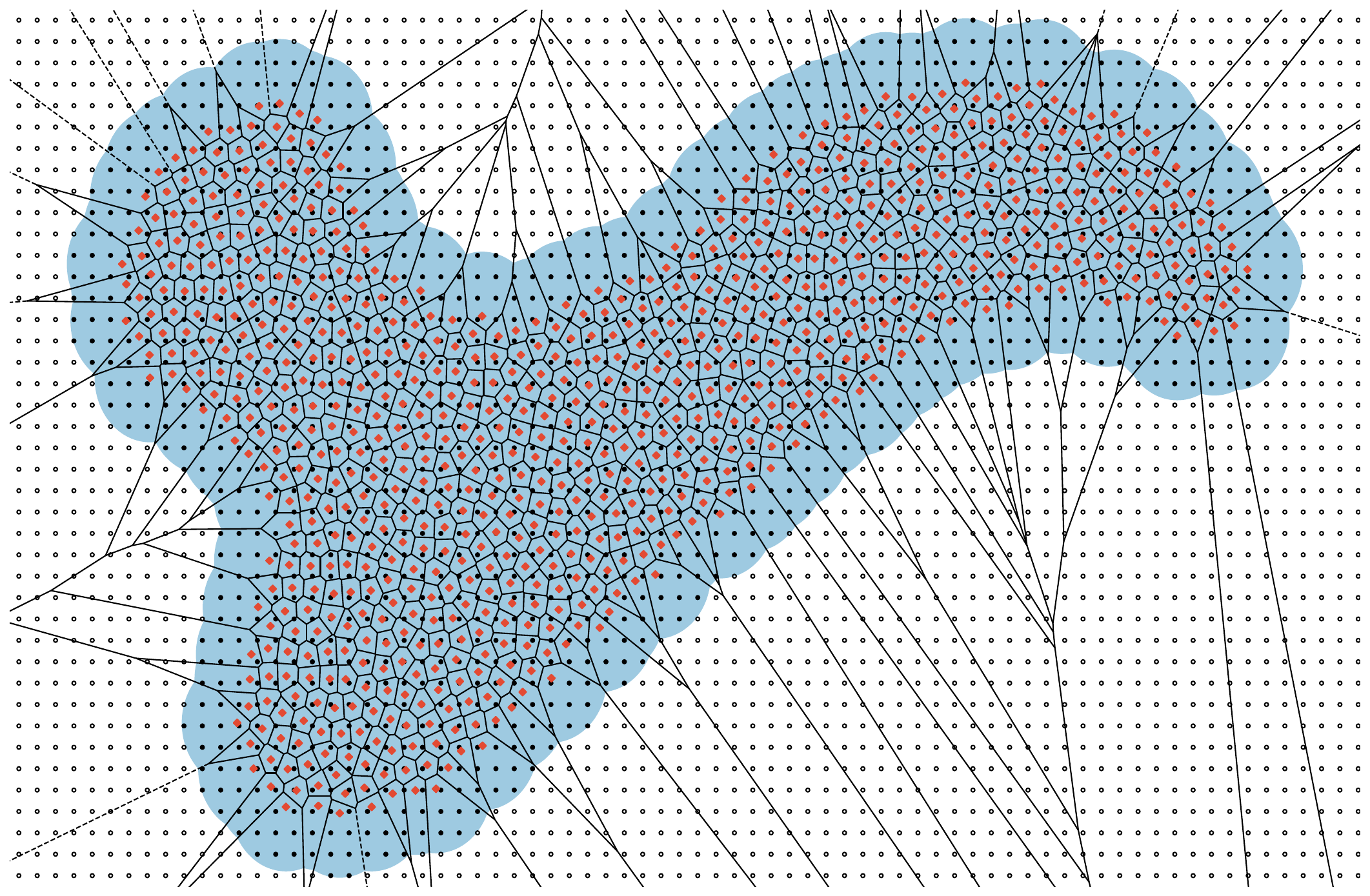}}
  \caption{Example data $K_0$ (red points), which is a finite subset of an underlying set with positive reach (a), and the Voronoi diagram (b) of the data $K_0$ and the set $K_0^R$ (blue) of points with distance not bigger than some $R>0$ from $K_0$.
  The algorithm uses a random grid $\eta$ (compare (\ref{VoroM_Estimator})) to estimate the Voronoi tensor $\mathcal{V}_R^{r,s}(K_0)$ of $K_0$, where only the points of $\eta$ inside $K_0^R$ are relevant for the estimation.}
  \label{fig:Algorithm}
\end{figure}

The Voronoi tensor $\mathcal{V}_R^{r,s}(K)$ with distance parameter $R\ge 0$ and rank parameter tuple $(r,s)\in\N_0^2$ of a nonempty compact set $K\subset\R^d$ is defined by
\begin{align} \label{VoroM}
\mathcal{V}_R^{r,s}(K) := \int_{K^R} p_K(x)^r (x-p_K(x))^s \,\mathcal{H}^d(\d x),
\end{align}
where $K^R$ is the set of points with distance at most $R>0$ from $K$
and $p_K$ is the (almost everywhere uniquely defined) metric projection
on $K$. The Voronoi tensor $\mathcal{V}_R^{r,s}(K)$ is the total Voronoi
tensor measure $\mathcal{V}_R^{r,s}(K;\R^d)$ introduced in
\cite{hug_voronoi-based_2017}. The present work is inspired by the
results in \cite{hug_voronoi-based_2017}, which in turn were motivated
and based on the contributions \cite{MR2594445,5669298}. Extensions to
$\delta$-Voronoi tensors, where $\delta$ is a distance-like function as
considered in \cite{CLMT15}, which have been introduced to improve the
robustness of estimation methods against outliers for the Voronoi
covariance measure, will not be pursued here. If $s>0$, then the
integral in \eqref{VoroM} of the symmetric tensor product of the
symmetric tensors $p_K(x)^r$ and $(x-p_K(x))^s$ can be restricted to
$K^R\setminus K$ without changing the value of the integral. For $r=s=0$
the Voronoi tensor with distance parameter $R$ is the volume of $K^R$.  
For a finite set $K_0\subset\R^d$ and $r,s\in\N_0$, the Voronoi tensors simplify to
\begin{align} \label{VoroM_finite}
\mathcal{V}_R^{r,s}(K_0) = \sum_{x\in K_0} x^r \int_{B(x,R)\cap V_x(K_0)} (y-x)^s \, \mathcal{H}^d(\d y),
\end{align}
where $B(x,R)$ is the closed ball with center $x$ and radius $R$ (we omit the upper index that indicates the dimension) and 
\begin{align*}
V_x(K_0) := \{ y\in\R^d : p_{K_0}(y)=x \}
\end{align*}
is the  Voronoi cell of $x\in K_0$ with respect to the set $K_0$. 
It is clear that each compact set $K\subset \R^d$ can be approximated with any specified precision by finite sets $K_0\subset\R^d$ with respect to the Hausdorff metric $d_H$ on compact subsets of $\R^d$ (see, e.g.,  \eqref{eq:apprdH}); we refer to \cite[p.~571]{schneider_stochastic_2008} for a definition of the Hausdorff metric.  If $K,K_0\subseteq B(o,\rho)$ for some $\rho>0$, then 
\begin{equation}\label{eq:appr1}
\left|\mathcal{V}_R^{r,s}(K)-\mathcal{V}_R^{r,s}(K_0)\right|\le C(d,R,\rho,r,s)\cdot d_H(K,K_0)^{\frac{1}{2}},
\end{equation}
which follows from \cite[Theorem 4.1]{hug_voronoi-based_2017}, where $C(d,R,\rho,r,s)$ is a constant depending only on the arguments in brackets. An inspection of the proof of Theorem 4.1 in \cite{hug_voronoi-based_2017} shows that the constant $C(d,R,\rho,r,s)$ can be chosen such that  $C(d,R,\rho,r,s)\le  C(d,R_0,\rho,r,s)$  if $0<R\le R_0$. Note that \eqref{eq:appr1} does not require that $K_0$ is a subset of $K$ (although an approximation of $K$ from inside can be guaranteed for a given $K$) and holds whenever $K,K_0$ are compact subsets of $B(o,\rho)$. Recall from \cite{hug_voronoi-based_2017} that $d_H(K,K_0)^{\frac{1}{2}}$ can be replaced by $d_H(K,K_0)$ if $r=s=0$.

Evaluating the integral in (\ref{VoroM_finite}) can be cumbersome and inefficient. Moreover, we do not 
require a precision beyond the approximation of $K$ via $K_0$. We, therefore, estimate the Voronoi tensors in (\ref{VoroM_finite}) by considering a random grid
\begin{align}\label{random_grid}
\eta=\sum_{z\in a\cdot\Z^d} \delta_{z+U_a}
\end{align}
with some scaling parameter $a>0$ and a uniform random vector $U_a\sim\mathcal{U}([0,a]^d)$. The tensor-valued random variable
\begin{align} \label{VoroM_Estimator}
\widehat{\mathcal{V}}_R^{r,s}(K_0,a) = a^d \sum_{x\in\eta} \1\big\{\|x-p_{K_0}(x)\|<R\big\} p_{K_0}(x)^r (x-p_{K_0}(x))^s
\end{align}
will then be used as an  estimator for ${\mathcal{V}}_R^{r,s}(K_0)$.
Since $K_0$ is a finite set, only finitely many summands do not vanish.
The estimated Voronoi tensor in (\ref{VoroM_Estimator}) can be viewed as a discrete version of the integral in (\ref{VoroM_finite}). 
Figure~\ref{fig:Algorithm} (a) provides an example data set $K_0$, which is given as a finite subset of an underlying set with positive reach.
In Figure~\ref{fig:Algorithm} (b) the idea of the estimator (\ref{VoroM_Estimator}) is demonstrated for this example data set $K_0$.
For each data point $x\in K_0$ one can see the Voronoi cell $V_x(K_0)$.

\begin{lemma}\label{Le4.1:unbiased}
Let $K_0\subset\R^d$ be a finite set and $a>0$. 
Then the estimator $\widehat{\mathcal{V}}_R^{r,s}(K_0,a)$ for ${\mathcal{V}}_R^{r,s}(K_0)$ is unbiased, that is, $\BE \widehat{\mathcal{V}}_R^{r,s}(K_0,a)={\mathcal{V}}_R^{r,s}(K_0)$.
\end{lemma}

\begin{proof} 
Let $\{z_i:i\in\N\}=a\cdot \Z^d$ be an enumeration of the points in the lattice $a\cdot \Z^d$. Then
\begin{align*}
   \widehat{\mathcal{V}}_R^{r,s}(K_0,a) &=a^d\sum_{i\ge 1}
    \1\{\|z_i+U_a - p_{K_0}(z_i+U_a)\|<R\}p_{K_0}(z_i+U_a)^r(z_i+U_a-p_{K_0}(z_i+U_a))^s\\
    &=a^d\sum_{i\ge 1}\sum_{w\in K_0}
     \1\{p_{K_0}(z_i+U_a)=w\}\1\{\|z_i+U_a-w\|<R\}w^r(z_i+U_a-w)^s.
\end{align*}
Note that $p_{K_0}(z_i+U_a)$ is uniquely defined $\BP$-almost surely, since $U_a$ has a density with respect to the Lebesgue measure. 

If the indicator functions are nonzero, then for every $i\in\N$ there is some $w\in K_0$ such that $\|z_i+U_a-w\|\le R$, and hence $z_i\in K_0+B(o,\sqrt{d} a+R)$. In particular, this shows that the effective ranges of the summations over $i\ge 1$ and $w\in K_0$ are finite and can be chosen independently of $U_a$. Moreover, \eqref{eq:neua1} yields
$$
|w^r(z_i+U_a-w)^s |\le 
\|w\|^r\cdot \|z_i+U_a-w\|^s\le C(K_0)^rR^s<\infty,
$$
where $C(K_0)$ is a constant that depends only on $K_0$, but not on a specific $w\in K_0$. 

These remarks allow us to interchange summation and expectation so that
\begin{align*}
    \BE \widehat{\mathcal{V}}_R^{r,s}(K_0,a) 
    &=a^d\sum_{i\ge 1}\sum_{w\in K_0}
    \BE \1\{p_{K_0}(z_i+U_a)=w\}\1\{\|z_i+U_a-w\|<R\}w^r(z_i+U_a-w)^s\\
    &=\sum_{w\in K_0}\sum_{i\ge 1}\int_{z_i+[0,a]^d}\1\{p_{K_0}(y)=w\}
    \1\{\|y-w\|<R\}w^r(y-w)^s\,  \mathcal{H}^d(\md y)\\
    &=\sum_{w\in K_0} \int_{\R^d}\1\{y\in V_w(K_0)\}
    \1\{\|y-w\|<R\}w^r(y-w)^s\, \mathcal{H}^d(\md y)\\
    &=\sum_{w\in K_0}w^r\int_{\R^d}\1\{y\in V_w(K_0)\cap B(w,R)\}
     (y-w)^s\,\mathcal{H}^d(\md y)
     =\mathcal{V}_R^{r,s}(K_0) ,
\end{align*}
which proves the assertion.
\end{proof}

Next we show that  $\widehat{\mathcal{V}}_ {R}^{r,s}(K_0,a)$ is a consistent estimator  for $\mathcal{V}_ {R}^{r,s}(K)$,  as $K_0\to K$ in the Hausdorff metric  and  $a=a(K_0)\to 0$. More precisely, we obtain the following quantitative estimate. We write 
$$\Nsf(K_0)=\min\{\|x_1-x_2\|:x_1,x_2\in K_0, x_1\neq x_2\}$$ 
for the minimal distance of any two different points in a finite set $K_0\subset\R^d$ of cardinality at least two (which we tacitly assume  all the time) and $s\wedge t$ for the minimum of $s,t\in\R$.

\begin{theorem}\label{Le4.2:convergence}
Let $K\subset\R^d$ be a compact set, let $r,s\in \N_0$ and $R>0$. Let $\rho>0$ be such that $K\subseteq B(o,\rho)$. If $K_0\subseteq B(o,\rho)$ is a finite set and $a>0 $, then almost surely 
$$
|\widehat{\mathcal{V}}_R^{r,s}(K_0,a)-\mathcal{V}_ {R}^{r,s}(K)|
    \le C'(
d,R,\rho,r,s)  \left(d_H(K,K_0)^{\frac{1}{2}}+\max\left\{\frac{a}{\Nsf(K_0)},\left(\frac{a}{ \Nsf(K_0)}\right)^d\right\}\right), 
$$
where $C'(d,R,\rho,r,s)$ is a positive constant, and $d_H(K,K_0)^{\frac{1}{2}}$ can be replaced by $d_H(K,K_0)$ if $r=s=0$.

In particular, if  $(K_0(i))_{i\in\N}$ is a sequence of finite sets in $\R^d$ and $(a_i)_{i\in\N}$ is a sequence of positive numbers such that $K_0(i)\to K$ in the Hausdorff metric and  $\Nsf(K_0(i))^{-1}a_i\to 0$, as $i\to\infty$, then  $\widehat{\mathcal{V}}_ {R}^{r,s}(K_0(i),a_i)\to  \mathcal{V}_ {R}^{r,s}(K)$, as $i\to\infty$, almost surely.
\end{theorem}

\begin{proof}
The constants employed in the proof are denoted by $c,c_1,c_2$ and depend only the parameters indicated in brackets. 

By the triangle inequality and in view of \eqref{eq:appr1}  it is sufficient to show that for a fixed finite set $K_0$, almost surely we have 
$$|\widehat{\mathcal{V}}_ {R}^{r,s}(K_0,a)-  \mathcal{V}_ {R}^{r,s}(K_0)|\le c(d,R,\rho,r,s)\cdot  \max\left\{\frac{a}{ \Nsf(K_0)},\left(\frac{a}{\Nsf(K_0)}\right)^d\right\}.$$ For this, observe that almost surely $p_{K_0}(z+U_a)$ is uniquely defined for each $z\in a\cdot\Z^d$. Moreover, since $K_0$ is finite and there are only finitely many $z\in a\cdot\Z^d$ having distance at most $\sqrt{d}aR$ from a point in $K_0$, all sums encountered in the following are finite. Hence, almost surely
\begin{align*}
   \widehat{\mathcal{V}}_R^{r,s}(K_0,a)
   &=a^d\sum_{z\in a\cdot \Z^d}
    \1\{\|z+U_a-p_{K_0}(z+U_a)\|<R\}p_{K_0}(z+U_a)^r(z+U_a-p_{K_0}(z+U_a))^s\\
    &=\sum_{x\in K_0}a^d\sum_{z\in a\cdot \Z^d}\1\{\|z+U_a-x\|<R\}\1\{z+U_a\in {V}_x(K_0)\}x^r(z+U_a-x)^s\\
    &=\sum_{x\in K_0}x^r\sum_{z\in a\cdot \Z^d} \1_{B(x,R)\cap  {V}_x(K_0)}(z+U_a)(z+U_a-x)^sa^d\\
    &=\sum_{x\in K_0}x^r\int \sum_{z\in a\cdot \Z^d} \1_{B(x,R)\cap  {V}_x(K_0)}(z+U_a)(z+U_a-x)^s\1_{z+[0,a)^d}(y)\,\mathcal{H}^d(\md y).
    \end{align*}
    Note that in the last step integration and summation can be interchanged, since for each $x\in K_0$ there are only finitely many summands $z\in a\cdot\Z^d$ such that the first indicator is nonzero (as pointed out above). 
    Moreover, we have
\begin{align*}
   \mathcal{V}_R^{r,s}(K_0) 
   &= \sum_{x\in K_0} x^r \int\1_{B(x,R)\cap  {V}_x(K_0)}(y)(y-x)^s\, \mathcal{H}^d(\md y)\\
   &= \sum_{x\in K_0} x^r \int\sum_{z\in a\cdot \Z^d}\1_{B(x,R)\cap  {V}_x(K_0)}(y)(y-x)^s\1_{z+[0,a)^d}(y)\, \mathcal{H}^d(\md y).
\end{align*}
    Hence, using \eqref{eq:neua1}, we get
\begin{align*}
&|\widehat{\mathcal{V}}_R^{r,s}(K_0,a)-\mathcal{V}_R^{r,s}(K_0)|\\
&\le \rho^r\sum_{x\in K_0} \int \sum_{z\in a\cdot \Z^d} \left| 
\1_{B(x,R)\cap  {V}_x(K_0)}(z+U_a)(z+U_a-x)^s- 
\1_{B(x,R)\cap  {V}_x(K_0)}(y)(y-x)^s
\right|\\
&\qquad\qquad\qquad\qquad \times 
\1_{z+[0,a)^d}(y)\, \mathcal{H}^d(\md y)  . 
\end{align*}
Suppose that $y\in z+[0,a)^d$. 
If $z+[0,a)^d\subseteq B(x,R)\cap  {V}_x(K_0)$, then \eqref{eq:neua2} shows that the expression $|\cdot|$ under the integral can be bounded from above by 
$$\left|(z+U_a-x)^s-(y-x)^s\right|\le s\cdot \|z+U_a-y\|\cdot R^{s-1}\le s\cdot \sqrt{d}R^{s-1}\cdot a.$$ 
If $(z+[0,a)^d)\cap  B(x,R)\cap  {V}_x(K_0)=\emptyset$, then the expression is zero.  In the remaining case, we have $(z+[0,a)^d)\cap \partial (B(x,R)\cap  {V}_x(K_0))\neq\emptyset$ and the expression $|\cdot|$ is bounded from above by $2R^s$ (see \eqref{eq:neua1}). Thus we obtain
\begin{align}\label{eq:bound1}
&|\widehat{\mathcal{V}}_R^{r,s}(K_0,a)-\mathcal{V}_R^{r,s}(K_0)|\nonumber\\
&\le \rho^r \sum_{x\in K_0}\left(\sum_{z\in a\cdot \Z^d}\1\{z+[0,a)^d\subseteq B(x,R)\cap  {V}_x(K_0)\}\cdot s\cdot \sqrt{d}R^{s-1}\cdot a\cdot a^d\right.\nonumber\\
&\qquad\qquad\qquad\qquad  \left.+\sum_{z\in a\cdot \Z^d}\1\{(z+[0,a)^d)\cap  \partial (B(x,R)\cap  {V}_x(K_0))\neq\emptyset\}\cdot 2R^s\cdot a^d\right)\nonumber\allowdisplaybreaks\\
&\le d\rho^rR^{s-1} \sum_{x\in K_0}\Big(s\cdot a \, \mathcal{H}^d( B(x,R)\cap  {V}_x(K_0)) \nonumber\\
&\qquad\qquad\qquad\qquad 
+R\,\mathcal{H}^d\left(\partial (B(x,R)\cap  {V}_x(K_0))+B(o,\sqrt{d}\cdot a)\right)
\Big)\nonumber\allowdisplaybreaks\\
&\le s\cdot d\rho^rR^{s-1}\cdot a\cdot \mathcal{H}^d\left(K_0+B(o,R)\right)\nonumber\\
&\qquad\qquad\qquad\qquad +
d\rho^rR^{s} \sum_{x\in K_0}\mathcal{H}^d\left(\partial (B(x,R)\cap  {V}_x(K_0))+B(o,\sqrt{d}\cdot a)\right),
\end{align}
where we used that the Voronoi cells have pairwise disjoint interiors. 
Since $K_0\subset B(o,\rho)$, we get 
\begin{equation}\label{eq:bound2}
\mathcal{H}^d( K_0+B(o,R))
\le \kappa_d(\rho+R)^d. 
\end{equation}
To deal with the remaining sum, we denote by $r_x$ the inradius of the Voronoi cell $V_x(K_0)$. Using a volume bound from the proof of  Lemma 3.6 in \cite{MR3449314} and the obvious inequality $r_x\ge \Nsf(K_0)/2$ for $x\in K_0$, we obtain
\begin{align*}
    &\mathcal{H}^d\left(\partial (B(x,R)\cap  {V}_x(K_0))+B(o,\sqrt{d}\cdot a)\right)\nonumber\\
    &\le 2\left[\mathcal{H}^d\left( B(x,R)\cap  {V}_x(K_0)+B(o,\sqrt{d}\cdot a)\right)-\mathcal{H}^d\left( B(x,R)\cap  {V}_x(K_0) \right)\right]\nonumber\\
    &\le 2\left[\mathcal{H}^d\left( B(x,R)\cap  {V}_x(K_0)+ \frac{\sqrt{d}\cdot a}{r_x\wedge R}\cdot ( B(x,R)\cap  {V}_x(K_0))\right)-\mathcal{H}^d\left( B(x,R)\cap  {V}_x(K_0) \right)
    \right]\nonumber\\
    &=2\left[\left(1+ \frac{\sqrt{d}\cdot a}{r_x\wedge R}\right)^d-1\right]
    \mathcal{H}^d\left( B(x,R)\cap  {V}_x(K_0) \right)\nonumber\\
    &\le c_1(d)\cdot \max\left\{\frac{a}{\Nsf(K_0)\wedge R},\left(\frac{a}{\Nsf(K_0)\wedge R}\right)^d\right\}\mathcal{H}^d\left( B(x,R)\cap  {V}_x(K_0) \right),\nonumber 
\end{align*}
where $c_1(d):=(1+2\sqrt{d})^d-1$. 
Hence,
\begin{align}\label{eq:bound3}
&\mathcal{H}^d\left(\partial (B(x,R)\cap  {V}_x(K_0))+B(o,\sqrt{d}\cdot a)\right)\nonumber\\
    &\le c_2(d,\rho,R)\cdot \max\left\{\frac{a}{\Nsf(K_0)},\left(\frac{a}{\Nsf(K_0)}\right)^d\right\}\mathcal{H}^d\left( B(x,R)\cap  {V}_x(K_0) \right),
\end{align}
where we can choose $c_2(d,\rho,R):=c_1(d)$ if $R\ge\Nsf(K_0)$ and 
$$c_2(d,\rho,R):=c_1(d)\max\left\{\frac{2\rho}{R},\left(\frac{2\rho}{R}\right)^d\right\}$$ 
if $R<\Nsf(K_0)$, since $\Nsf(K_0)\le 2\rho$.  
The sum of the volumes $\mathcal{H}^d\left( B(x,R)\cap  {V}_x(K_0) \right)$ over $x\in K_0$ is bounded from above by   $\kappa_d(\rho+R)^d$ (as used before). Combination of \eqref{eq:bound1}, \eqref{eq:bound2}, and \eqref{eq:bound3} finally yields
\begin{align*}
&|\widehat{\mathcal{V}}_R^{r,s}(K_0,a)-\mathcal{V}_R^{r,s}(K_0)|\\
&\le d\rho^rR^{s-1}(\rho+R)^d \cdot \left(s\cdot a +  R\cdot  c_2(d,\rho,R) \cdot \max\left\{\frac{a}{\Nsf(K_0)},\left(\frac{a}{ \Nsf(K_0)}\right)^d\right\}\right),
\end{align*}
which implies the required bound, since $\Nsf(K_0)\le 2\rho$.
\end{proof}

From the comments after \eqref{eq:appr1} and the proof of Theorem~\ref{Le4.2:convergence},  which yields more explicit information about the constants involved, we get the following consequence. 

\begin{corollary}\label{cor4.4:bound}
Let $K\subset\R^d$ be a compact set, and let $r,s\in \N_0$. Let $\rho>0$ be such that $K\subseteq B(o,\rho)$. If $K_0\subseteq B(o,\rho)$ is a finite set, $ R\le R_0$, and $0<a\le \Nsf(K_0)\wedge R$, then almost surely 
$$
|\widehat{\mathcal{V}}_R^{r,s}(K_0,a)-\mathcal{V}_ {R}^{r,s}(K)|
    \le C''(
d,R_0,\rho,r,s)  \left(d_H(K,K_0)^{\frac{1}{2}}+ \frac{a}{\Nsf(K_0)\wedge R} \right), 
$$
where $C''(d,R_0,\rho,r,s)>0$ is a constant depending only on $d,R_0,\rho,r,s$, and $d_H(K,K_0)^{\frac{1}{2}}$ can be replaced by $d_H(K,K_0)$ if $r=s=0$. Moreover, if  $0<a\le \Nsf(K_0)\le R$, then
\begin{equation}\label{eq:response}
|\widehat{\mathcal{V}}_R^{r,s}(K_0,a)-\mathcal{V}_{R}^{r,s}(K_0)|
\le d\rho^rR^{s}(\rho+R)^d(s+c_1(d)) \,\frac{a}{\Nsf(K_0)} .
\end{equation}
\end{corollary}

\begin{remark}\label{rem:sec4a}
   Increasing the number of points in $K_0$ admits a better approximation of $K$ by $K_0$. This part of the approximation is already covered in previous work, which we summarized in \eqref{eq:appr1}. The main contribution of Theorem \ref{Le4.2:convergence} is to bound $|\widehat{\mathcal{V}}^{r,s}_R(K_0,a)-{\mathcal{V}}^{r,s}_R(K_0)|$ from above. An explicit upper bound is stated in \eqref{eq:response}. In the derivation of this bound, we encountered two cases. In the first case, we considered $z\in a\cdot \Z^d$ for which $(z+[0,a)^d)\subseteq B(x,R)\cap  {V}_x(K_0)$. For the corresponding sum, we get the upper bound $d\rho^rR^{s-1}(\rho+R)^dsa$, which improves as $a$ becomes smaller, independently of $N(K_0)$. The second main case deals with $z\in a\cdot \Z^d$ for which $(z+[0,a)^d)\cap \partial(B(x,R)\cap  {V}_x(K_0))\neq\emptyset$. For the corresponding second sum in \eqref{eq:bound1} we obtain the alternative upper bound 
    \begin{align}\label{eq:extraeq}
        &d\rho^rR^{s} \sum_{x\in K_0}\mathcal{H}^d\left(\partial (B(x,R)\cap  {V}_x(K_0))+B(o,\sqrt{d}\cdot a)\right)\nonumber\\
        &\le c_d\rho^rR^{s}|K_0|\cdot a\cdot \sum_{i=1}^dR^{d-i}a^{i-1}.
    \end{align}
    A simple volume bound shows that $|K_0|\le \left(\frac{2\rho}{N(K_0)}+1\right)^d$, which cannot be improved in general. In this way, we obtain an upper bound that involves ratios $\frac{a}{N(K_0)^i}$ also with $i>1$ instead of just the ratio $\frac{a}{N(K_0)}$. As a consequence of our approach, increasing $K_0$ requires more boundary terms to be controlled. If $K_0$ is considered to be fixed, then of course the error term decreases linearly with $a$. In a simulation study, we found that keeping $a$ fixed and increasing $|K_0|$ (or decreasing $N(K_0)$) when approximating a rectangle $K$ by finite point clouds $K_0$ does not improve the approximation of ${\mathcal{V}}^{r,s}_R(K_0)$ by $\widehat{\mathcal{V}}^{r,s}_R(K_0,a)$. 
It seems plausible that increasing $|K_0|$ (and decreasing $N(K_0)$, accordingly) should be matched by a corresponding decrease of the grid size $a$. 
\end{remark}

\begin{remark}\label{rem:sec4b}
   The task of finding a useful upper bound for the volume of all points having distance at most $\sqrt{d}\cdot a$ from  $ \partial (B(x,R)\cap  {V}_x(K_0))$ in \eqref{eq:bound3} is reminiscent of the situation treated in \cite[Theorem 3]{LCL17}. However, the assumption that the boundary of the set $X$ there should have positive reach is clearly not satisfied by the boundary of the Voronoi cell $V_x(K_0)$. 
\end{remark}

\subsection{Voronoi-FD algorithm}\label{sec:Voronoi-FD}

In the following, we describe a first method for estimating Minkowski tensors, which we call the Voronoi-FD algorithm, where ``FD'' stands for finite difference.
The method is based on Corollary~\ref{cor:approx} and works for arbitrary finite unions of compact sets with positive reach and (surface) tensor functionals $\Phi^{r,s}_{d-1}$ with $s\ge 1$.  

The following corollary of Theorem~\ref{thm:unions} shows how surface tensors can be approximated by Voronoi tensors. Recall the definition of the reach measure $\mu_{d-1}(A;\cdot)$ of a nonempty compact subset of $\R^d$ from Section~\ref{sec:asymptotics}, which equals the support measure $\Lambda_{d-1}(A,\cdot)$ if $A\in\mathcal{U}^d$ (see Remark~\ref{rem:sec3} which ensures that the Minkowski tensor $\Phi^{r,s}_{d-1}$ in \eqref{eq:4consistent} is consistently defined for all $r,s\in\N_0$). An explicit description of the reach measure $\mu_{d-1}(A;\cdot)$ is provided in \eqref{eq:reachd-1measure} (see \cite[Proposition 4.1]{HLW04}).

\begin{corollary}\label{cor:approx}
Let $K\subset\R^d$ be a nonempty compact set such that \eqref{eq2} holds. 
If $r\in\N_0$ and $s\in \N$, then 
\begin{align}\label{part1Cor4.4}
    \lim_{\varepsilon\to 0_+}\frac{1}{\varepsilon^{1+s}}
\mathcal{V}^{r,s}_\varepsilon(K)
=r!s!\kappa_{s+1}\Phi^{r,s}_{d-1}(K)
\end{align}
and 
\begin{equation}\label{eq:erg2}
\lim_{\varepsilon\to 0_+}\frac{1}{\varepsilon}\left(
\mathcal{V}^{r,0}_\varepsilon(K)-
\mathcal{V}^{r,0}_{\varepsilon^2}(K)\right)
=r! 2\,\Phi^{r,0}_{d-1}(K),
\end{equation}
where, for $r,s\in\N_0$,
\begin{equation}\label{eq:4consistent}
\Phi^{r,s}_{d-1}(K):=\frac{1}{r!s!}\frac{2}{\omega_{1+s}} \int_{\R^d\times\mathbb{S}^{d-1}} x^ru^s\,\mu_{d-1}(K;\md(x,u))
\end{equation}
and $\mu_{d-1}(K;\cdot)$ is a nonnegative Borel measure that satisfies \eqref{eq:reachd-1measure}.
\end{corollary}

\begin{proof}
First, we consider the case where $s\in\N$. We apply Theorem~\ref{thm:unions} with $f(x)=x^r$, $x\in\R^d$.  Using  the  definition \eqref{VoroM} of the total Voronoi tensor measure, we get the required expression on the left-hand side, after division by $1+s$. Accordingly, the right-hand side of equation \eqref{eq:Theorem3.3} turns into 
$$
\frac{2}{s+1}\int x^ru^s\, \mu_{d-1}(K;\d(x,u)),
$$
where also relation  \eqref{eq:reachd-1measure} was used. 
The assertion now follows from the definition of the Minkowski tensor $\Phi^{r,s}_{d-1}(K)$.   

If $s=0$, then  an application of Theorem~\ref{Thm1} and relation \eqref{eq:reachd-1measure} yield
\begin{align*}
\lim_{\varepsilon\to 0_+}\frac{1}{\varepsilon}\left(
\mathcal{V}^{r,0}_\varepsilon(K)-
\mathcal{V}^{r,0}_{\varepsilon^2}(K)\right)
&=\lim_{\varepsilon\to 0_+}\frac{1}{\varepsilon}\left(
\mathcal{V}^{r,0}_\varepsilon(K)-\mathcal{V}^{r,0}_0(K)\right)
 -\lim_{\varepsilon\to 0_+}\frac{\varepsilon}{\varepsilon^2}\left(
\mathcal{V}^{r,0}_{\varepsilon^2}(K)- 
\mathcal{V}^{r,0}_{0}(K)\right)\\
&=2r!\Phi^{r,0}_{d-1}(K),
\end{align*}
which proves the assertion.
\end{proof}

\begin{remark}\label{rem:speedcondlate}
Corollary \ref{cor:approx} applies, for instance, when $A$ is a finite union of compact sets of positive reach (by Lemma \ref{Lem:3.1}) or $A$ is a parallel set of an arbitrary compact set (by Remark \ref{rem:erg1}). The discussion in Remark \ref{rem:slowlimit} shows that no general statement is possible concerning the speed of convergence in \eqref{part1Cor4.4} or \eqref{eq:erg2}. 
\end{remark}

Combination of Corollary~\ref{cor4.4:bound} and Corollary~\ref{cor:approx} finally yields the next result. It serves as a theoretical foundation for the approximation of surface tensors of finite unions of sets with positive reach via Corollary~\ref{cor:approx}. We remark that for any nonempty compact set $K\subset\R^d$, we can approximate it to arbitrary precision with a discrete set, for example, the finite set
\begin{equation}\label{eq:apprdH}
K(n):=\left\{(\sqrt{d}n)^{-1}\cdot z:z\in\Z^d, {\rm dist}(K,(\sqrt{d}n)^{-1}\cdot z)\le \frac{1}{n}\right\}
\end{equation}
satisfies $d_H(K(n),K)\le \frac{1}{n}$. 

\begin{theorem}\label{cor4.5:convergence}
Let $K\subset\R^d$ be a nonempty compact set such that \eqref{eq2} holds. Let $r\in \N_0$, $s\in \N$, and let $\rho>0$ be such that $K\subseteq B(o,\rho)$. If $(\varepsilon_n)_{n\in\N}$ satisfies $\varepsilon_n\to 0_+$ as $n\to\infty$, then there exist sequences $(K_0(n))_{n\in\N}$ of finite subsets of $\R^d$ with $K_0(n)\to K$ in the Hausdorff distance as $n\to\infty$ and $(a_n)_{n\in\N}$ of real numbers  such that, almost surely,
$$
\lim_{n\to\infty}\frac{\widehat{\mathcal{V}}_{\varepsilon_n}^{r,s} (K_0(n),a_n)}{\varepsilon_n^{1+s}}
=r!s!\kappa_{s+1}\Phi^{r,s}_{d-1}(K)
$$
and
$$
\lim_{n\to\infty}\frac{1}{\varepsilon_n}\left(\widehat{\mathcal{V}}_{\varepsilon_n}^{r,0} (K_0(n),a_n) -\widehat{\mathcal{V}}_{\varepsilon_n^2}^{r,0} (K_0(n),a_n)\right)
=r!2\, \Phi^{r,0}_{d-1}(K).
$$
\end{theorem}

\begin{proof} 
We consider only the case $s=0$, the argument for $s\in\N$ is similar (but  easier). By Corollary~\ref{cor:approx}, for $n\in\N$ there is some $\varepsilon(n)>0$ such that
    \begin{align}\label{eq:no0b}
\left|\frac{1}{\varepsilon}\left(\mathcal{V}^{r,0}_\varepsilon(K)-\mathcal{V}^{r,0}_{\varepsilon^2}(K) \right)-r!2\, \Phi^{r,0}_{d-1}(K)\right|\le \frac{1}{2n}\quad \text{if }0<\varepsilon\le \varepsilon(n).
    \end{align}
where $\varepsilon(n+1)<\varepsilon(n)\le \varepsilon(1)\le 1$ for $n\in\N$ and $\varepsilon(n)\to 0_+$ as $n\to\infty$. 

We can assume that $\varepsilon_m\le \varepsilon(1)$ for $m\in\N$. For $m\in\N$ there is a unique $n\in\N$ such that $\varepsilon_m\in (\varepsilon(n+1),\varepsilon(n)]$. By the remarks preceding Theorem~\ref{cor4.5:convergence}, there is a finite set $K_0(m)\subset B(o,\rho)$ (if $m$ and therefore $n$ is sufficiently large) such that
\begin{equation}\label{eq:no1b}
d_H(K,K_0(m))\le \left(C''(d,\varepsilon(1),\rho,r,0)8n\right)^{-2}\cdot \varepsilon_m^{2}.
\end{equation}
Moreover, we choose
\begin{equation}\label{eq:no2b}
a_m\le \frac{\Nsf(K_0(m))\wedge\varepsilon_m^2}{C''(d,\varepsilon(1),\rho,r,0)8n}\cdot \varepsilon_m ,
\end{equation}
where we can assume that $C''(d,\varepsilon(1),\rho,r,0)4n\ge 1$, so that 
$$0<a_m\le \Nsf(K_0(m))\wedge \varepsilon_m^2\le  \Nsf(K_0(m))\wedge \varepsilon_m \le 1.  
$$
Corollary~\ref{cor4.4:bound}, \eqref{eq:no1b} and \eqref{eq:no2b} then yield almost surely
\begin{align*}
\left|\widehat{\mathcal{V}}_{\varepsilon_m}^{r,0} (K_0(m),a_m)-\mathcal{V}^{r,0}_{\varepsilon_m}(K)\right|\le \frac{\varepsilon_m}{4n}   ,\quad \left|\widehat{\mathcal{V}}_{\varepsilon_m^2}^{r,0} (K_0(m),a_m)-\mathcal{V}^{r,0}_{\varepsilon_m^2}(K)\right|\le \frac{\varepsilon_m}{4n} ,
\end{align*}
hence
\begin{equation}\label{eq:no3b}
    \left|\frac{1}{\varepsilon_m}\left(\widehat{\mathcal{V}}_{\varepsilon_m}^{r,0} (K_0(m),a_m)-\widehat{\mathcal{V}}_{\varepsilon_m^2}^{r,0}(K_0(m),a_m)\right)
    -\frac{1}{\varepsilon_m}\left( {\mathcal{V}}_{\varepsilon_m}^{r,0} (K)- {\mathcal{V}}_{\varepsilon_m^2}^{r,0}(K )\right)\right|\le \frac{1}{2n}.    
\end{equation}
Combining \eqref{eq:no0b} with $\varepsilon=\varepsilon_m$  and  \eqref{eq:no3b}, we arrive at   
 \begin{align} \label{eq:finbound}
\left|\frac{1}{\varepsilon_m}\left(\widehat{\mathcal{V}}_{\varepsilon_m}^{r,0} (K_0(m),a_m)-\widehat{\mathcal{V}}_{\varepsilon_m^2}^{r,0}(K_0(m),a_m)\right)-r!2\,\Phi^{r,0}_{d-1}(K)\right|\le \frac{1}{n},
    \end{align}
which implies the assertion.
\end{proof}

\begin{remark}\label{rem:aepsilon} 
Suppose that $\varepsilon_n=\varepsilon(n)$ for $n\in\N$, where $\varepsilon(n)$ is chosen so that \eqref{eq:no0b} is valid if $s=0$ and 
    \begin{align}\label{eq:no0bb}
\left| {\varepsilon^{-(1+s)}} \mathcal{V}^{r,s}_\varepsilon(K) -r!s!\kappa_{s+1}\, \Phi^{r,s}_{d-1}(K)\right|\le \frac{1}{2n}\quad \text{if }0<\varepsilon\le \varepsilon(n)
    \end{align}
if $s\in\N$. 
An inspection of the proof of Theorem \ref{cor4.5:convergence} shows that, for fixed $d,\rho,r,s,R$, we can choose $a_n$ of the order $n^{-2}\varepsilon_n^{3(1+s)}$ so that \eqref{eq:finbound} holds (with $m=n$) for $s=0$ and, similarly,  
\begin{align*}
\left| \varepsilon_n^{-(1+s)} \,\widehat{\mathcal{V}}_{\varepsilon_n}^{r,s} (K_0(n),a_n)-r!s!\kappa_{s+1}\,\Phi^{r,s}_{d-1}(K)\right|\le \frac{1}{n}
    \end{align*}
if $s\in\N$.    

An explicit choice of the sequence $\varepsilon(n)$, $n\in\N$, in \eqref{eq:no0b} and \eqref{eq:no0bb} will have to depend on the underlying set $K$, as the examples in Remark \ref{rem:slowlimit} demonstrate (see also Remark \ref{rem:speedcondlate}). In the special case where $K$ has positive reach, we can choose $\varepsilon_n$ of the order $n^{-1}$ and $a_n$ of the order $n^{-5-3s}$.
\end{remark}

The Voronoi-FD algorithm estimates the tensors $\Phi^{r,s}_{d-1}(K)$ from Corollary~\ref{cor:approx}, where $K$ is a nonempty compact set such that \eqref{eq2} holds,  by using \eqref{part1Cor4.4} if $s\ge 1$ and \eqref{eq:erg2} if $s=0$. In particular, the algorithm can be applied for a set $K$ that is a finite union of compact sets with positive reach. 
A very small value for $\varepsilon$ is chosen for this purpose, and the estimator from (\ref{VoroM_Estimator}) is used.
We tested the performance for a rectangle $[-\frac{3}{2},\frac{3}{2}]\times [-\frac{5}{2},\frac{5}{2}]$ (see Example~\ref{example:4.5}), a spherical shell with inner radius 1 and outer radius 2 (see Example~\ref{Ex:Shell}), and for a rectangle from which a smaller (open) rectangle has been removed (see Example~\ref{Ex:cuttedRect}).
The generated finite test data are obtained by intersecting these objects with a grid. 
The simulated estimates can be found in Table~\ref{table:Cor4.1}. 
To achieve good results, however, we needed an $a$ that is significantly smaller than $\varepsilon$, which itself must already be small (compare with \eqref{eq:response}, which guarantees a good approximation if $a\ll \Nsf(K_0)\le \varepsilon=R$).
In practice, the choice of $\varepsilon$ can be limited by the available $av(K_0)$, defined in \eqref{Def:av}, and the typical length scale of the spatial structure.
Since a small value of $a$ considerably slows down the computation of the estimator from (\ref{VoroM_Estimator}), this is a substantial disadvantage in practice.
Therefore, for estimating the tensors from Corollary~\ref{cor:approx} we propose to use the alternative approach via a least squares problem (Voronoi-LSQ algorithm) described in Subsection \ref{sec:Voronoi-LSQ}, whenever it is applicable.
For example, a single rendition (i.e., a run of the algorithm that provides an estimate) for the spherical shell $S$ in the simulations for Table~\ref{table:Cor4.1} using the Voronoi-FD algorithm took about 21 minutes on a standard personal computer, whereas in the simulations from Table~\ref{table:Shell} using the Voronoi-LSQ algorithm, all 25 renditions for the same spherical shell took only about 3 minutes on a standard personal computer.
Note that two renditions differ in the stationarized grid (\ref{random_grid})
since a different random offset of the grid is chosen for each rendition.
Surprisingly, the Voronoi-LSQ algorithm led to excellent results even in cases where we only have a heuristic justification (see Table~\ref{table:cuttedRect} for Example~\ref{Ex:cuttedRect}).
Finally, we point out that we observed that the second part of Theorem~\ref{cor4.5:convergence}, namely the approach for $s=0$, requires far too much runtime and memory capacity for practical applications. 
Since $a$ must be significantly smaller than $\varepsilon$ and $\varepsilon^2$ for the application of the Theorem, and $\varepsilon$ itself should already be a very small value, $a$ becomes too small (and thus the number of points in the lattice process too large) for practical use.
However, for $r=s=0$ this problem can be overcome, as the following remark demonstrates.

\begin{remark}\label{rem:trace}{\rm 
Let $K$ be a nonempty compact set such that \eqref{eq2} holds (which is satisfied if $K$ is a finite union of compact sets with positive reach). If $\tr(T)$ denotes the trace of a tensor $T\in\mathbb{T}^2(\R^d)$, that is, $\tr(T) := \sum_{j=1}^d T(e_j,e_j)$, then
\begin{align*}
    \tr\big( \Phi_{d-1}^{0,2}(K) \big) = \frac{1}{4\pi} \Phi_{d-1}^{0,0}(K),
\end{align*}
where we used that $\tr(u^2)=1$ for $u\in\mathbb{S}^{d-1}$ and the linearity of the trace operator.
}
\end{remark}

\begin{table}[t]
\centering
\captionsetup{
  labelfont = {bf},
  format = plain,
  belowskip = 1ex,
  width = \textwidth
}
\caption{Results for the two-dimensional rectangle $R_b$, $b=(3,5)^\top$, with side lengths 3 and 5 (see Example~\ref{example:4.5}), the spherical shell $S$ with inner radius 1 and outer radius 2 (see Example~\ref{Ex:Shell}) and $R_{p,q}$ from Example~\ref{Ex:cuttedRect} with  $p=(1,2)^\top$, $q=(3,5)^\top$ by using the Voronoi-FD algorithm.
We intersected $R_b,S,R_{p,q}$ with a grid of resolution $\Delta=0.0005$ and choose $\varepsilon=0.05$ and $a=0.0005$.
We took the average of 25 renditions.
}
\begin{tabular}{ccc|ccc}
  \toprule
  Tensors &  Exact value & Voronoi-FD & Tensors & Exact value & Voronoi-FD \\
  \midrule
  (\ten 0 2 1 ($R_b$))$_{1,1}$ & 0.3979 & 0.4029 & (\ten 0 2 1 ($R_b$))$_{1,2}$ & 0 & -$10^{-4}$ \\ [1ex]
  (\ten 0 2 1 ($R_b$))$_{2,2}$ & 0.2387 & 0.2438 & (\ten 0 2 1 ($S$))$_{1,1}$ & 0.3754 & 0.3754 \\ [1ex]
  (\ten 0 2 1 ($S$))$_{1,2}$ & 0 & $10^{-4}$ & (\ten 0 2 1 ($S$))$_{2,2}$ & 0.3755 & 0.3757 \\ [1ex]
  (\ten 0 2 1 ($R_{p,q}$))$_{1,1}$ & 0.5560 & 0.5564 & (\ten 0 2 1 ($R_{p,q}$))$_{1,2}$ & 0 & -$10^{-4}$ \\ [1ex]
  (\ten 0 2 1 ($R_{p,q}$))$_{2,2}$ & 0.3173 & 0.3178 & \\ [1ex]
\bottomrule
\end{tabular}
\label{table:Cor4.1}
\end{table}

\subsection{Voronoi-LSQ algorithm}\label{sec:Voronoi-LSQ}

In the following, we explain another algorithm to estimate Minkowski tensors, which is applicable for  arbitrary tensor functionals $\Phi^{r,s}_{j}$.
It is based on solving a least squares problem, which is why we call it the Voronoi-LSQ algorithm. Importantly, it offers some practical advantages.
We focus on the current improvements in comparison to the original algorithm suggested in  \cite{hug_voronoi-based_2017}.
The digitized estimator of the volume tensors and especially the least squares fit with respect to the radial dependence
strongly improve the robustness of the algorithm.
However, its theoretical foundation is only available for sets with positive reach.
Nevertheless, in Section~\ref{sec:NonconvexTests} (see Table \ref{table:cuttedRect}), we will see that it provides surprisingly accurate results for a set without positive reach.
The Voronoi-LSQ algorithm is based on the idea from \cite{hug_voronoi-based_2017} (see Section 3.2 there) to estimate the Minkowski tensors of a set $K$ with positive reach  by first approximating $K$ by a simpler set $K_0$ (for instance, $K_0$ could be a finite subset of $K$) and by solving a linear system $Ax=b$. 
The solution $x$ of the system is the $(d+1)$-dimensional vector containing the estimated Minkowski tensors $\hat{\Phi}_d^{r,s}(K_0),\dots,\hat{\Phi}_0^{r,s}(K_0)$, where $r,s\in\N_0$.
The vector $b$ contains the estimated Voronoi tensors $\hat{\mathcal{V}}_{R_0}^{r,s}(K_0),\dots,\hat{\mathcal{V}}_{R_d}^{r,s}(K_0)$, for $d+1$ different values $0<R_0<\dots<R_d$ of the radius, and the matrix $A$ is a Vandermonde-type matrix.
 It is shown in \cite{hug_voronoi-based_2017} that the solution of this system yields the exact values of the Minkowski tensors whenever the exact values of the Voronoi tensors are plugged in and $R_d$ is smaller than the reach of $K$. 
 Since inverting a matrix is numerically not stable, we replace the matrix inversion by a least squares problem, which allows us to use more than only $d+1$ equations, which further enhances the robustness of the algorithm.
However, the basic idea of the Voronoi-LSQ algorithm remains to estimate the Minkowski tensors of $K$ via the Voronoi tensors of a finite set $K_0$ by using the estimator from (\ref{VoroM_Estimator}).
Here $K_0$ can be viewed as a finite approximation of a possibly infinite set $K$.
In practice $K_0$ is usually considered to be a finite sample of points in $K$, while for generating test data one can intersect $K$ with a grid.

Now fix $a>0$, an integer $n\geq d+1$, and $0<R_1<\dots<R_n$.
Extending the linear system from \cite[(6)]{hug_voronoi-based_2017}, we define the estimators $\widehat{\Phi}_j^{r,s}(K_0,a,\RR)$, $i\in\{0,\dots,d\}$, as the solution of the linear least squares problem
\begin{align} \label{LSQ}
\widehat{\mathcal{V}}_ {R_i}^{r,s}(K_0,a) = \sum_{k=0}^d r!s! \kappa_{k+s} R_i^{k+s} \widehat{\Phi}_{d-k}^{r,s}(K_0,a,\RR), \qquad i\in \{1,\ldots,n\},
\end{align}
where $\RR:=(R_0,\dots,R_n)$.
Recall that $\kappa_n$ denotes the volume of an $n$-dimensional unit ball. Moreover, a symmetric tensor $T\in\mathbb{T}^{r+s}(\R^d)$ of rank $r+s$ is determined by the values $T(e_{i_1},\dots,e_{i_{r+s}})$ for $1\leq i_1\leq\dots\leq i_{r+s}\leq d$ (see (\ref{02-1.2a})).
The Voronoi-LSQ algorithm computes the tensors $\widehat{\Phi}_j^{r,s}(K_0,a,\RR)$, $i\in\{0,\dots,d\}$, by solving the least squares problem from (\ref{LSQ}) for each choice of $1\leq i_1\leq\dots\leq i_{r+s}\leq d$, using standard methods. The following lemma is well known (see \cite[Theorems 1.1.5 and 1.1.6]{B2024} or \cite[Theorem 9]{Hansen2013}).

\begin{lemma} \label{Cor:LQF_Continuity}
Let $s\in\N_0$ and $n\geq d+1$.
Let $x_1,y_1,\dots,x_n,y_n\in\R$, $c_0,\dots,c_d\in\R\setminus\{0\}$ and
\begin{align} \label{Def:Designmatrix}
   X:=  
\left(c_jx_i^{s+j}\right)_{\substack{i=1,\ldots,n\\ j=0,\ldots,d}}
\in\R^{n\times (d+1)}, \qquad y:=(y_1,\ldots,y_n)^\top
\in\R^{n}.
\end{align}
If $x_i\neq x_j$ for $i\neq j$, then $X$ has rank $d+1$ and $\mathbf{a}=(X^\top X)^{-1}X^\top y\in\R^{d+1}$ is the unique vector satisfying
\begin{align*}
    \mathbf{a} = \argmin\Big\{ \sum_{i=1}^n (b_0c_0x_i^s+\dots+b_dc_dx_i^{s+d}-y_i)^2 \mid (b_0,\dots,b_d)^\top\in\R^{d+1} \Big\}.
\end{align*}
\end{lemma}

\begin{remark}\label{rem:afterCor} {\rm 
If $s\ge 1$, then we can also consider the matrix in $\R^{n\times d}$ with $n\ge d$  obtained from $X$ by deleting the first column and setting $b_0=0$, which corresponds to the definition $\Phi_d^{r,s}=0$ if $s\neq 0$.
}
\end{remark}

\begin{remark}\label{rem:citeintro}
As a consequence of Lemma~\ref{Cor:LQF_Continuity} the least squares problem from (\ref{LSQ}) is continuous, that is whenever the input vector converges to some vector $y\in\R^n$, the solution will converge to the solution of the least squares  problem with input $y$.
In particular, if $K$ is a set with positive reach and $R_n$ is less than the reach of $K$, then the tensors $\widehat{\Phi}_j^{r,s}(K_0,a,\RR)$ from (\ref{LSQ}) converge to $ {\Phi}_j^{r,s}(K)$ as $\widehat{\mathcal{V}}_{R_i}^{r,s}(K_0,a)$ converges to $\mathcal{V}_ {R_i}^{r,s}(K)$, for which Theorem~\ref{Le4.2:convergence} provides sufficient conditions. 
\end{remark}

The approximating tensors $\widehat{\Phi}_j^{r,s}(K_0,a,\RR)$ depend on the choice of the fixed radii $0<R_1<\cdots<R_n$, but the  Minkowski tensors $ {\Phi}_j^{r,s}(K)$ will be independent of the radii (as they should be) as long as $K$ has reach greater than $R_n$. 
We would like to point out that although the solution to the least squares problem from Lemma~\ref{Cor:LQF_Continuity} involves a matrix inversion, the Voronoi-LSQ algorithm for solving the problem does not compute an inverse matrix but computes the unconstrained least-squares solution, for example, using an algorithm as in \cite{branch}.
Lemma~\ref{Cor:LQF_Continuity} merely serves to show that the solution is continuous with respect to the input.
In our situation we apply Lemma~\ref{Cor:LQF_Continuity} for each choice of $1\leq i_1\leq\dots\leq i_{r+s}\leq d$ to
\begin{align*}
    x_i = R_i, \quad c_i = r!s! \kappa_{i+s},\quad  y_i=\widehat{\mathcal{V}}_ {R_i}^{r,s}(K_0,a)(e_{i_1},\dots,e_{i_{r+s}}),
\end{align*}
to obtain $\widehat{\Phi}_{d-j}^{r,s}(K_0,a,\RR)(e_{i_1},\dots,e_{i_{r+s}})$ for $j\in\{0,\dots,d\}$.
So there are two main tasks involved in the Voronoi-LSQ algorithm: To estimate the Voronoi tensors of $K_0$ with respect to the $n$ values $R_1,\dots,R_n$ by the estimator in (\ref{VoroM_Estimator}) and then to solve the system (\ref{LSQ}). 
Therefore the only randomness involved in the algorithm is contained in the random grid $\eta$. 

We now explain the role of the relevant quantities. 
As an input the Voronoi-LSQ algorithm takes a finite set $K_0$, the rank parameters $r,s\in\N_0$ of the Minkowski tensors, the number $n\geq d+1$ of equations as well as the maximal radius $R_n$.
The smallest radius $R_1$ used is the average nearest neighbor distance in $K_0$, that is
\begin{align} \label{Def:av}
av(K_0):=  \frac{1}{|K_0|} \sum_{x\in K_0} \min\{\|x-z\|:z\in K_0\setminus\{x\}\},
\end{align}
where $|K_0|$ denotes the cardinality of $K_0$. 
The other radii are chosen equidistantly between $R_1$ and $R_n$.
Therefore we have the additional condition that $R_n$ has to be greater than $av(K_0)$, and we recommend to choose $R_n$ to be at least $(d+1)\cdot av(K_0)$.
The resolution $a$ of the random grid $\eta$ is also chosen as $av(K_0)$.

It is also possible to provide the algorithm with an observation window
\begin{align*}
W = \prod_{i=1}^d  [a_i,b_i], \quad a_i<b_i \text{ for  } i\in \{1,\ldots,d\},
\end{align*}
instead of $R_n$.
In this case the algorithm computes the distance of the data from the boundary of the observation window in each direction and chooses the minimum of these values for $R_n$.
By the distance of $K_0$ from the boundary of $W$ in some direction $s\cdot e_j$, where $s\in\{-1,1\}$ and $e_j$ denotes the $j$-th standard unit vector, we mean
\begin{align*}
&\min\{ b_j-x_j : x=(x_1,\dots,x_d)^\top\in K_0 \}, \quad s=1, \\
&\min\{ x_j-a_j : x=(x_1,\dots,x_d)^\top\in K_0 \}, \quad s=-1.
\end{align*}
The output of the Voronoi-LSQ algorithm is a list containing the $d+1$ estimated Minkowski tensors $\widehat{\Phi}_d^{r,s}(K_0,a,\RR),\dots,\widehat{\Phi}_0^{r,s}(K_0,a,\RR)$ (in this order), where a tensor is given as a dictionary whose keys are the indices of the individual values.
Note that the tensors are symmetric and therefore the algorithm does not estimate all the values since some of them are redundant.
For a tuple of indices $(i_1,\dots,i_{r+s})\in \{1,\ldots,d\}^{r+s}$ with $i_1\leq\dots\leq i_{r+s}$ the algorithm estimates the value of the tensor corresponding to these indices, but not those corresponding to permutations of the tuple since they are the same.
Table~\ref{table:Input} shows a small summary of the input of the Voronoi-LSQ algorithm.

\begin{table}[t!]
  \captionsetup{ labelfont = {bf}, format = plain, belowskip = 0.1ex, width = \textwidth }
  \centering
\caption{Input parameters of the Voronoi-LSQ algorithm}
\begin{tabular}{ll}
   \toprule
   $K_0$ & finite input data \\
   $r,s$ & rank parameters of the Minkowski tensors \\
   $n$ & number of equations considered ($n\geq d+1$) \\
   $R_n$ & maximal radius considered (cf. \eqref{Faustregel}) \\
   $a$ & resolution of the grid $\eta$ (optional; default: $av(K_0)$) \\
   $W$ & observation window containing $K_0$ (optional instead of $R_n$) \\
   $\mathrm{jobs}$ & number of cores used for parallelization (optional; default: -1, i.e., all available) \\
   \bottomrule
\end{tabular}
\label{table:Input}
\end{table}

\subsection{Implementation}
\label{sec:implementation}

Both the Voronoi-FD and the Voronoi-LSQ algorithms are implemented in Python, and the code is available via an open-source package~\cite{VorominkCode}.
One can apply the algorithm by using the Python functions \texttt{Voronoi\_LSQ} and \texttt{Voronoi\_FD} in the file \texttt{VorominkEstimation.py}.
Details can be found in the description of the code, which can also be found in \cite{VorominkCode}.
Closely following the above described estimation of the Minkowski tensors, the user chooses some rank parameters $r,s\in\N_0$. Then the Voronoi-LSQ algorithm will return estimators of the $d+1$ Minkowski tensors \ten r s 0$(K),\dots,$\ten r s d$(K)$, based on the input data $K_0$ and the dimension $d$, for a (in general) unknown set $K$. 
Both algorithms also offer the option to rotate the random grid $\eta$ on which the algorithm is based.
This option is useful for test data that lies somehow parallel to the axis spanned by the standard vectors.
For real data this rotation should, in general, not be necessary and can be skipped to further reduce the run time (especially in high dimensions).

We found the run times of our code for the Voronoi-LSQ algorithm to be convenient for practically relevant examples.
For example, a single rendition of the simulations for a two-dimensional rectangle (results can be found in the supplement), containing about 150,000 points, took an average of 30 seconds on a standard personal computer.
In this case, the algorithm was applied for 5 different rank parameter tuples $(r, s)$ (all possible combinations with $r + s \leq 2$), and for each tuple, 50 runs were performed (resulting in a total of 250 runs and a computation time of 1.5 hours).
The computation time scales with the number of points of the grid process $\eta$. Moreover, the computation time increases strongly with the dimension $d$ and with a refinement of the resolution of the grid process, which is taken as the average nearest neighbor distance $av(K_0)$ of the input data $K_0$. 
For large data sets, the computation time can be further decreased by increasing the number of cores used in an efficient parallelization.
According to Theorem~\ref{Le4.2:convergence}, it would theoretically be better to choose the grid resolution even smaller than $av(K_0)$.
However, in our simulations, we found that this did not significantly improve the results but considerably increased the runtime.

\subsection{Choice of parameters for the Voronoi-LSQ algorithm}
\label{sec:choice}

The choices of the parameters $n,R_n$ have a major influence on the results of the Voronoi-LSQ algorithm.
Therefore, we share the experience we made while working with the algorithm.
For the parameter $n$ we obtained better results for higher values.
We often used the value $n=50$. 

The choice of the parameter $R_n$ is much more intricate.
The theory requires the value of $R_n$ to be smaller than the reach of the original set $K$ that is represented by the data $K_0$.
In practice, this value is usually unknown. 
If the original set is supposed to be convex, $R_n$ could in principle be arbitrarily big.
In the general case, one should choose this parameter depending on an assessment of the data.
It is important that $R_n$ is not too small.
We recommend it to be at least $d+1$ times the average nearest neighbor distance $av(K_0)$ in the test data $K_0$, that is,
\begin{align} \label{Faustregel}
    R_n \,\geq\, (d+1)\, av(K_0).
\end{align}
From our experience, this value sometimes needs to be increased further and \eqref{Faustregel} should be treated as a lower bound.
For more advanced applications and requirements, Bayesian optimization can be a helpful tool for obtaining
the best possible parameters.

\section{Test simulations for the Voronoi-LSQ algorithm} \label{sec:TestLSQ}

In this section, we demonstrate the performance of the Voronoi-LSQ
algorithm for some test cases. Each test data set $A_0$ used in this
section represents some simple set $A$, like a sphere or a rectangle.
The test data set $A_0$ is obtained by intersecting $A$ with a cubic
grid with resolution (or lattice constant) $\Delta$. Note that in this case
$\Nsf(A_0)=av(A_0)=\Delta$.

As a first step, we compare the Voronoi-LSQ algorithm to the original
algorithm~\cite{hug_voronoi-based_2017} that is based on a
matrix inversion and that has been implemented in \cite[Section
3.2]{christensen}. For this purpose, we used the same setup and
parameters as in Figure 3.4(c) of \cite{christensen}. Specifically, we
estimate the surface area of a two-dimensional ball of radius 1 using
the Voronoi-LSQ algorithm.
We repeat this estimation using  
20 different values of $R_n$; these values are equal to those of $R_{\max}$ from Figure 3.4(c) in \cite{christensen}.
In each case, we use $n=50$.
As in \cite{christensen}, we choose $\Delta = 0.001$. Our
estimated values  range between
3.14152 and 3.14147 for the different values of $R_{n}$. These values correspond to an error of about only
53--87\% of the previous error produced by the matrix inversion
algorithm. In that sense, the accuracy can improve by almost up to a
factor of two. The improvement compared to the algorithm introduced in
Section 3.3 of \cite{christensen} is even larger (see Figure 3.4(c)
in \cite{christensen}).

In the next two subsections, we test the Voronoi-LSQ algorithm on
simple examples from different classes of sets.
In Subsection \ref{sec:ConvergenceAnalysis}, we then perform a
convergence analysis to examine the convergence rate of the
Voronoi-LSQ algorithm for the examples considered in Subsections
\ref{sec:ConvexTests} and \ref{sec:NonconvexTests}.

\subsection{Convex test cases}\label{sec:ConvexTests}

In this section we present some simulation results showing how the Voronoi-LSQ algorithm works for specific simulated test data.
We tested the algorithm for rectangles in dimension $d\in\{2,3\}$ as well as for parallel sets of a rectangle in dimension 2.
Table~\ref{table:formulas_rectangle} shows formulas for the Minkowski tensors with rank at most 2 of a rectangle of the form $[-\frac{p_1}{2},\frac{p_1}{2}]\times[-\frac{p_2}{2},\frac{p_2}{2}]\subset\R^2$.
These formulas can be obtained by the following simplification of the Minkowski tensors of polytopes.
 For a polytope $P\subset\R^d$ we obtain from \eqref{eq:interpretlambda1}, \eqref{02-1.5} and \eqref{02-1.6} that
\begin{align*}
\Phi_k^{r,s}(P) &= \frac{1}{r!s!} \frac{1}{\omega_{d-k+s}} \sum_{F\in \mathcal{F}_k(P)} \int_F \int_{N(P,F)\cap \mathbb{S}^{d-1}} x^r u^s \,\mathcal{H}^{d-k-1}(\mathrm{d}u) \,\mathcal{H}^{k}(\mathrm{d}x), \qquad k<d, \\
\Phi_d^{r,0}(P) &= \frac{1}{r!} \int_P x^r \, \mathrm{d}x, \qquad \qquad \qquad \Phi_d^{r,s}(P) = 0, \qquad s>0.
\end{align*}

\begin{example}\label{example:4.5} {\rm 
Let $R_p=[-\frac{p_1}{2},\frac{p_1}{2}]\times[-\frac{p_2}{2},\frac{p_2}{2}]\subset\R^2$ be a two-dimensional rectangle with side lengths $p_1$ and $p_2$ and $p=(p_1,p_2)^\top\in\R^2$.
Then the above formula yields for $r=2,s=0,k=1$ that
\begin{align*}
\Phi_1^{2,0}(R_p) = \frac{1}{4} \sum_{F\in \mathcal{F}_1(R_p)} \int_F  x^2 \,\mathcal{H}^{1}(\mathrm{d}x).
\end{align*}
Using $x^2=x_1^2e_1^2+2x_1x_2e_1e_2+x_2^2e_2^2$, for an edge of the form $F=[-\frac{p_1}{2},\frac{p_1}{2}]\times\{\pm\frac{p_2}{2}\}$ we obtain
\begin{align*}
\int_F  x^2 \,\mathcal{H}^{1}(\mathrm{d}x) &= \int_{-\frac{p_1}{2}}^\frac{p_1}{2} y^2e_1^2 \pm yp_2e_1e_2 + \frac{p_2^2}{4} e_2^2 \, \mathrm{d}y 
= \frac{p_1^3}{12}e_1^2 + \frac{p_1p_2^2}{4}e_2^2.
\end{align*}
Therefore the summation over all the edges of $R_p$ yields
\begin{align*}
\Phi_1^{2,0}(R_p) 
&= \frac{1}{24} \left( (p_1^3+3p_1^2p_2)e_1^2 + (3p_1p_2^2+p_2^3)e_2^2 \right).
\end{align*}
 
For $r=s=k=1$ we obtain
\begin{align*}
\Phi_1^{1,1}(R_p) &= \frac{1}{2\pi} \sum_{F\in \mathcal{F}_1(R_p)} \int_F \int_{N(R_p,F)\cap \mathbb{S}^{1}} x u \,\mathcal{H}^{0}(\mathrm{d}u) \,\mathcal{H}^{1}(\mathrm{d}x).
\end{align*}
As before we start by computing the integral for an edge of the form $F=[-\frac{p_1}{2},\frac{p_1}{2}]\times\{\pm\frac{p_2}{2}\}$.
The normal cone of this set intersected with the unit sphere is given by $\{\pm e_2\}$.
Since the tensor $xe_2$ can be written as $x_1e_1e_2+x_2e_2^2$, we obtain
\begin{align*}
\int_F \int_{N(R_p,F)\cap \mathbb{S}^{1}} x u \,\mathcal{H}^{0}(\mathrm{d}u) \,\mathcal{H}^{1}(\mathrm{d}x) &= \pm \int_F xe_2 \,\mathcal{H}^{1}(\mathrm{d}x) =
 \pm \int_{-\frac{p_1}{2}}^\frac{p_1}{2} ye_1e_2 \pm \frac{p_2}{2} e_2^2 \, \mathrm{d}y 
= \frac{p_1p_2}{2} e_2^2.
\end{align*}
After taking the summation over all edges, we get
\begin{align*}
\Phi_1^{1,1}(R_p) &= \frac{p_1p_2}{2\pi} (e_1^2+e_2^2)=\frac{p_1p_2}{2\pi}\, Q.
\end{align*}}
\end{example}

The results of the Voronoi-LSQ algorithm for a two-dimensional rectangle can be found in the supplement and are very close to the exact values.
We also tested a three-dimensional rectangle in the same way and obtained very good results.
Further we considered parallel sets of a two-dimensional rectangle and the results again were very close to the exact values (relative deviations of less
than 1\% for rectangles of varying sizes with up to 1.3 million points).
The detailed results, along with those for a three-dimensional rectangle, where we once again achieved very good outcomes, can be found in the supplementary materials.

\begin{table}[t!]
\centering
\captionsetup{
  labelfont = {bf},
  format = hang,
  belowskip = 1ex,
  width = \textwidth
}
\caption{Formulas for the Minkowski tensors with rank at most 2 of a rectangle of the form $[-\frac{p_1}{2},\frac{p_1}{2}]\times[-\frac{p_2}{2},\frac{p_2}{2}]\subset\R^2$, where the values of \ten 1 0 k, \ten 0 1 k vanish for all $k\in\{0,1,2\}$.}
\scalebox{0.95}{
\begin{tabular}{cc|cc|cc}
  \toprule
  Tensor & Formula & Tensor & Formula & Tensor & Formula \\
  \midrule
  
  \ten 0 0 0 & 1 & \ten 0 0 1 & $p_1+p_2$ & \ten 0 0 2 & $p_1\cdot p_2$ \\ [1ex]
  \ten 0 2 0 & $\frac{1}{4\pi}\left( e_1^2+e_2^2 \right)$ & \ten 0 2 1 & $\frac{1}{4\pi}\left( p_2e_1^2+p_1e_2^2 \right)$ & \ten 0 2 2 & 0  \\ [1ex]
  \ten 2 0 0 & $\frac{1}{8}\left( p_1^2e_1^2+p_2^2e_2^2 \right)$ & \ten 2 0 1 & $\frac{1}{24}\sum_{i=1}^2 \left( p_i^3 + 3p_i^2p_j \right) e_i^2 $ & \ten 2 0 2 & $\frac{1}{24}\left( p_1^3p_2e_1^2+p_2^3p_1e_2^2 \right)$ \\ [1ex]
  \ten 1 1 0 & $\frac{1}{2\pi}\left( p_1e_1^2+p_2e_2^2 \right)$ & \ten 1 1 1 &  $\frac{p_1p_2}{2\pi}\left( e_1^2+e_2^2 \right)$ & \ten 1 1 2 & 0  \\ [1ex]
\bottomrule
\end{tabular}
}
\label{table:formulas_rectangle}
\end{table}

\subsection{Nonconvex test cases}
\label{sec:NonconvexTests}

If $A\subset\R^d$ is a set of positive reach which is the closure of its interior, the boundary $\partial A$ of $A$ is of class $C^2$, and $A$ has a unique exterior unit outer normal vector $\nu(A,x)$ at each $x\in\partial A$ (that is, $\partial A=\partial^1A$), then we obtain as in \eqref{eq:interpretlambda2} that
$$
\Lambda_j(A,\cdot)=\frac{1}{\omega_{d-j}}\int_{\partial A}\1\{(x,\nu(A,x))\in\cdot\} \sum_{|I|=d-1-j}\prod_{i\in I}k_i(A,x)\, \mathcal{H}^{d-1}(\d x),
$$
where $k_1(A,x),\ldots,k_{d-1}(A,x)\in \R$ are the principal curvatures of $A$ at $x\in\partial A$ (with respect to the outer unit normal) and the summation extends over all subsets $I\subseteq\{1,\ldots,d-1\}$ of cardinality $|I|=d-1-j$ (if $j=d-1$ the empty product is interpreted as $1$). Note that if $A$ is convex, then $k_i(K,x)\ge 0$ for $i=0,\ldots,d-1$. Thus, for $k\in\{0,\dots,d-1\}$,  it follows from \eqref{02-1.5} that
$$
\Phi^{r,s}_k(A)=\frac{1}{r!s!\omega_{d-k+s}}\int_{\partial A}x^r\nu(A,x)^s\sum_{|I|=d-1-k}\prod_{i\in I}k_i(A,x)\, \mathcal{H}^{d-1}(\d x).
$$

\begin{example} \label{Ex:Shell} {\rm 
Let $A=\{x\in\R^d:\rho_1\le \|x\|\le \rho_2\}$ denote a spherical shell with inner radius $\rho_1$ and outer radius $\rho_2$, where $0\le \rho_1<\rho_2<\infty$. 
Observe that the reach of $A$ is $\rho_1$ whenever $\rho_1>0$ (if $ \rho_1=0$, then  $A$ is a Euclidean ball with radius $\rho_2$).
Since the principal curvatures of $A$ at  $x\in\partial A$  are equal to $\rho_2^{-1}$, if $\|x\|=\rho_2$, and equal to $-\rho_1^{-1}$, if $\|x\|=\rho_1>0$, we simply write $k(A,x)$ for these principal curvatures and get
\begin{align*}
    \Phi^{r,s}_k(A)&=\frac{\binom{d-1}{k}}{r!s!\omega_{d-k+s}}\int_{\partial A}x^r\nu(A,x)^sk(A,x)^{d-1-k} \, \mathcal{H}^{d-1}(\d x)\\
    &=\frac{\binom{d-1}{k}}{r!s!\omega_{d-k+s}}
    \int_{ \bbS^{d-1}} \left((\rho_2 u)^ru^s \rho_2^{k+1-d}\rho^{d-1}_2 + (\rho_1 u)^r (-u)^s (-\rho_1)^{k+1-d}\rho^{d-1}_1\right)\,\mathcal{H}^{d-1}(\d u)\\
    &=\frac{\binom{d-1}{k}}{r!s!\omega_{d-k+s}}\left(\rho_2^{r+k}+(-1)^{s+d-1-k}\rho_1^{r+k}\right) \int_{ \bbS^{d-1}}u^{r+s}\, \mathcal{H}^{d-1}(\d u)\\
    &=\1\{r+s\text{ even}\}\frac{\binom{d-1}{k}2\omega_{d+r+s}}{r!s!\omega_{d-k+s}\omega_{r+s+1}}\left(\rho_2^{r+k}+(-1)^{s+d-1-k}\rho_1^{r+k}\right)Q^{\frac{r+s}{2}},
\end{align*}
where the transformation formula and \cite[Lemma~7]{HW18} have been used to determine the integral. 

If $k=d-1$, this specializes to
$$
 \Phi^{r,s}_{d-1}(A)=\1\{r+s\text{ even}\}\frac{2\omega_{d+r+s}}{r!s!\omega_{s+1}\omega_{r+s+1}}\left(\rho_2^{r+d-1}+(-1)^{s}\rho_1^{r+d-1}\right)Q^{\frac{r+s}{2}},
$$
and if $k=d-1$ and $r=0$, we obtain
$$
 \Phi^{0,s}_{d-1}(A)=\1\{s\text{ even}\}\frac{2\omega_{d+s}}{s!\omega_{s+1}^2}\left(\rho_1^{d-1}+\rho_2^{d-1}\right)Q^{\frac{s}{2}}.
$$
Table \ref{table:Shell} shows simulation results of the Voronoi-LSQ algorithm for the two-dimensional spherical shell with inner radius $\rho_1=1$ and outer radius $\rho_2=2$.}
\end{example}

\begin{table}[t!]
\centering
\captionsetup{
  labelfont = {bf},
  format = plain,
  belowskip = 1ex,
  width = \textwidth
}
\caption{Results of the Voronoi-LSQ algorithm for the two-dimensional spherical shell with inner radius $\rho_1=1$ and outer radius $\rho_2=2$ intersected with a grid of resolution $\Delta=0.005$.
The parameter choices were $n=50$ and $R_{n}=0.9$ (0.9 times the reach of the spherical shell).
We took the average of 25 renditions.
All values are rounded to the fourth significant digit.}
\begin{tabular}{ccc|ccc}
  \toprule
  Tensors &  Exact value & Voronoi-LSQ & Tensors & Exact value & Voronoi-LSQ \\
  \midrule
  \
  \ten 0 0 0 & 0 & $-10^{-3}$ & \ten 0 0 1 & 9.4248 & 9.4441 \\ [1ex]
  \ten 0 0 2 & 9.4248 & 9.4018 & (\ten 0 2 1)$_{1,1}$ & 0.3749 & 0.3749 \\ [1ex]
  (\ten 0 2 1)$_{1,2}$ & 0 & $10^{-4}$ & (\ten 0 2 1)$_{2,2}$ & 0.3749 & 0.3748 \\ [1ex]
  (\ten 1 1 1)$_{1,1}$ & 1.5 & 1.495 & (\ten 1 1 1)$_{1,2}$ & 0 & $-10^{-3}$ \\ [1ex]
  (\ten 1 1 1)$_{2,2}$ & 1.5 & 1.492 & & & \\ [1ex]
\bottomrule
\end{tabular}
\label{table:Shell}
\end{table}

\begin{table}[b!]
\centering
\captionsetup{
  labelfont = {bf},
  format = plain,
  belowskip = 1ex,
  width = \textwidth
}
\caption{Results of the Voronoi-LSQ algorithm for $R_{p,q}$ with $p=(1,2)^\top$, $q=(3,5)^\top$ intersected with a grid of resolution $\Delta=0.01$.
The parameter choices were $n=50$ and $R_{n}=0.45$ ($0.9\cdot0.5$ times the length of the smaller side of $R_p$).
We took the average of 25 renditions.
}
\begin{tabular}{ccc|ccc}
  \toprule
  Tensors &  Exact value & Voronoi-LSQ & Tensors & Exact value & Voronoi-LSQ \\
  \midrule
  \
  \ten 0 0 0 & 0 & -0.2837 & \ten 0 0 1 & 11 & 11.01  \\ [1ex]
  \ten 0 0 2 & 13 & 13.00 & (\ten 0 2 1)$_{1,1}$ & 0.5570 & 0.5477 \\ [1ex]
  (\ten 0 2 1)$_{1,2}$ & 0 & $-10^{-4}$ & (\ten 0 2 1)$_{2,2}$ & 0.3183 & 0.3110 \\ [1ex]
  (\ten 1 1 1)$_{1,1}$ & 2.069 & 2.073 & (\ten 1 1 1)$_{1,2}$ & 0 & $10^{-3}$ \\ [1ex]
  (\ten 1 1 1)$_{2,2}$ & 2.069 & 2.073 & & & \\ [1ex]
\bottomrule
\end{tabular}
\label{table:cuttedRect}
\end{table}

\begin{example} \label{Ex:cuttedRect} {\rm 
Let $R_p=[-\frac{p_1}{2},\frac{p_1}{2}]\times[-\frac{p_2}{2},\frac{p_2}{2}]\subset\R^2$ be a two-dimensional rectangle with side lengths $p_1$ and $p_2$ and $p=(p_1,p_2)^\top\in\R^2$. For $0<p_1<q_1$ and $0<p_2<q_2$ we consider $R_{p,q}:=R_q\setminus R_p^\circ$, i.e., the interior of the rectangle $R_p$ is removed from the enclosing rectangle $R_q$. The set $R_{p,q}$ does not have positive reach, but it is a polyconvex set and in particular a finite union of compact sets with positive reach. The surface tensors $\Phi_1^{r,s}(R_{p,q})$ can be expressed in terms of the surfaces tensors of $R_p$ and $R_q$ by
$$
\Phi_1^{r,s}(R_{p,q})=\Phi_1^{r,s}(R_{q})+(-1)^s\Phi_1^{r,s}(R_{p}),
$$
where $r,s\in\N_0$; see Table~\ref{table:formulas_rectangle} for explicit formulas. Table \ref{table:cuttedRect} shows simulation results of the Voronoi-LSQ algorithm for $R_{p,q}$ with $p=(1,2)^\top$, $q=(3,5)^\top$.}
\end{example}

\subsection{Convergence analysis} \label{sec:ConvergenceAnalysis}

\begin{figure}[t]
  \captionsetup{ labelfont = {bf}, format = plain }
  \centering
  \includegraphics[height=0.6\linewidth]{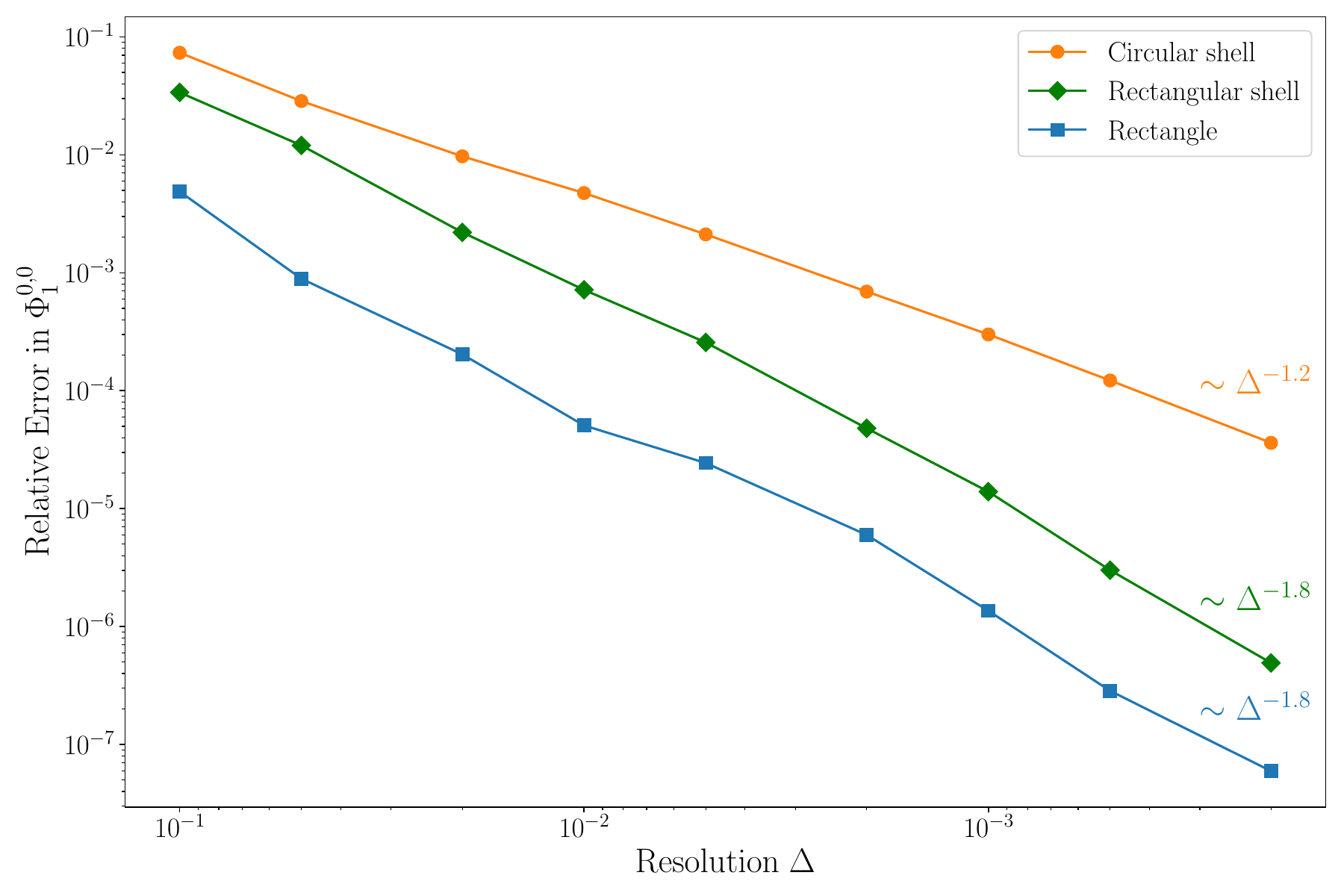}
  \caption{The relative error in estimating  $\Phi_1^{0,0}$ (which equals half of the surface area) of the sets from Examples \ref{example:4.5}, \ref{Ex:Shell} and \ref{Ex:cuttedRect} using the Voronoi-LSQ algorithm.
  The test data were obtained by intersecting the sets with a grid of resolution $\Delta$ and the relative error is shown as a function of $\Delta$.}
  \label{fig:ConvergenceDiagram}
\end{figure}

After comparing the results of the Voronoi-LSQ algorithm for a fixed choice of $R_n$ with the true values, we now investigate the dependence of the error on the resolution of the test sets.
Recall that the test data were generated by intersecting the test sets with a grid of a given resolution.
Figure \ref{fig:ConvergenceDiagram} shows the relative error of the estimated surface areas for the sets from Examples \ref{example:4.5}, \ref{Ex:Shell} and \ref{Ex:cuttedRect}, 
plotted as a function of the resolution, which ranges over about three orders of magnitude.
Over this range, the error in the surface area drops by about five orders of magnitude.
All algorithm parameters and the precise specifications of the example sets were adopted exactly as in the simulations presented in Subsections \ref{sec:ConvexTests} and \ref{sec:NonconvexTests}.
As theoretically expected, the error of the algorithm's results converges to zero as the resolution decreases for the rectangle and the circular shell.
Our simulations suggest that this holds for the rectangular shell as well, although no theoretical proof is available.
A relative error of 0.1\% is achieved at a resolution of 0.002 (circular shell), 0.01 (rectangular shell), and 0.05 (rectangle), respectively.

\section{Isotropic random polytopes}
\label{sec:beta}

Computing Minkowski tensors of random polygons can be beneficial for a
range of applications, examples are convex hulls in image classification
and object detection~\cite{preparata_computational_1985,
yang_image_1999}, as well as modeling and characterizing the structure
of cellular materials, like polycrystalline metals, foams, or biological
tissues~\cite{schroder-turk_tensorial_2010,
schroder-turk_minkowski_2011, klatt_cell_2017}. Here we demonstrate the
reliability of our Voronoi-LSQ algorithm also for random shapes by
applying it to a prominent model of random convex polygons, known as
beta-polytopes; for examples, see Fig.~\ref{fig:betaPolytopes}.

We recall that, in cases where it is known that the set to be estimated
has positive reach (which is the case for convex sets), the Voronoi-LSQ
algorithm is more suitable than the Voronoi-FD algorithm, as it requires
significantly less computation time.
Therefore, we restrict our attention in this section to the Voronoi-LSQ
algorithm.

In Section~\ref{sec:betatheo} we first show how the expected
value of the Minkowski tensors $\Phi^{0,s}_k(Z)$ for an isotropic random
polytope $Z$ can be expressed in terms of the metric tensor $Q$ and the
expected value of the intrinsic volume $V_k(Z)$. From this relation we
then derive explicit formulas for beta-polytopes, a  class of isotropic
random polytopes that has proved to be useful in stochastic geometry
(see, e.g., \cite{KABLUCHKO2020107333}). In Section~\ref{sec:betasim},
the exact evaluations of the resulting formulas are compared to the
results of a simulation study based on our algorithm.

\begin{figure}[t]
  \captionsetup{ labelfont = {bf}, format = plain }
  \centering
  \subfloat[][]{\includegraphics[height=0.45\linewidth]{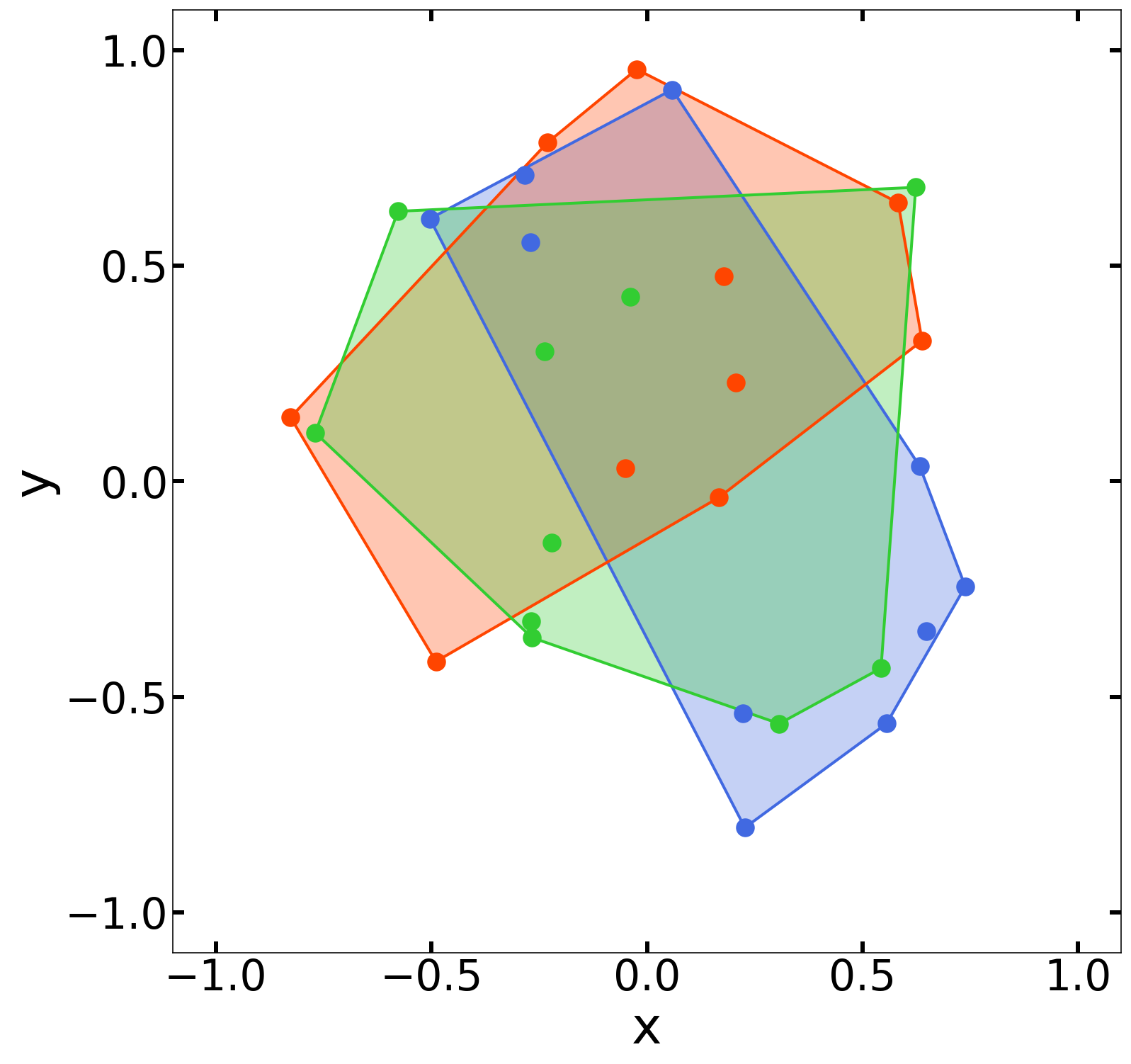}}
  \quad
  \subfloat[][]{\includegraphics[height=0.45\linewidth]{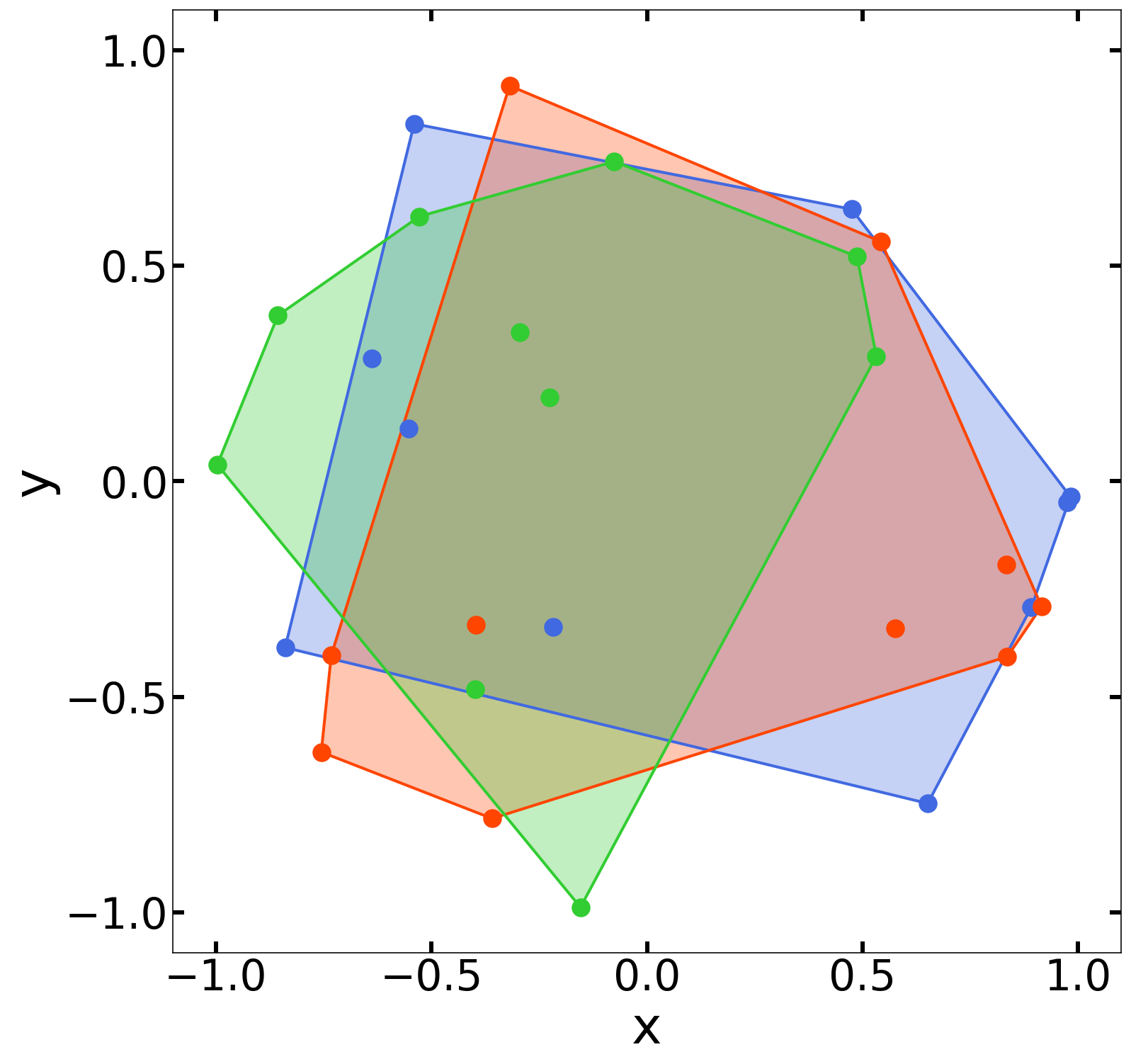}}
  \caption{Three independent realizations of a two-dimensional
  beta-polytope $P_{10,2}^\beta$ for (a) $\beta=\frac{1}{2}$ and (b)
  $\beta=-\frac{1}{2}$.}
  \label{fig:betaPolytopes}
\end{figure}

\subsection{Exact expectations}
\label{sec:betatheo}

In this section, we focus on the case $r=0$. A random polytope $Z$ in $\R^d$ is said to be isotropic if $\rho Z$ and $Z$ have the same  distribution for all $\rho\in SO(d)$.  

\begin{lemma}\label{lem:6.1}
Let $Z$ be an isotropic random polytope in $\R^d$. Let $k\in\{0,\ldots,d-1\}$. If $\mathbb{E} V_k(Z)<\infty$, then
$$
\BE \Phi^{0,s}_k(Z)=\1\{s \text{ \rm even}\}\frac{2\omega_{d+s}\omega_{d-k}}{s!\omega_d\omega_{s+1}\omega_{d-k+s}}\BE  V_k(Z)\cdot Q^{\frac{s}{2}};
$$
in particular, the surface tensors are given by
\begin{equation}\label{eq:6.2}
\BE \Phi^{0,s}_{d-1}(Z)=\1\{s \text{ \rm even}\}\frac{4\omega_{d+s}}{s!\omega_d\omega_{s+1}^2}\BE  V_{d-1}(Z)\cdot Q^{\frac{s}{2}},
\end{equation}
\end{lemma}

\begin{proof}
Let $\rho\in SO(d)$. Since $Z$ is isotropic, we deduce from the definition of $\Phi^{0,s}_k(Z)$ that
\begin{align}\label{eq:6.1}
\BE \Phi^{0,s}_k(Z)&=\frac{1}{s!\omega_{d-k+s}}\BE\sum_{F\in \mathcal{F}_k(\rho Z)}   V_k(F)\int_{N(\rho Z,F)\cap \bbS^{d-1}} u^s\, \mathcal{H}^{d-k-1}(\mathrm{d} u)\nonumber\\
&=\frac{1}{s!\omega_{d-k+s}}\BE\sum_{F\in \mathcal{F}_k(Z)}   V_k(\rho F)\int_{\rho N( Z,F)\cap \bbS^{d-1}} u^s\, \mathcal{H}^{d-k-1}(\mathrm{d} u)\nonumber\\
&=\frac{1}{s!\omega_{d-k+s}}\BE\sum_{F\in \mathcal{F}_k(Z)}   V_k( F)\int_{ N( Z,F)\cap \bbS^{d-1}} (\rho v)^s\, \mathcal{H}^{d-k-1}(\mathrm{d} v).
\end{align}
Integration of \eqref{eq:6.1} over all $\rho\in SO(d)$ with respect to the Haar probability measure $\nu$ on $SO(d)$ and Fubini's theorem yield
\begin{align*}
\BE \Phi^{0,s}_k(Z)&=
\frac{1}{s!\omega_{d-k+s}}\BE\sum_{F\in \mathcal{F}_k(Z)}   V_k( F)\int_{ N( Z,F)\cap \bbS^{d-1}}\int_{SO(d)} (\rho v)^s\,
\nu(\mathrm{d}\rho)\,\mathcal{H}^{d-k-1}(\mathrm{d} v).
\end{align*}
The application of Fubini's theorem (in particular, interchanging expectation and integration) is justified since  $\mathbb{E} V_k(Z)<\infty$. From
$$
\int_{SO(d)} (\rho v)^s\,
\nu(\mathrm{d}\rho)=\frac{1}{\omega_d}\int_{\bbS^{d-1}} w^s\, \mathcal{H}^{d-1}(\mathrm{d}w)=\1\{s\text{ \rm even}\}\frac{2\omega_{d+s}}{\omega_d\omega_{s+1}}Q^{\frac{s}{2}}
$$
(see, e.g., \cite[Lemma~7]{HW18} and the reference given there) we deduce that
\begin{align*}
\BE \Phi^{0,s}_k(Z)&=\1\{s\text{ \rm even}\}\frac{2\omega_{d+s}}{s!\omega_d\omega_{s+1}\omega_{d-k+s}}Q^{\frac{s}{2}}\BE \sum_{F\in\mathcal{F}_k(Z)}V_k(F)\mathcal{H}^{d-k-1}(N(Z,F)\cap \bbS^{d-1})\\
&=\1\{s\text{ \rm even}\}\frac{2\omega_{d+s}\omega_{d-k}}{s!\omega_d\omega_{s+1}\omega_{d-k+s}}\BE V_k(Z)\cdot Q^{\frac{s}{2}},
\end{align*}
which proves the assertion.
\end{proof}

We consider the case where $Z=P^\beta_{\nn,d}$ is a beta-polytope with $\beta>-1$. To define this class of random polytopes, let
$$f_{d,\beta}(x):=c_{d,\beta} \1\{\|x\|<1\}\left( 1-\left\| x \right\|^2 \right)^\beta,\quad x\in\R^d,$$
with the normalizing constant  
$$
c_{d,\beta}:= \frac{ \Gamma\left( \frac{d}{2} + \beta + 1 \right) }{ \pi^{ \frac{d}{2} } \Gamma\left( \beta+1 \right) }.
$$
Hence $f_{d,\beta}$ is the density of a beta-distribution in $\R^d$.  

Let $X_1,\ldots,X_\nn$ be i.i.d.~random points in $\R^d$ whose common distribution function has Lebesgue density $f_{d,\beta}$. If $\nn\ge d+1$, then
$$
P^{\beta}_{\nn,d}:=\conv\{X_1,\ldots,X_\nn\}
$$
is called a beta-polytope (see \cite{KABLUCHKO2020107333})  in $\R^d$ (with parameters $\beta$ and $\nn$). Since the density $f_{d,\beta}$ is rotation invariant, the distribution  of
$X_1,\ldots,X_\nn$ is rotation invariant, and hence $P^{\beta}_{\nn,d}$ is an isotropic random polytope.
Figure \ref{fig:betaPolytopes} shows three independent realisations of a beta-polytope $P_{10,d}^\beta$ for $\beta\in\{\frac{1}{2},-\frac{1}{2}\}$.
It can be seen here that in the case $\beta = \frac{1}{2}$, points closer to the origin are favored compared to the case $\beta = -\frac{1}{2}$, which leads to a  random polytope having a smaller expected intrinsic volume. 
In the following we provide explicit expressions whose numerical evaluation confirms this intuitive observation.
It follows from
\cite[Corollary 2.4]{KTT19} that
\begin{align*}
\BE V_{d-1} (P^\beta_{\nn,d})&=\frac{d (2\beta+d+1)}{2^d\Gamma\left(\frac{d}{2}\right)} \binom{\nn}{d} 
\left(\frac{\Gamma\left(\beta+\frac{d+2}{2}\right)}{\Gamma\left(\beta+\frac{d+3}{2}\right)}\right)^d
\!\!\int_{-1}^1(1-h^2)^{d\beta+\frac{(d-1)(d+2)}{2}}F_{1,\beta+\frac{d-1}{2}}(h)^{\nn-d}\, \mathrm{d} h
\end{align*}
with 
$$
F_{1,\beta+\frac{d-1}{2}}(h):=\frac{\Gamma\left(\beta+\frac{d+2}{2}\right)}{\sqrt{\pi}\Gamma\left(\beta+\frac{d+1}{2}\right)}\int_{-1}^h(1-x^2)^{\beta+\frac{d-1}{2}}\, \mathrm{d}x,\quad h\in [-1,1].
$$
This expectation can be rewritten in the form
\begin{align*}
\BE V_{d-1} (P^\beta_{\nn,d})&=
\frac{d\pi^{\frac{d}{2}}}{(2\beta+d+1)^{d-1}\Gamma\left(\frac{d}{2}\right)} \binom{\nn}{d} 
\left(\frac{\Gamma\left(\beta+\frac{d+2}{2}\right)}{\sqrt{\pi}\Gamma\left(\beta+\frac{d+1}{2}\right)}\right)^\nn\\
&\qquad\times \int_{-1}^1(1-h^2)^{d\beta+\frac{d-1}{2}(d+2)}
\left(\int_{-1}^h(1-x^2)^{\beta+ \frac{d-1}{2}}\, \mathrm{d}x\right)^{\nn-d}\, \mathrm{d}h.
\end{align*}
If $\beta=-\frac{1}{2}$, then
\begin{align*}
\BE V_{d-1} (P^{-\frac{1}{2}}_{\nn,d})
&=
\frac{\pi^{\frac{d}{2}}}{d^{d-2}\Gamma\left(\frac{d}{2}\right)} \binom{\nn}{d} 
\left(\frac{\Gamma\left(\frac{d+1}{2}\right)}{\sqrt{\pi}\Gamma\left(\frac{d}{2}\right)}\right)^\nn
\int_{-1}^1\sqrt{1-h^2}^{d^2-2}
\left(\int_{-1}^h\sqrt{1-x^2}^{ d-2}\, \mathrm{d}x\right)^{\nn-d}\, \mathrm{d}h.
\end{align*}
For $d=2$ and $\nn\ge 3$ this simplifies to
$$
\BE V_{1} (P^{-\frac{1}{2}}_{\nn,2})=\pi\frac{\nn-1}{\nn+1}.
$$
If $\beta=\frac{1}{2}$, then
\begin{align*}
\BE V_{d-1} (P^{\frac{1}{2}}_{\nn,d})
&=
\frac{d\pi^{\frac{d}{2}}}{(d+2)^{d-1}\Gamma\left(\frac{d}{2}\right)} \binom{\nn}{d} 
\left(\frac{\Gamma\left(\frac{d+3}{2}\right)}{\sqrt{\pi}\Gamma\left(\frac{d+2}{2}\right)}\right)^\nn\\
&\qquad\times 
\int_{-1}^1\sqrt{1-h^2}^{d^2+2d-2}
\left(\int_{-1}^h\sqrt{1-x^2}^{ d}\, \mathrm{d}x\right)^{\nn-d}\, \mathrm{d}h.
\end{align*}
If $d=2$, this simplifies to
$$
\BE V_{1} (P^{\frac{1}{2}}_{\nn,2})=\pi\cdot \nn(\nn-1)\frac{9}{4^{\nn+1}}\int_{-1}^1(1-h)^3(1+h)^{2\nn-1}(2-h)^{\nn-2}\, \mathrm{d}h.
$$
Expansion of $(2-h)^{\nn-2}=(1+(1-h))^{\nn-2}$  and repeated partial integration yield
$$
\BE V_{1} (P^{\frac{1}{2}}_{\nn,2})=\pi\cdot9(\nn-1)\sum_{j=0}^{\nn-2}\frac{\binom{\nn-2}{j}}{\binom{2\nn+j+3}{j+3}}2^{j}.
$$

From \eqref{eq:6.2} we have
$$
\BE \Phi^{0,s}_{d-1}(P^\beta_{\nn,d})=\1\{s \text{ \rm even}\}\frac{4\omega_{d+s}}{s!\omega_d\omega_{s+1}^2}\BE  V_{d-1}(P^\beta_{\nn,d})\cdot Q^{\frac{s}{2}},
$$
with $\BE  V_{d-1}(P^\beta_{\nn,d})$ given as above.

Similar formulas can be derived for $\BE \Phi^{0,s}_{k}(P^\beta_{\nn,d})$, where 
$k\in\{0,\ldots,d-1\}$, $s\in\N_0$ or  $k=d$,  $s=0$, from results provided in  \cite{KTT19} in combination with Lemma~\ref{lem:6.1}. The same is true for some other classes of isotropic random polytopes, including  beta$^{\prime}$-polytopes or symmetric versions of beta- and beta$^{\prime}$-polytopes.
For example for the expected intrinsic volumes of a beta-polytope Proposition 2.3 in \cite{KTT19} yields
\begin{align*}
\mathbb{E} V_k\left(P^\beta_{\nn,d}\right) = \binom{d}{k} \frac{\kappa_d}{\kappa_k\kappa_{d-k}} \mathbb{E} V_k\left(P^{\beta+\frac{d-k}{2}}_{\nn,k}\right).
\end{align*}

\subsection{Simulation study}
\label{sec:betasim}

\begin{table}[t!]
\centering
\captionsetup{
  labelfont = {bf},
  format = plain,
  belowskip = 1ex,
  width = \textwidth
}
\caption{Results of the Voronoi-LSQ algorithm for two- and four-dimensional beta-polytopes  $P_{10,d}^\beta$.
We intersected the simulated polytopes with a grid of resolution $\Delta=0.005$ for $d=2$ and $\Delta=0.02$ for $d=4$.  
The parameter choices were $n=50$ and $R_{n}=1$ for $d=2$ ($R_{n}=0.5$ for $d=4$).  
We took the average of 50 and 10 renditions for $d=2$ and $d=4$, respectively, and computed the exact values with the formulas of Section~\ref{sec:betatheo}.
The table shows the values of the expectation \EP {k} {\beta} {10} {d}.}
\begin{tabular}{ccr|cc|ccr|cc}
  \toprule
  $d$ & $k$ & $\beta$ & Exact value & Voronoi-LSQ & $d$ & $k$ & $\beta$ & Exact value & Voronoi-LSQ \\
  \midrule
  
  $2$ & $1$ & $\frac{1}{2}$ & 2.08\dots & 2.05(3) & $4$ & $3$ & $\frac{1}{2}$ & 0.81\dots & 0.84(4) \\ [1ex]
  $2$ & $1$ & $-\frac{1}{2}$ & 2.57\dots & 2.57(3) & $4$ & $3$ & $-\frac{1}{2}$ & 1.35\dots & 1.56(8) \\ [1ex]
  $2$ & $2$ & $\frac{1}{2}$ & 1.06\dots & 1.03(3) & $4$ & $4$ & $\frac{1}{2}$ & 0.107\dots & 0.093(8) \\ [1ex]
  $2$ & $2$ & $-\frac{1}{2}$ & 1.71\dots & 1.72(6) & $4$ & $4$ & $-\frac{1}{2}$ & 0.21\dots & 0.23(2) \\

\bottomrule
\end{tabular}
\label{table:betapolynomes_d=2}
\end{table}

In this section we simulate specific beta-polytopes in dimension 2 and 4 and estimate the Minkowski tensors of the resulting polytopes by using the Voronoi-LSQ algorithm.
We restrict our focus to the Voronoi-LSQ algorithm, as its applicability to convex sets, such as polytopes, is theoretically verified and it is computationally more efficient than the Voronoi-FD algorithm.
Since the formulas from Section~\ref{sec:betatheo} simplify for $d=2$ and $\beta\in\{-\frac{1}{2},\frac{1}{2}\}$, we focus on these two values for $\beta$ (also for $d=4$).
Table~\ref{table:betapolynomes_d=2} shows the results for the expected volume and surface area of a two- and a four-dimensional beta-polytope with the choice $\nn=10$.
For $d=2$ we also simulated the case $\nn=100$ and obtained similar results.
Furthermore, Table~\ref{table:betapolynomes_d=2_s=2} contains the results for the expectation of the tensor $\Phi_{1}^{0,2}(P_{10,2}^{\beta})_{i,j}$ (left) and $\Phi_{1}^{0,4}(P_{10,2}^{\beta})_{i,j,k,l}$ (right) respectively, and Table~\ref{table:betapolynomes_d=4_s=2} gives the results for the expectation of the tensor $\Phi_{3}^{0,2}(P_{10,4}^{\beta})_{i,j}$.

\begin{table}[t!]
\centering
\captionsetup{
  labelfont = {bf},
  format = plain,
  belowskip = 1ex,
  width = \textwidth
}
\caption{Results of the Voronoi-LSQ algorithm for two-dimensional beta-polytopes $P_{10,2}^\beta$.
We intersected the simulated polytopes with a grid of resolution $\Delta=0.005$.
The parameter choices were $n=50$ and $R_{n}=1$.
We took the average of 50 (left) and 10 (right) renditions and computed the exact values with the formulas of Section~\ref{sec:betatheo}.
The table shows the average values of $\Phi_{1}^{0,2}(P_{10,2}^{\beta})$ and $\Phi_{1}^{0,4}(P_{10,2}^{\beta})$ on the left and right side, respectively.}

\begin{tabular}{rrrl|rrrrl|r}
  \multicolumn{5}{c}{$\E\left[ \Phi_{1}^{0,2}(P_{10,2}^{\beta})_{i,j} \right]$} & & \multicolumn{4}{c}{$\E\left[ \Phi_{1}^{0,4}(P_{10,2}^{\beta})_{l,l,l,l} \right]$} \\[1ex]
  \cline{1-5}\cline{7-10}
  $\beta$ & $i$ & $j$ & Exact value & Voronoi-LSQ & \hspace{0.5cm} & $\beta$ & $l$ & Exact value & Voronoi-LSQ \\
  \cmidrule{1-5}\cmidrule{7-10}

  $\frac{1}{2}$ & 1 & 1 & 0.083\dots & 0.084(2) & & $\frac{1}{2}$ & 1 & 0.0025\dots & 0.0025(1) \\[1ex]
  $\frac{1}{2}$ & 1 & 2 & 0 & -0.002(2) & & $\frac{1}{2}$ & 2 & 0.0025\dots & 0.0023(1) \\[1ex]
  $\frac{1}{2}$ & 2 & 2 & 0.083\dots & 0.084(2) & & $-\frac{1}{2}$ & 1 & 0.0031\dots & 0.0030(1) \\[1ex]
  $-\frac{1}{2}$ & 1 & 1 & 0.102\dots & 0.102(2) & & $-\frac{1}{2}$ & 2 & 0.0031\dots & 0.0031(1) \\[1ex]
  \cline{7-10}
  $-\frac{1}{2}$ & 1 & 2 & 0 & 0.001(2) \\[1ex]
  $-\frac{1}{2}$ & 2 & 2 & 0.102\dots & 0.101(2) \\
\cmidrule{1-5}
\end{tabular}

\begin{tablenotes}
The obtained values for index tuples $(i,j,k,l)$ (for the tensor of rank 4), where not all entries are the same, were all not greater than 1.1$\cdot 10^{-3}$ (here the exact values are 0).
\end{tablenotes}
\label{table:betapolynomes_d=2_s=2}
\end{table}

\begin{table}[t!]
\centering
\captionsetup{
  labelfont = {bf},
  format = plain,
  belowskip = 1ex,
  width = \textwidth
}
\caption{Results of the Voronoi-LSQ algorithm for four-dimensional beta-polytopes $P_{10,4}^\beta$.
We intersected the simulated polytopes with a grid of resolution $\Delta=0.02$.
The parameter choices were $n=50$ and $R_{n}=0.5$.
We took the average of 10 renditions and computed the exact values with the formulas of Section~\ref{sec:betatheo}.
The table shows the values of the expectation $\E\left[ \Phi_{3}^{0,2}(P_{10,4}^{\beta})_{i,j} \right]$.}
\begin{tabular}{rrrr|c|rrrr|c}
  \toprule
  $\beta$ & $i$ & $j$ & Exact value & Voronoi-LSQ & $\beta$ & $i$ & $j$ & Exact value & Voronoi-LSQ \\
  \midrule

  $\frac{1}{2}$ & 1 & 1 & 0.016\dots & 0.010(2) & $-\frac{1}{2}$ & 1 & 1 & 0.027\dots & 0.022(3) \\[1ex]
  $\frac{1}{2}$ & 2 & 2 & 0.016\dots & 0.011(1) & $-\frac{1}{2}$ & 2 & 2 & 0.027\dots & 0.020(3) \\[1ex]
  $\frac{1}{2}$ & 3 & 3 & 0.016\dots & 0.011(1) & $-\frac{1}{2}$ & 3 & 3 & 0.027\dots & 0.022(4) \\[1ex]
  $\frac{1}{2}$ & 4 & 4 & 0.016\dots & 0.014(2) & $-\frac{1}{2}$ & 4 & 4 & 0.027\dots & 0.019(2) \\
  
\bottomrule
\end{tabular}
\begin{tablenotes}
The obtained values for $j\neq i$ were all not greater than 2.2$\cdot 10^{-3}$ (here the exact values are 0).
\end{tablenotes}
\label{table:betapolynomes_d=4_s=2}
\end{table}

\section{Experimental data}
\label{sec:exp}

Finally, we apply our methods to experimental data to demonstrate their
robustness in real-life scenarios. Therefore, we pick two distinct use
cases, where the anisotropy of a random spatial structure can be easily
related to physical (or biological) properties: a cellular structure
(here represented by voxels) and a rough surface (here represented by a
point cloud).

The first example consists of three-dimensional tomography data by
Stinville \textit{et al.}~\cite{stinville_multi-modal_2022} for a
nickel-based superalloy that forms a polycrystalline material. This
material is studied under mechanical loading. Here, we are especially
interested in the anisotropy of the single cells, which can be well
approximated by sets of positive reach. The high resolution of this data
set enables accurate estimates via our Voronoi-LSQ algorithm despite
significant size disparities between cells (which
can have volumes ranging from $\mu\text{m}^3$ to mm$^3$).

In our second example, we challenge our method and probe it under
extreme conditions; we apply it to a data set that represents black silicon, i.e., nanorough
surfaces with structural features probably almost down to the atomic scale and
thus even below the already high resolution of the atomic-force
microscopy data from Spengler \textit{et
al.}~\cite{spengler_strength_2019}. The nanorough surfaces exhibit deep
cuts and high ridges; hence, they can hardly be modeled by sets of
positive reach (since the reach would be smaller than the resolution).
Instead, these nanorough surfaces can be well
approximated via the boundary of finite unions of compact sets with
positive reach, namely spherical caps, because of the fabrication process via
an etching procedure.
We specifically want to compare our two algorithms Voronoi-FD and -LSQ
for the surface area, and demonstrate consistency with previous results
based on triangulations.

\subsection{Metallic grains in a polycrystalline material}
\label{sec:grains}

\begin{figure}[t]
  \captionsetup{ labelfont = {bf}, format = plain }
  \centering
  \includegraphics[width=\linewidth]{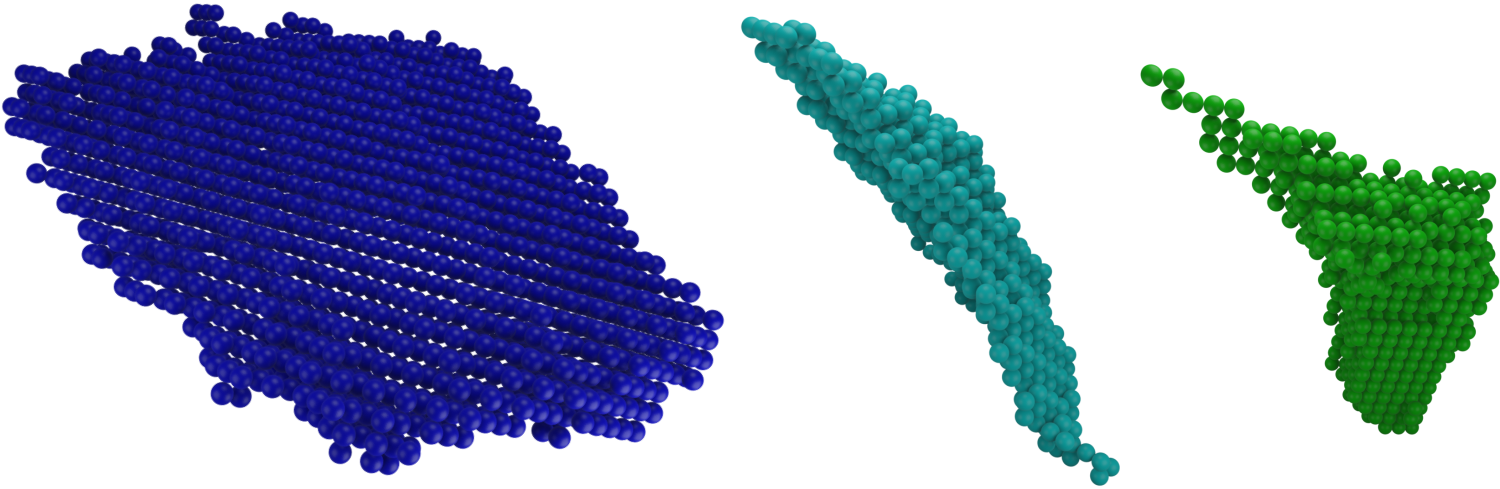}
  \caption{Examples of three (out of the 2546) metallic grains in a nickel-based superalloy
    from the experimental data set of Stinville \textit{et
    al.}~\cite{stinville_multi-modal_2022}. Each of the three grains is represented
    by a point cloud or more precisely, by balls centered at each voxel;
  the diameter of each ball equals the voxel size, i.e., 1\,$\mu$m.}
  \label{fig:grains}
\end{figure}

\begin{figure}[t]
  \captionsetup{ labelfont = {bf}, format = plain }
  \centering
  \subfloat[][]{\includegraphics[width=0.48\linewidth]{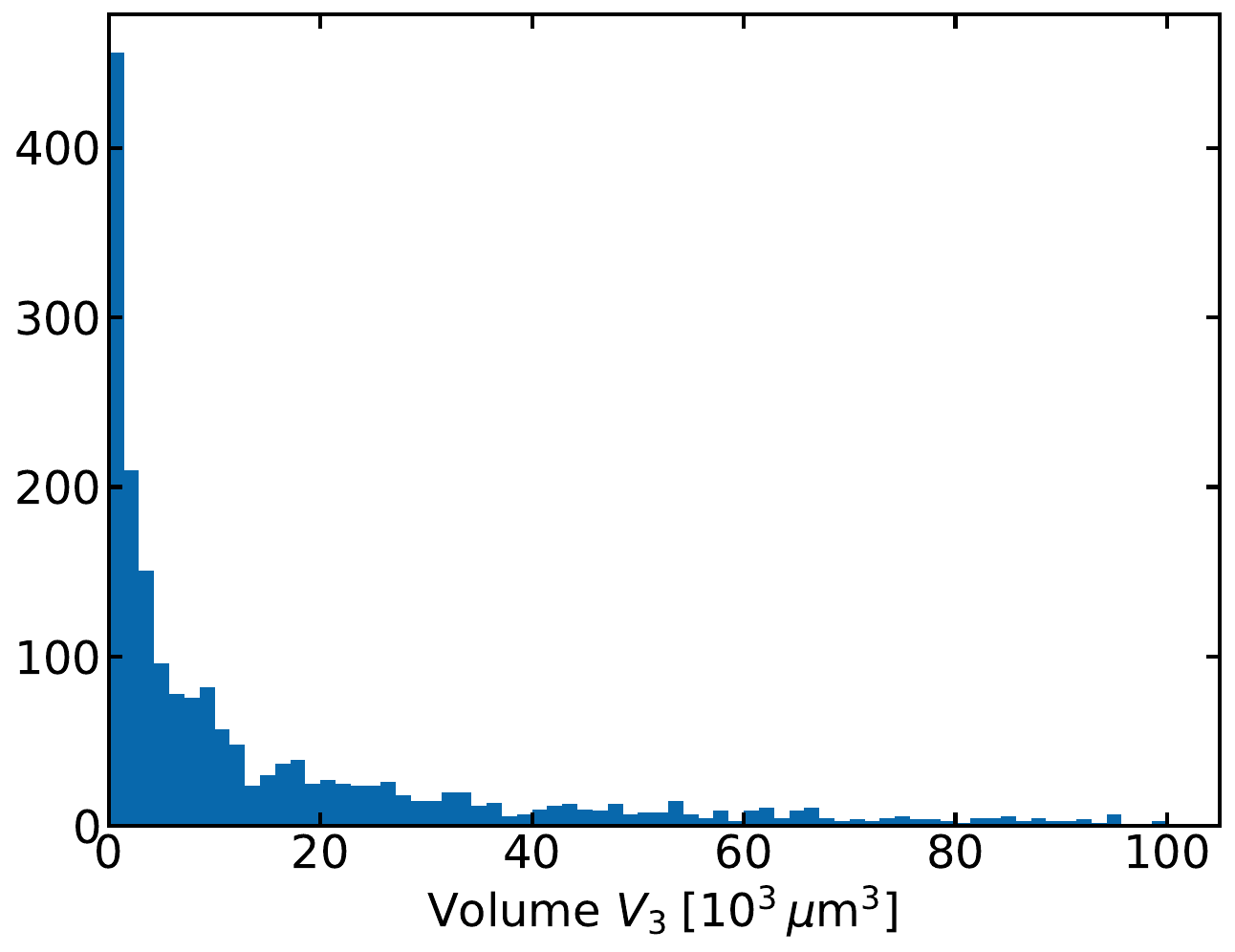}}
  \quad
  \subfloat[][]{\includegraphics[width=0.48\linewidth]{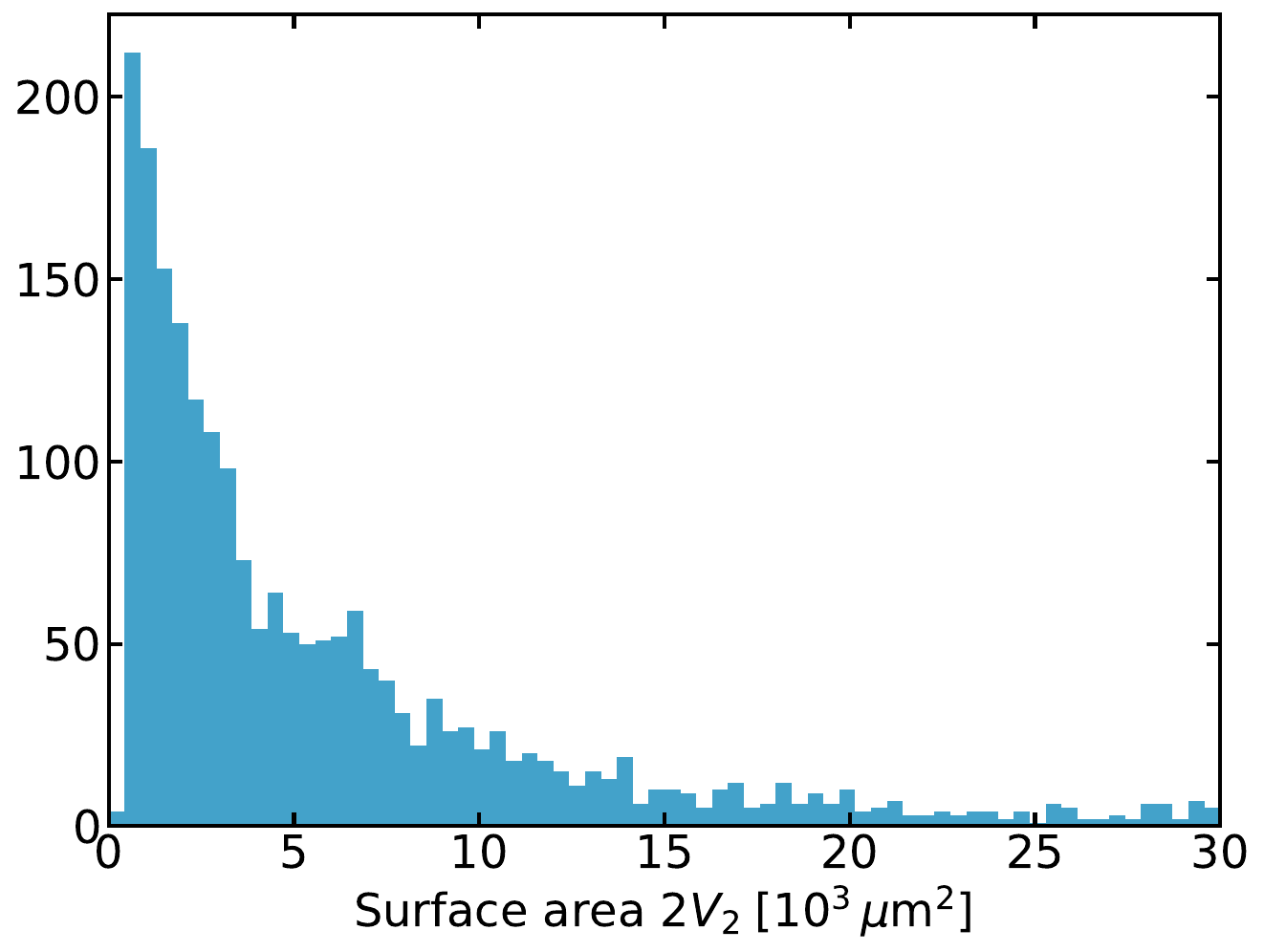}}
  \caption{Non-normalized histograms of the estimated volumes (a) and surface areas (b) of the metallic
  grains from Section~\ref{sec:grains} by using the Voronoi-LSQ algorithm. Besides a large number of small cells, we observe a few exceptionally
  large cells. For better visualization, only a range of values is shown, i.e., the most extreme cases
  are excluded here.
  The underlying experimental data of the grains is from~\cite{stinville_multi-modal_2022}.}
  \label{fig:IntrinsicVolumes}
\end{figure}

\begin{figure}[t]
  \captionsetup{ labelfont = {bf}, format = plain }
  \centering
  \subfloat[][]{\includegraphics[width=0.48\linewidth]{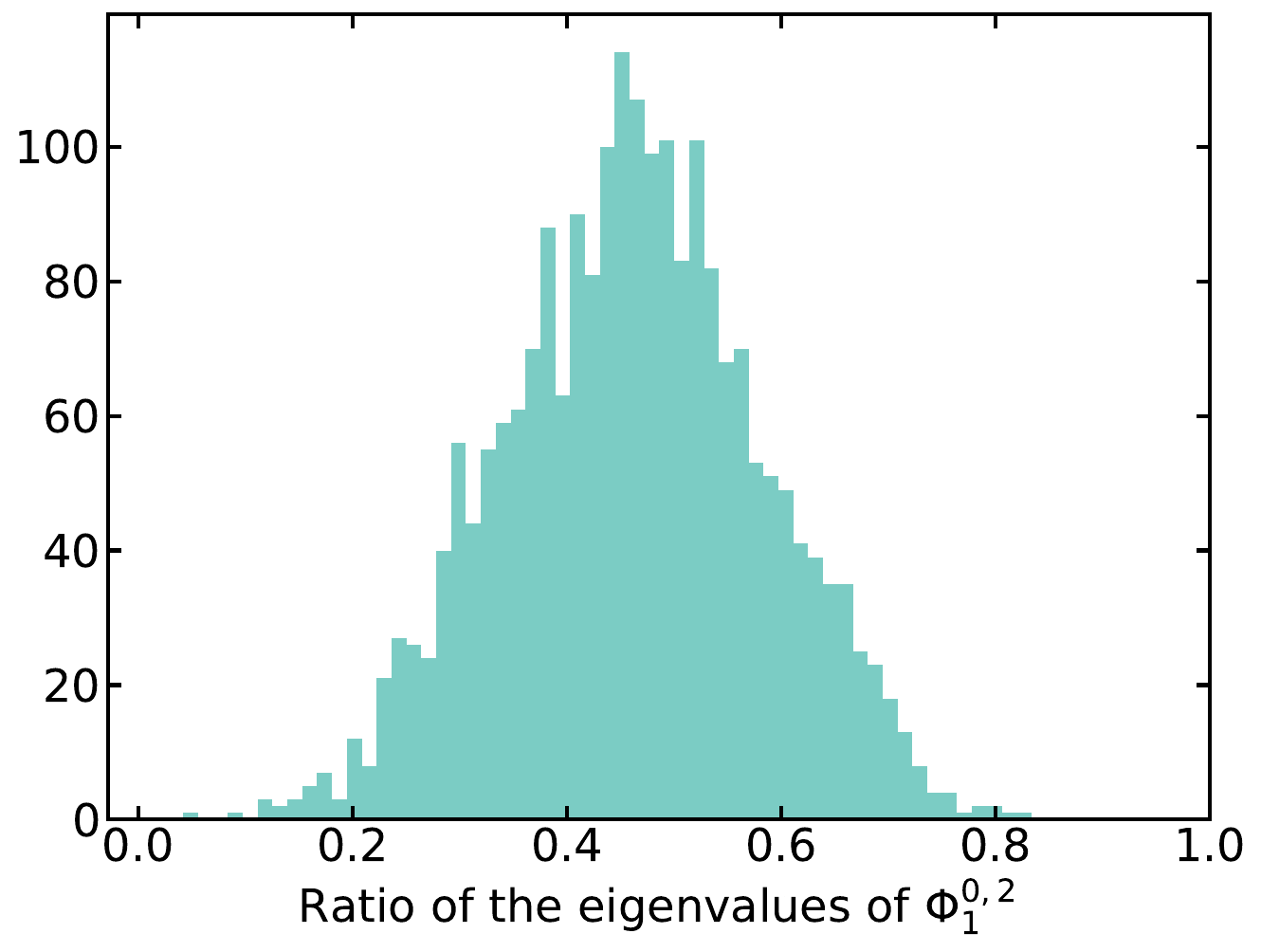}}
  \quad
  \subfloat[][]{\includegraphics[width=0.48\linewidth]{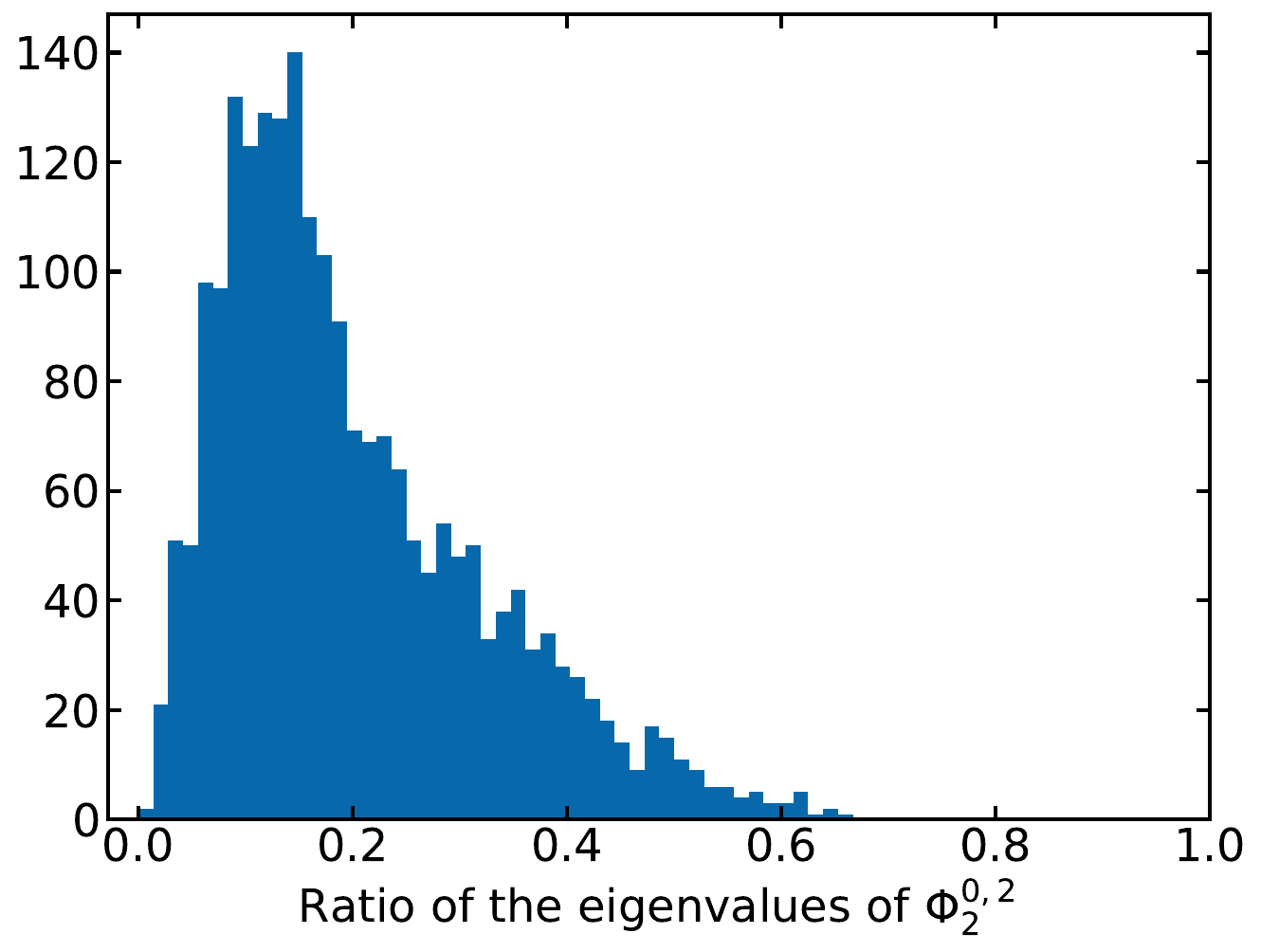}}
  \caption{Non-normalized histograms of anisotropy indices for the metallic grains from 
  Section~\ref{sec:grains} by using the Voronoi-LSQ algorithm. More specifically, the plots show histograms of the ratios between
  the smallest and largest absolute eigenvalues of the interfacial tensor $\Phi_1^{0,2}$ (a)
  and the curvature tensor $\Phi_2^{0,2}$ (b).
  The underlying experimental data of the grains is from~\cite{stinville_multi-modal_2022}.}
  \label{fig:Tensor}
\end{figure}

Stinville \textit{et al.}~\cite{stinville_multi-modal_2022} provides
three-dimensional tomography measurements of a polycrystalline
material. The data set was collected to investigate plastic
deformations during mechanical loading. Such insights into
structure-property relations can be used for the prediction and design
of mechanical properties based on the material's microstructure.
Obviously, accurate predictions require sensitive, yet robust structural
characteristics.

A natural geometric question for cellular materials (especially if
subject to mechanical loading) is how to quantify anisotropy of single
(deformed) cells, both the degree of anisotropy and its preferred
orientations. 
As indicated by Fig.~\ref{fig:intro_schematic}, the Minkowski tensors
$\Phi_{d-1}^{0,s}$ and $\Phi_{d-2}^{0,s}$ quantify anisotropy with
respect to different geometric
properties, namely surface area and mean curvature, respectively.

Here, we want to compare these two different measures of anisotropy
by applying our Voronoi-LSQ algorithm to the single metallic
grains from \cite{stinville_multi-modal_2022}. The data set contains,
among others, a voxelized reconstruction of 2546 individual
three-dimensional grains; see Figure~\ref{fig:grains}, for three
examples showing the voxel centers as a point cloud. The resolution
of the grid is 1\,$\mu$m. To exclude numerical artifacts, we apply a
data cut and only consider grains with at least 500 voxel centers,
which results in a data set with 2180 grains. Since those grains can
be considered to be sets of positive reach, we again restrict our
focus to the Voronoi-LSQ algorithm, as in Section~\ref{sec:beta}.
For each grain, we took the average over 10 renditions. We varied
$R_n$ between $4\cdot R_1$ and $24\cdot R_1$. The results from
Figure~\ref{fig:IntrinsicVolumes} and \ref{fig:Tensor} refer to the
choice $R_n=24\cdot R_1$.

We begin by demonstrating that our Voronoi-LSQ algorithm provides
robust estimates of the volume and surface area of the cells, which vary
by several orders of magnitude. Figure~\ref{fig:IntrinsicVolumes}
shows histograms of the estimated volumes and surface areas,
respectively. The distributions peak at very small values, but they
exhibit comparatively strong tails. Interestingly, the tails are more
pronounced for the volumes than the surface areas.

We then quantify the anisotropy of the cells via our estimates of
the tensors $\Phi_1^{0,2}$ and $\Phi_2^{0,2}$. Figure~\ref{fig:Tensor}
shows a histogram of the ratios between the smallest and biggest
absolute eigenvalues of these tensors. These ratios represent
scalar indices of anisotropy, which intuitively quantify the degree of
anisotropy with respect to the corresponding Minkowski tensor, and which
have already been successfully applied in
physics~\cite{schroder-turk_tensorial_2010,
schroder-turk_minkowski_2011, schroder-turk_minkowski_2013,
klatt_cell_2017, klatt_mean-intercept_2017}.

For the metallic grains, we detect via these indices a higher degree
of anisotropy with respect to the curvature than the surface area. A detailed
physical interpretation of these geometric findings is beyond the
scope of this paper, but our results underscore 
the need to choose a measure of anisotropy that reflects the
geometric property which is physically relevant.

\subsection{Nanorough surfaces}
\label{sec:nano}

Next, we apply our algorithms to a distinctly different experimental
data set that represents a nanorough surface known as ``black silicon'', from
\cite{spengler_strength_2019}, measured via atomic force
microscopy~(AFM). The surface is represented as a point cloud that is
formed by the $z$-values measured on a grid of $x$- and $y$-values.

Here, we want to estimate the surface area of the black silicon. Note
that the point cloud represents the interface as a two-dimensional
sheet. Therefore, at almost every point of the interface there are
two normal vectors with opposite orientation, i.e., facing `up' and
`down'. Hence, $\Phi_2^{0,0}$ of the two-dimensional sheet is twice the
value of $\Phi_2^{0,0}$ of the black silicon within the observation
window. The latter equals half of the surface area of the black silicon
within the observation window. Hence, the estimate of $\Phi_2^{0,0}$ of
the Voronoi-LSQ algorithm can directly be used as an estimate of the
surface area of the black silicon.

The roughness at the nanoscale was obtained by a
specific etching protocol; see \cite{spengler_strength_2019} for further
details. Both the protocol and three-dimensional visualizations of
the data sets (see Figure~\ref{fig:nanosamples}) suggest that the
nanorough surface can be described (or at least approximated) by a
finite union of compact sets with positive reach, but it is an
extreme test case for our methods because of the sharp features in the
rough surface. In fact, the nanorough surface is a difficult example
for any estimator of Minkowski tensors because there are features
probably almost down to the atomic scale and thus below the already high
resolution of about 6\,nm in $x$- and $y$-direction.

\begin{figure}[t]
  \captionsetup{ labelfont = {bf}, format = plain }
  \centering
  \subfloat[][]{\includegraphics[width=0.32\linewidth]{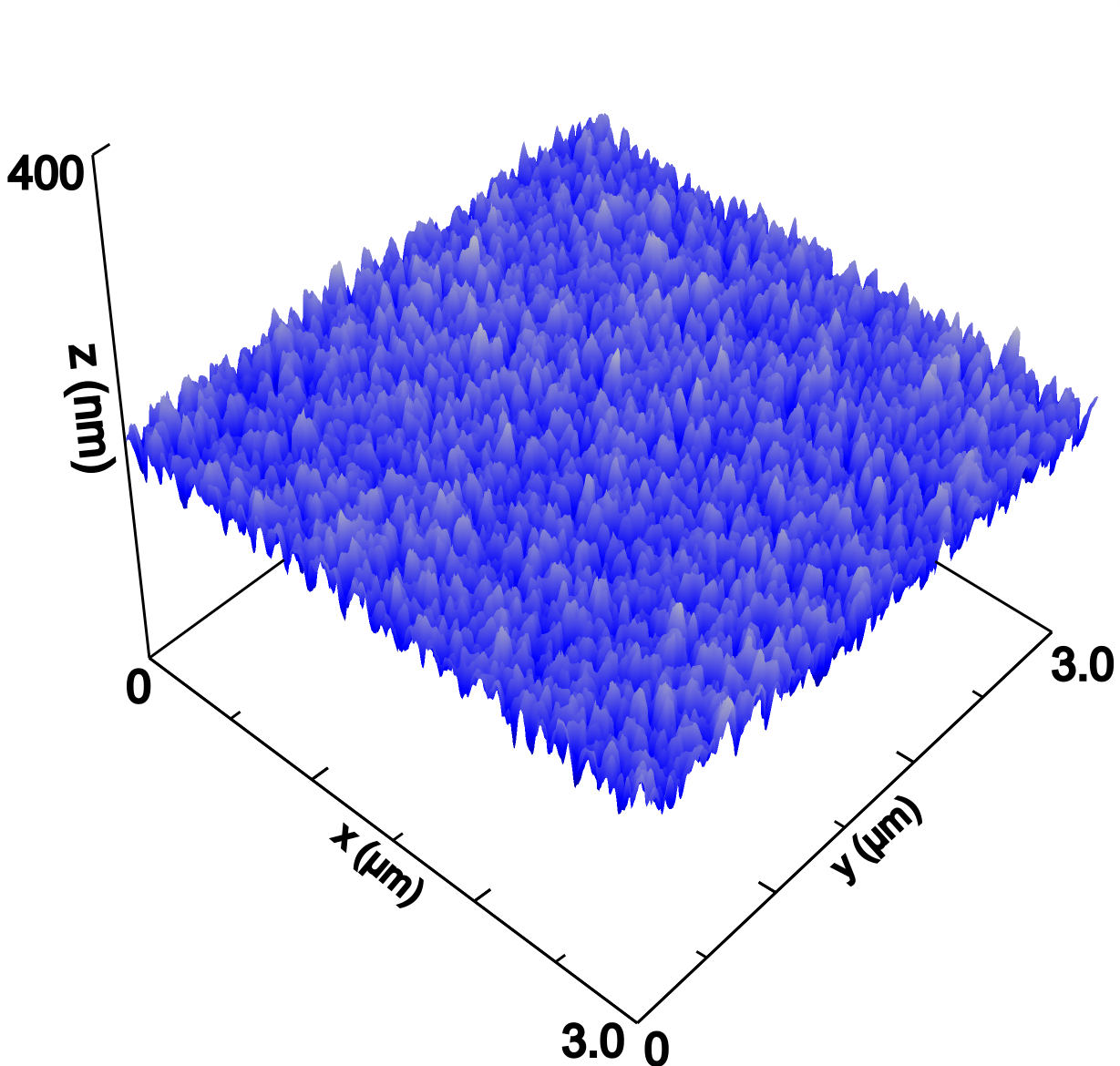}}
  \hfill
  \subfloat[][]{\includegraphics[width=0.32\linewidth]{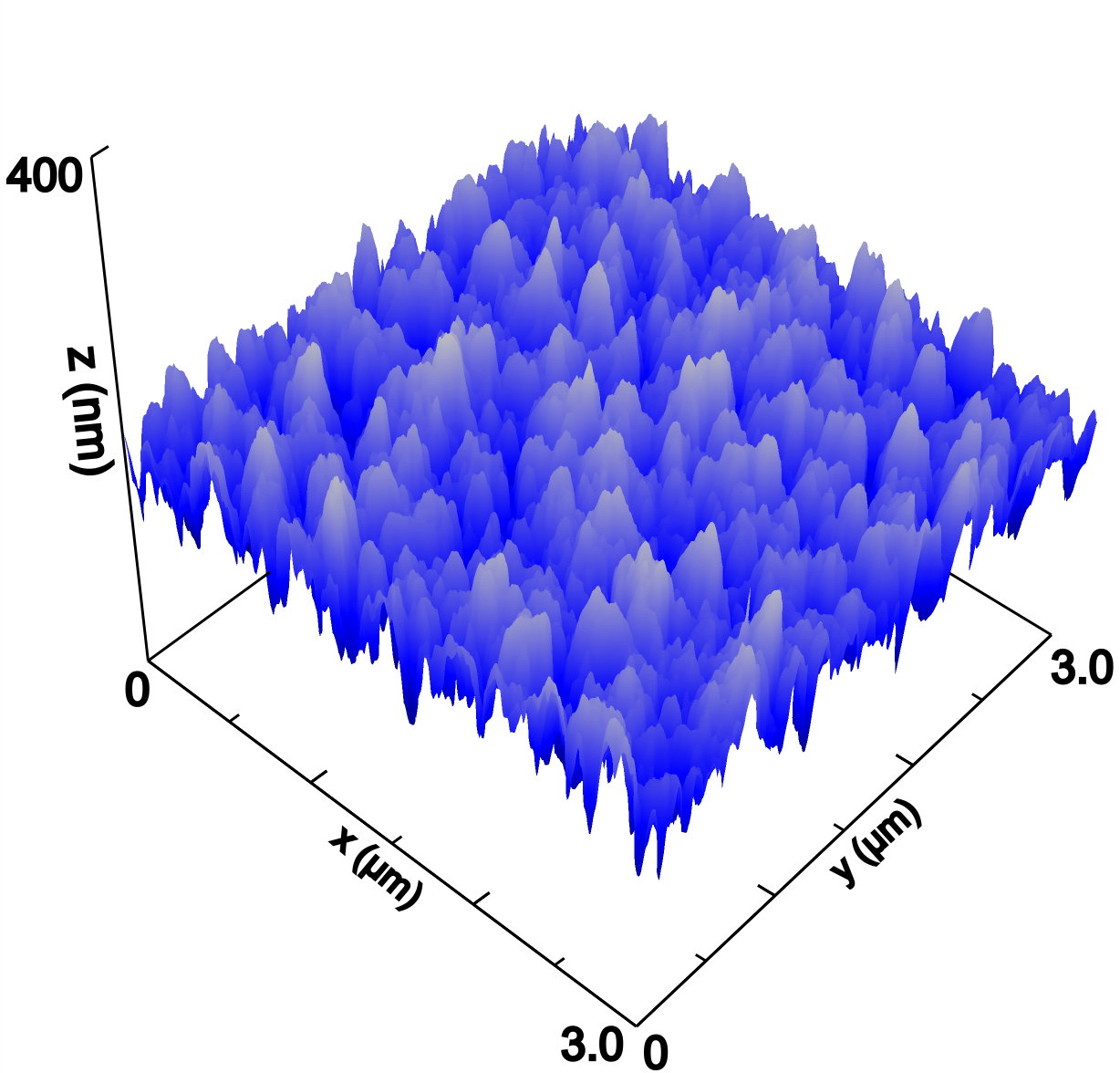}}
  \hfill
  \subfloat[][]{\includegraphics[width=0.32\linewidth]{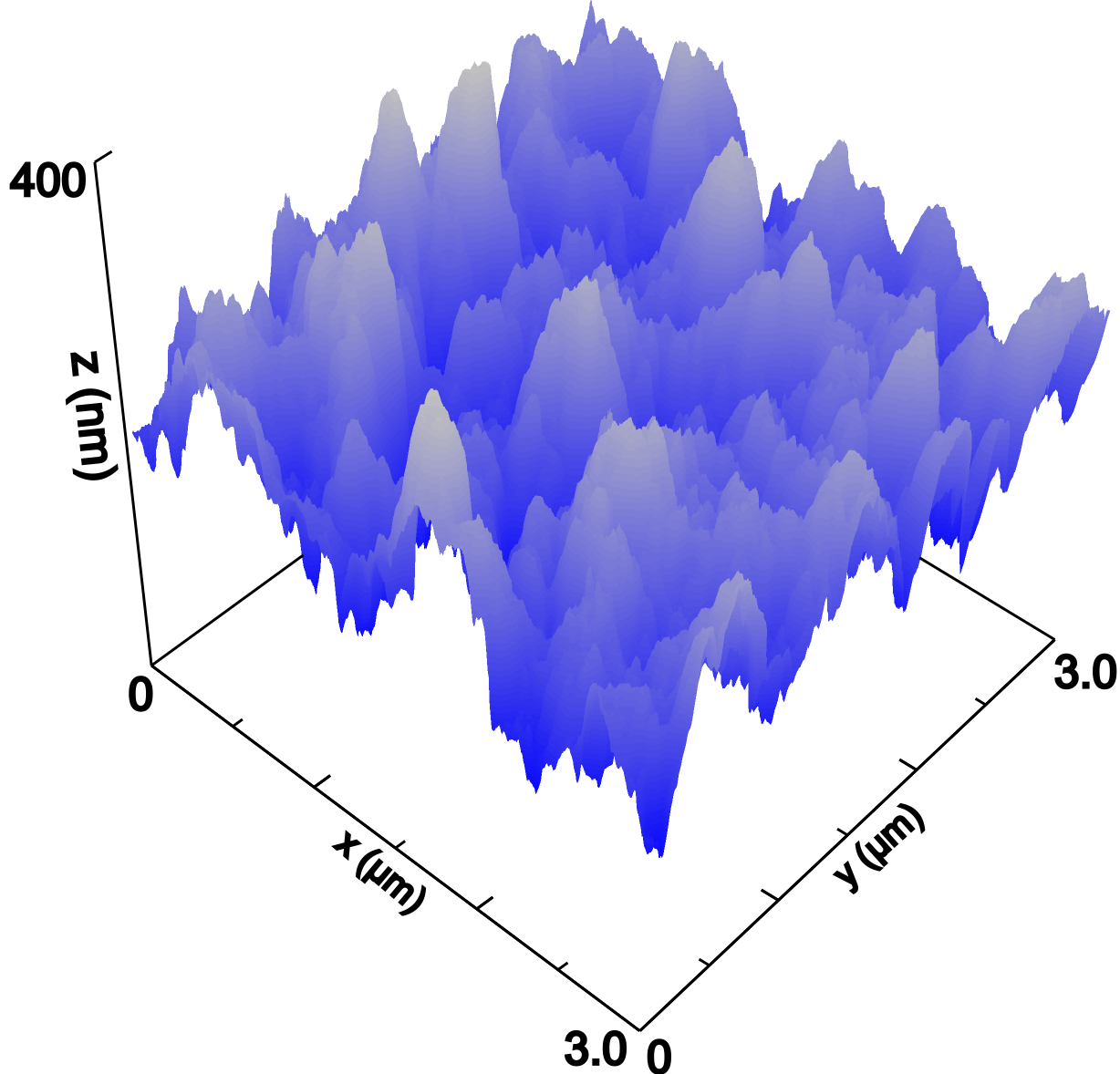}}
  \caption{Examples of nanorough surfaces from Figure~\ref{fig:Nanoscale} with three different 
  degrees of roughness, i.e., an RMS value of (a) 7\,nm, (b) 24\,nm, and (c) 35\,nm. The data 
  is taken from \cite{spengler_strength_2019}; the figures have been created by Jens Uwe Neurohr.}
  \label{fig:nanosamples}
\end{figure}

We cannot eliminate pixelation errors merely by increasing the
resolution to an arbitrarily high level. Instead, the data requires
accurate estimates of Minkowski tensors at a reasonable resolution.
Here we want to compare our algorithms to previous results from a
triangulation of the surface and show consistency between the results,
as well as a robustness of our algorithm for a reasonable range of
values of $R_n$.

\begin{figure}[p]
  \captionsetup{labelfont={bf}, format=plain}
  \centering
  \subfloat{\includegraphics[width=0.95\linewidth]{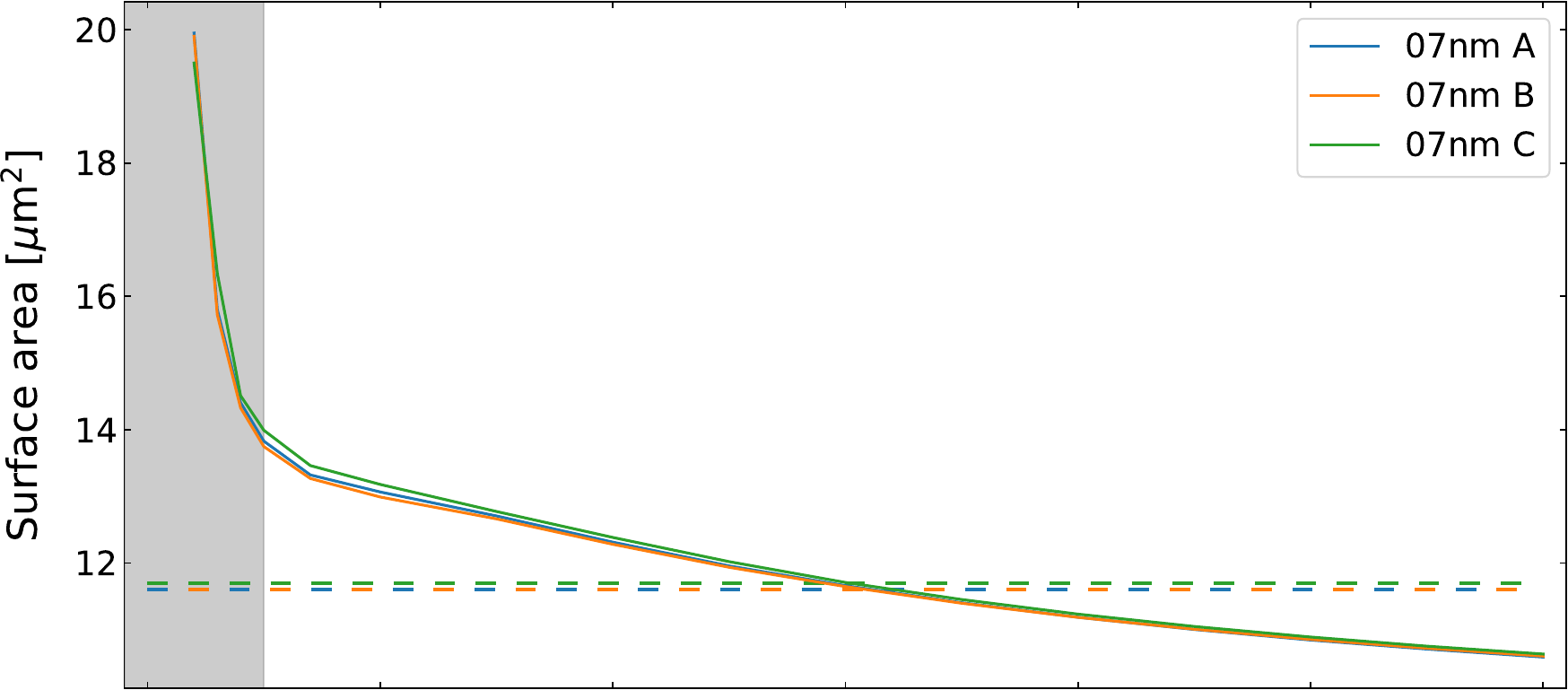}}\\[1ex]
  \subfloat{\includegraphics[width=0.95\linewidth]{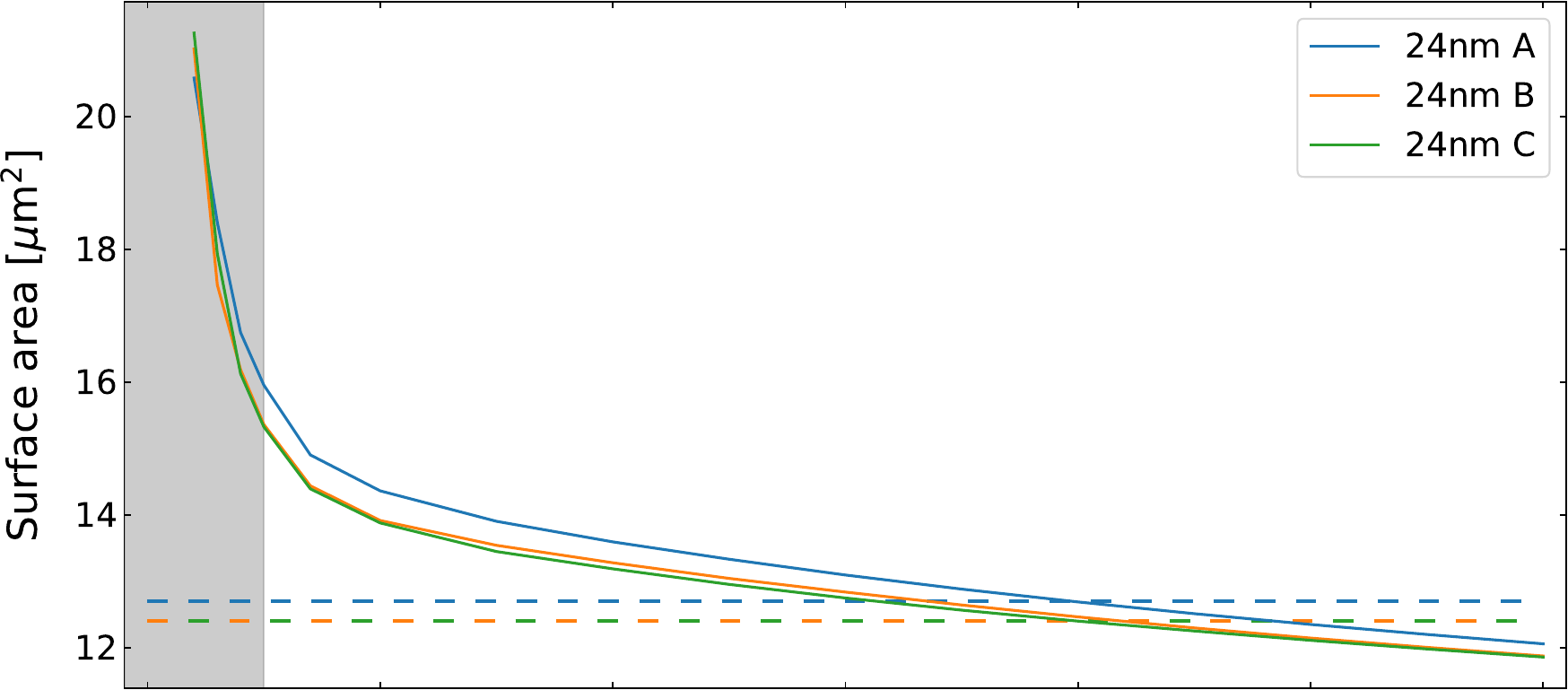}}\\[1ex]
  \subfloat{\includegraphics[width=0.95\linewidth]{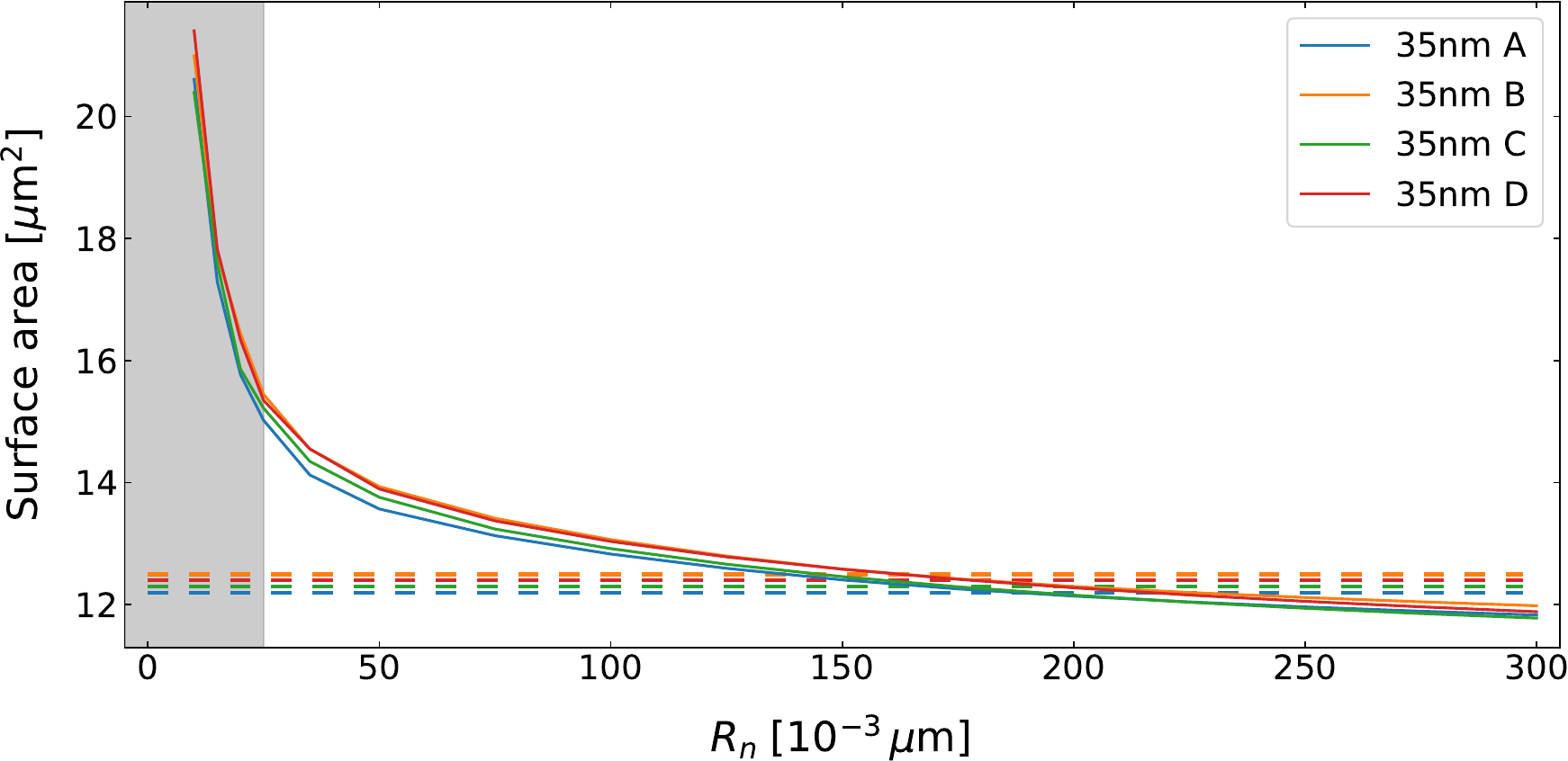}}
  \caption{Estimated surface areas for the different datasets as a
  function of $R_n$. The dashed lines indicate the values obtained via
  triangulation. The gray area corresponds to a choice of $R_n$ that
  violates the rule \eqref{Faustregel}.}
  \label{fig:Nanoscale}
\end{figure}

Here, we analyze ten samples in total with three different values
of the root mean square (RMS) of height distribution (7\,nm, 24\,nm, and
35\,nm). Each sample is measured within a scan window of size
$3\,\mu\text{m}\times3\,\mu\text{m}$ and a resolution of $512\times 512$
pixels, i.e., each surface is represented by a grid of three-dimensional
points, where the $z$-coordinate represents the height of the surface at
the corresponding $(x,y)$ position. 

Figure~\ref{fig:Nanoscale} shows the estimates of the surface area via
the Voronoi-LSQ algorithm as a function of $R_n$. A proper choice of
$R_n$ has to balance robustness (via a large range of radii) and
accuracy (via a small range of radii that resolves fine details of the
nanorough structure). From our test cases, we derived as a rule of thumb
that $R_n$ should be at least four times the average nearest neighbor
distance in the data; see \eqref{Faustregel}. This rule of thumb is indicated in
Fig.~\ref{fig:Nanoscale} by the gray-shaded area.

For larger values of $R_n$, we find consistent results, i.e., no
strong dependence on $R_n$ for a considerable range of
values, roughly from $25$ to $150\,$nm. For this range, our
estimated values are slightly larger than those obtained from a
triangulation of the surface as in \cite{spengler_strength_2019} using
methods from \cite{schroder-turk_minkowski_2013}. Considering how rough
these surfaces are and hence challenging for our method of
extrapolation, we observe a relatively good agreement of our methods
with the results from the triangulated data. Since the
triangulation-based method tends to underestimate the surface area, our
results are consistent within the accuracy of the data.

Our estimates via the Voronoi-LSQ algorithm are also consistent with
those from the Voronoi-FD algorithm shown in
Table~\ref{tab:nanorough_FD}; however, the results from the Voronoi-FD algorithm depend more strongly on the
choice of the radius $\varepsilon$. If the radius is too large, a boundary
correction is required for a two-dimension sheet and more importantly
the parallel set hides small features of the rough surface. If the
radius is too small, artifical `gaps' and `bumps' can lead to an
overestimation of the surface area. Here, we choose a radius of twice
the average nearest neighbor distance in the data. New data with an even higher
resolution will be needed to compare the results of the algorithms more precisely.

\begin{table}[t!]
\centering
\captionsetup{
  labelfont = {bf},
  format = plain,
  belowskip = 1ex,
  width = \textwidth
}
\caption{Estimated values of the surface area in $\mathrm{\mu m}^2$ via the Voronoi-FD algorithm in comparison with triangulation results based on data from \cite{spengler_strength_2019}.
We choose $\varepsilon=2\cdot a$ and $a$ as the average nearest neighbor distance in the data.
We used several samples with the same RMS value.
The underlying experimental data of the nanorough surfaces is from \cite{spengler_strength_2019}.}
\begin{tabular}{c|c|c|c|c|c}
  \toprule
  RMS & \textbf{Triangulation} & \textbf{Voronoi-FD} & RMS & \textbf{Triangulation} & \textbf{Voronoi-FD} \\
  \midrule
7 & 11.6 & 12.6 & 24 & 12.7 & 13.9  \\[1ex]
 & 11.6 & 12.6 & & 12.4 & 13.5  \\[1ex]
 & 11.7 & 12.6 & & 12.4 & 13.5 \\[1ex]
35 & 12.2 & 13.4 & 35 & 12.3 & 13.5  \\[1ex]
 & 12.5 & 13.6 & & 12.4 & 13.6 \\
\bottomrule
\end{tabular}
\label{tab:nanorough_FD}
\end{table}

\section*{Acknowledgments}
We thank Karin Jacobs and Christian Spengler for providing their data of nanorough
surfaces, and Jens Uwe Neurohr for Figure~\ref{fig:nanosamples}.
This work was supported in part by the Deutsche Forschungsgemeinschaft
(DFG, German Research Foundation) through the SPP 2265, under grant
numbers HU 1874/5-1, ME 1361/16-1, KL 3391/2-2, WI 5527/1-1, and LO 418/25-1, as well as by the
Helmholtz Association and the DLR via the Helmholtz Young Investigator
Group ``DataMat''.
The authors gratefully acknowledge the scientific support and HPC
resources provided by the Erlangen National High Performance Computing
Center (NHR@FAU) of the Friedrich-Alexander-Universität
Erlangen-Nürnberg (FAU). The hardware is funded by the German Research
Foundation (DFG).

\bibliographystyle{my-plain}
\bibliography{voromink}

@book{adler_fractured_2013,
  ids = {AdlerEtAl2012},
  title = {Fractured Porous Media},
  author = {Adler, Pierre M. and Thovert, Jean-Fran{\c c}ois and Mourzenko, Valeri V.},
  year = {2013},
  edition = {first},
  pages = {184},
  publisher = {{Oxford Univ. Press}},
  address = {{Oxford}},
  comment_url = {http://www.oxfordscholarship.com/view/10.1093/acprof:oso/9780199666515.001.0001/acprof-9780199666515},
  langid = {english}
}

@book{adler_random_2007,
  ids = {AdlerTaylor2007},
  title = {Random Fields and Geometry},
  author = {Adler, Robert J. and Taylor, Jonathan E.},
  year = {2007},
  series = {Springer Monographs in Mathematics},
  publisher = {{Springer}},
  address = {{New York}},
  langid = {english},
  lccn = {QA274.45 .A345 2007}
}

@article {Alesker_geomdedicata,
   author = {Alesker, S.},
   title = {Description of continuous isometry covariant valuations on convex sets},
   journal = {Geom. Dedicata},
   fjournal = {Geometriae Dedicata},
   volume = {74},
   year = {1999},
   number = {3},
   pages = {241--248},
   issn = {0046-5755,1572-9168},
   mrclass = {52B45 (52A38)},
   mrnumber = {1669363},
   mrreviewer = {Uwe\ Schnell},
   doi = {10.1023/A:1005035232264},
   url = {https://doi.org/10.1023/A:1005035232264},
}

@article {Alesker_annalsofmathematics,
   author = {Alesker, S.},
   title = {Continuous rotation invariant valuations on convex sets},
   journal = {Ann. of Math. (2)},
   fjournal = {Annals of Mathematics. Second Series},
   volume = {149},
   year = {1999},
   number = {3},
   pages = {977--1005},
   issn = {0003-486X,1939-8980},
   mrclass = {52A39 (52B45)},
   mrnumber = {1709308},
   mrreviewer = {P.\ McMullen},
   doi = {10.2307/121078},
   url = {https://doi.org/10.2307/121078},
}

@article{barbosa_integral-geometry_2014,
  ids = {BarbosaEtAl:2014},
  title = {Integral-Geometry Characterization of Photobiomodulation Effects on Retinal Vessel Morphology},
  author = {Barbosa, Marconi and Natoli, Riccardo and Valter, Kriztina and Provis, Jan and Maddess, Ted},
  year = {2014},
  journal = {Biomed. Opt. Express, BOE},
  volume = {5},
  pages = {2317--2332},
  publisher = {{OSA}},
  doi = {10.1364/BOE.5.002317},
  urldate = {2018-08-06},
  comment_url = {http://www.opticsinfobase.org/boe/abstract.cfm?URI=boe-5-7-2317},
  langid = {english}
}

@article{barbosa_novel_2019,
  title = {Novel Morphometric Analysis of Higher Order Structure of Human Radial Peri-Papillary Capillaries: Relevance to Retinal Perfusion Efficiency and Age},
  author = {Barbosa, Marconi and Maddess, Ted and Ahn, Samyoul and {Chan-Ling}, Tailoi},
  year = {2019},
  journal = {Sci Rep},
  volume = {9},
  pages = {1--16},
  doi = {10.1038/s41598-019-49443-z},
  urldate = {2019-11-21},
  langid = {english}
}

@article{beisbart_extended_2006,
  ids = {BeisbartEtAl:2006},
  title = {Extended Morphometric Analysis of Neuronal Cells with {{Minkowski}} Valuations},
  author = {Beisbart, C. and Barbosa, M. S. and Wagner, H. and Costa, L. da F.},
  year = {2006},
  journal = {Eur. Phys. J. B},
  volume = {52},
  pages = {531--546},
  publisher = {{EDP Sciences}},
  doi = {10.1140/epjb/e2006-00328-1},
  urldate = {2018-08-06},
  comment_url = {http://dx.doi.org/10.1140/epjb/e2006-00328-1},
  langid = {english}
}

@article{bobel_kinetics_2016,
  ids = {MT_plasma,bobel_kinetics_2016-1},
  title = {Kinetics of Fluid Demixing in Complex Plasmas: {{Domain}} Growth Analysis Using {{Minkowski}} Tensors},
  author = {B{\"o}bel, A. and R{\"a}th, C.},
  year = {2016},
  journal = {Phys. Rev. E},
  volume = {94},
  pages = {013201},
  doi = {10.1103/PhysRevE.94.013201},
  urldate = {2016-09-30}
}

@article{branch,
author = {Branch, Mary Ann and Coleman, Thomas F. and Li, Yuying},
title = {A Subspace, Interior, and Conjugate Gradient Method for Large-Scale Bound-Constrained Minimization Problems},
journal = {SIAM Journal on Scientific Computing},
volume = {21},
number = {1},
pages = {1-23},
year = {1999},
doi = {10.1137/S1064827595289108},
}

@article{callens_local_2021,
  title = {The Local and Global Geometry of Trabecular Bone},
  author = {Callens, Sebastien J. P. and {Tourolle n{\'e} Betts}, Duncan C. and M{\"u}ller, Ralph and Zadpoor, Amir A.},
  year = {2021},
  journal = {Acta Biomaterialia},
  volume = {130},
  pages = {343--361},
  doi = {10.1016/j.actbio.2021.06.013},
  urldate = {2023-07-25},
  langid = {english}
}

@book{chiu_stochastic_2013,
  ids = {ChiuEtAl2013},
  title = {Stochastic {{Geometry}} and Its {{Applications}}},
  author = {Chiu, S. N. and Stoyan, D. and Kendall, W. S. and Mecke, J.},
  year = {2013},
  edition = {Third},
  publisher = {{Wiley}},
  address = {{Chichester}}
}

@phdthesis{christensen,
  author       = {Christensen, S. T.},
  title        = {Reconstruction of Topology and Geometry from Digitisations},
  school       = {Aarhus University},
  year         = {2016},
  type         = {Dissertation}
}

@article{cohn_new_2003,
  title = {New Upper Bounds on Sphere Packings {{I}}},
  author = {Cohn, Henry and Elkies, Noam},
  year = {2003},
  journal = {Ann. Math.},
  volume = {157},
  pages = {689--714},
  doi = {10.4007/annals.2003.157.689},
  urldate = {2024-02-29},
  langid = {english}
}

@article{collischon_tracking_2021,
  title = {Tracking down the Origin of Superbubbles and Supergiant Shells in the {{Magellanic Clouds}} with {{Minkowski}} Tensor Analysis},
  author = {Collischon, Caroline and Sasaki, Manami and Mecke, Klaus and Points, Sean D. and Klatt, Michael A.},
  year = {2021},
  journal = {A\&A},
  volume = {653},
  pages = {A16},
  publisher = {{EDP Sciences}},
  doi = {10.1051/0004-6361/202040153},
  urldate = {2023-05-09},
  langid = {english}
}

@article{ebner2018,
  title = {Goodness-of-Fit Tests for Complete Spatial Randomness Based on {{Minkowski}} Functionals of Binary Images},
  author = {Ebner, Bruno and Henze, Norbert and Klatt, Michael A. and Mecke, Klaus},
  year = {2018},
  journal = {Electron J. Stat.},
  volume = {12},
  pages = {2873--2904},
  doi = {10.1214/18-EJS1467},
  comment_abstract = {We propose a class of goodness-of-fit tests for complete spatial randomness (CSR). In contrast to standard tests, our procedure utilizes a transformation of the data to a binary image, which is then characterized by geometric functionals. Under a suitable limiting regime, we derive the asymptotic distribution of the test statistics under the null hypothesis and almost sure limits under certain alternatives. The new tests are computationally efficient, and simulations show that they are strong competitors to other tests of CSR. The tests are applied to a real data set in gamma-ray astronomy, and immediate extensions are presented to encourage further work.},
  comment_file = {2018\textsubscript{E}bner\textsubscript{e}t\textsubscript{a}l\textsubscript{G}oodness-of-fit\textsubscript{t}ests\textsubscript{f}or\textsubscript{c}omplete\textsubscript{s}patial\textsubscript{r}andomness\textsubscript{b}ased\textsubscript{o}n\textsubscript{M}inkowski.pdf:/home/mklatt/.mozilla/firefox/0ne35d4v.default-1468252463280/zotero/storage/D8M9SFPS/2018\textsubscript{E}bner\textsubscript{e}t\textsubscript{a}l\textsubscript{G}oodness-of-fit\textsubscript{t}ests\textsubscript{f}or\textsubscript{c}omplete\textsubscript{s}patial\textsubscript{r}andomness\textsubscript{b}ased\textsubscript{o}n\textsubscript{M}inkowski.pdf:application/pdf},
  comment_issn = {1935-7524},
  comment_language = {en},
  comment_month = {09},
  comment_url = {https://projecteuclid.org/euclid.ejs/1537257628},
  comment_urldate = {2018-09-19}
}

@article{ernesti_characterizing_2020,
author = {Ernesti, Felix and Schneider, Matti and Winter, Steffen and Hug, Daniel and Last, Günter and Böhlke, Thomas},
title = {Characterizing digital microstructures by the {M}inkowski-based quadratic normal tensor},
journal = {Mathematical Methods in the Applied Sciences},
volume = {46},
number = {1},
pages = {961-985},
doi = {https://doi.org/10.1002/mma.8560},
url = {https://onlinelibrary.wiley.com/doi/abs/10.1002/mma.8560},
eprint = {https://onlinelibrary.wiley.com/doi/pdf/10.1002/mma.8560},
year = {2023}
}

@article{gott_topology_1990,
  title = {Topology of Microwave Background Fluctuations - {{Theory}}},
  author = {Gott III, J. Richard and Park, Changbom and Juszkiewicz, Roman and Bies, William E. and Bennett, David P. and Bouchet, Francois R. and Stebbins, Albert},
  year = {1990},
  journal = {ApJ},
  volume = {352},
  pages = {1},
  doi = {10.1086/168511},
  urldate = {2024-02-29},
  langid = {english}
}

@book{hansen_theory_2013,
  title = {Theory of {{Simple Liquids}}: {{With Applications}} to {{Soft Matter}}},
  author = {Hansen, Jean-Pierre and McDonald, Ian R.},
  year = {2013},
  edition = {4th},
  publisher = {{Academic Press}},
  address = {{Amsterdam}},
  urldate = {2019-08-04}
}

@article{hug_voronoi-based_2017,
  ids = {hug_voronoi-based_2017-1},
  title = {Voronoi-{{Based Estimation}} of {{Minkowski Tensors}} from {Finite Point Samples}},
  author = {Hug, Daniel and Kiderlen, Markus and Svane, Anne Marie},
  year = {2017},
  journal = {Discrete Comput. Geom.},
  volume = {57},
  pages = {545--570},
  doi = {10.1007/s00454-016-9851-x},
  urldate = {2020-05-29}
}

@incollection{jensen_valuations_2017,
  title = {Valuations in {{Image Analysis}}},
  booktitle = {Tensor {{Valuations}} and {{Their Applications}} in {{Stochastic Geometry}} and {{Imaging}}},
  author = {Svane, Anne Marie},
  editor = {Jensen, Eva B. Vedel and Kiderlen, Markus},
  year = {2017},
  volume = {2177},
  pages = {435--454},
  publisher = {{Springer International Publishing}},
  address = {{Cham}},
  doi = {10.1007/978-3-319-51951-7_15},
  urldate = {2024-02-29}
}

@article{joby_search_2019,
  title = {Search for Anomalous Alignments of Structures in {{Planck}} Data Using {{Minkowski Tensors}}},
  author = {Joby, P.K. and Chingangbam, Pravabati and Ghosh, Tuhin and Ganesan, Vidhya and Ravikumar, C.D.},
  year = {2019},
  journal = {J. Cosmol. Astropart. Phys.},
  volume = {2019},
  pages = {009--009},
  doi = {10.1088/1475-7516/2019/01/009},
  urldate = {2024-02-29}
}

@article{KABLUCHKO2020107333,
title = {Beta polytopes and {P}oisson polyhedra: f-vectors and angles},
journal = {Advances in Mathematics},
volume = {374},
pages = {107333},
year = {2020},
issn = {0001-8708},
doi = {https://doi.org/10.1016/j.aim.2020.107333},
url = {https://www.sciencedirect.com/science/article/pii/S0001870820303613},
author = {Zakhar Kabluchko and Christoph Thäle and Dmitry Zaporozhets}
}

@article{klatt_anisotropy_2017,
  title = {Anisotropy in Finite Continuum Percolation: Threshold Estimation by {{Minkowski}} Functionals},
  author = {Klatt, Michael A and {Schr{\"o}der-Turk}, Gerd E and Mecke, Klaus},
  year = {2017},
  journal = {J. Stat. Mech. Theor. Exp.},
  volume = {2017},
  pages = {023302},
  doi = {10.1088/1742-5468/aa5a19},
  urldate = {2017-02-20}
}

@incollection{klatt_cell_2017,
  title = {Cell {{Shape Analysis}} of {{Random Tessellations Based}} on {{Minkowski Tensors}}},
  booktitle = {Tensor {{Valuations}} and {{Their Applications}} in {{Stochastic Geometry}} and {{Imaging}}},
  author = {Klatt, Michael A. and Last, G{\"u}nter and Mecke, Klaus and Redenbach, Claudia and Schaller, Fabian M. and {Schr{\"o}der-Turk}, Gerd E.},
  editor = {Vedel Jensen, Eva B. and Kiderlen, Markus},
  year = {2017},
  series = {Lecture {{Notes}} in {{Mathematics}}},
  volume = {2177},
  pages = {385--421},
  publisher = {{Springer International Publishing}},
  address = {{Cham}},
  doi = {10.1007/978-3-319-51951-7_13},
  urldate = {2017-06-26},
  langid = {english}
}

@article{klatt_characterization_2022,
  title = {Characterization of Anisotropic {{Gaussian}} Random Fields by {{Minkowski}} Tensors},
  author = {Klatt, Michael Andreas and H{\"o}rmann, Max and Mecke, Klaus},
  year = {2022},
  journal = {J. Stat. Mech.},
  volume = {2022},
  pages = {043301},
  publisher = {{IOP Publishing}},
  doi = {10.1088/1742-5468/ac5dc1},
  urldate = {2022-04-21},
  langid = {english}
}

@article{klatt_detecting_2020,
  title = {Detecting Structured Sources in Noisy Images via {{Minkowski}} Maps},
  author = {Klatt, Michael A. and Mecke, Klaus},
  year = {2020},
  journal = {EPL},
  volume = {128},
  pages = {60001},
  publisher = {{EDP Sciences, IOP Publishing and Societ{\`a} Italiana di Fisica}},
  doi = {10.1209/0295-5075/128/60001},
  urldate = {2023-06-20},
  langid = {english}
}

@article{klatt_mean-intercept_2017,
  title = {Mean-Intercept Anisotropy Analysis of Porous Media. {{II}}. {{Conceptual}} Shortcomings of the {{MIL}} Tensor Definition and {{Minkowski}} Tensors as an Alternative},
  author = {Klatt, Michael A. and {Schr{\"o}der-Turk}, Gerd E. and Mecke, Klaus},
  year = {2017},
  journal = {Med. Phys.},
  volume = {44},
  pages = {3663--3675},
  doi = {10.1002/mp.12280},
  urldate = {2017-08-07},
  langid = {english}
}

@article{legland_computation_2011,
  title = {Computation of {{Minkowski}} Measures on 2d and 3d Binary Images},
  author = {Legland, David and Ki{\^e}u, Ki{\^e}n and Devaux, Marie-Fran{\c c}oise},
  year = {2011},
  journal = {Image Anal Stereol},
  volume = {26},
  pages = {83},
  doi = {10.5566/ias.v26.p83-92},
  urldate = {2019-07-08},
  comment_abstract = {Minkowski functionals encompass standard geometric parameters such as volume, area, length and the Euler-Poincar{\'e} characteristic. Software tools for computing approximations of Minkowski functionals on binary 2D or 3D images are now available based on mathematical methods due to Serra, Lang and Ohser. Minkowski functionals can not be used to describe spatial heterogeneity of structures. This description can be performed by using Minkowski measures, which are local versions of Minkowski functionals. In this paper, we discuss how to extend the computation of Minkowski functionals to the computation of Minkowski measures. Approximations of Minkowski measures are computed using filtering and look-up table transformations. The final result is represented as a grey-level image. Approximation errors are investigated based on numerical examples. Convergence and non convergence of the measure approximations are discussed. The measure of surface area is used to describe spatial heterogeneity of a synthetic structure, and of an image of tomato pericarp.},
  comment_file = {2011\textsubscript{L}egland\textsubscript{e}t\textsubscript{a}l\textsubscript{C}OMPUTATION\textsubscript{O}F\textsubscript{M}INKOWSKI\textsubscript{M}EASURES\textsubscript{O}N{$_2$}D\textsubscript{A}ND{$_3$}D\textsubscript{B}INARY\textsubscript{I}MAGES.pdf:/home/mklatt/postdoc-final/references/2011\textsubscript{L}egland\textsubscript{e}t\textsubscript{a}l\textsubscript{C}OMPUTATION\textsubscript{O}F\textsubscript{M}INKOWSKI\textsubscript{M}EASURES\textsubscript{O}N{$_2$}D\textsubscript{A}ND{$_3$}D\textsubscript{B}INARY\textsubscript{I}MAGES.pdf:application/pdf},
  comment_issn = {1854-5165, 1580-3139},
  comment_language = {en},
  comment_month = {05},
  comment_url = {https://www.ias-iss.org/ojs/IAS/article/view/811},
  comment_urldate = {2019-07-08},
  langid = {english}
}

@article{MantzJacobsMecke2008,
  ids = {mantz_utilising_2008},
  title = {Utilising {{Minkowski Functionals for Image Analysis}}},
  author = {Mantz, H. and Jacobs, K. and Mecke, K.},
  year = {2008},
  journal = {J. Stat. Mech.},
  volume = {12},
  pages = {P12015},
  doi = {10.1088/1742-5468/2008/12/P12015}
}

@article{mecke_integral_1998,
  ids = {Mecke1998},
  title = {Integral {{Geometry}} in {{Statistical Physics}}},
  author = {Mecke, K.},
  year = {1998},
  journal = {Int. J. Mod. Phys. B},
  volume = {12},
  pages = {861--899},
  doi = {10.1142/S0217979298000491}
}

@article{mecke_robust_1994,
  ids = {MeckeBuchertWagner1994,mecke_robust_1994-1,mecke_robust_1994-2,mecke_robust_1994-3},
  title = {Robust Morphological Measures for Large-Scale Structure in the Universe},
  author = {Mecke, K. R. and Buchert, {\relax Th}. and Wagner, H.},
  year = {1994},
  journal = {Astron. Astrophys.},
  volume = {288},
  pages = {697}
}

@article {MR3449314,
    AUTHOR = {Hug, Daniel and Last, G\"unter and Schulte, Matthias},
     TITLE = {Second-order properties and central limit theorems for
              geometric functionals of {B}oolean models},
   JOURNAL = {Ann. Appl. Probab.},
  FJOURNAL = {The Annals of Applied Probability},
    VOLUME = {26},
      YEAR = {2016},
    NUMBER = {1},
     PAGES = {73--135},
      ISSN = {1050-5164,2168-8737},
   MRCLASS = {60D05 (52A22 60F05 60G55 60H07)},
  MRNUMBER = {3449314},
MRREVIEWER = {Matthias\ Reitzner},
       DOI = {10.1214/14-AAP1086},
       URL = {https://doi.org/10.1214/14-AAP1086},
}

@book{mecke_statistical_2000,
  ids = {MeckeStoyan2000},
  title = {Statistical Physics and Spatial Statistics: The Art of Analyzing and Modeling Spatial Structures and Pattern Formation},
  editor = {Mecke, Klaus R. and Stoyan, Dietrich},
  year = {2000},
  series = {Lecture Notes in Physics},
  publisher = {{Springer}},
  address = {{Berlin ; New York}},
  langid = {english},
  lccn = {QC174.8 .S72 2000}
}

@ARTICLE{5669298,
  author={Mérigot, Quentin and Ovsjanikov, Maks and Guibas, Leonidas J.},
  journal={IEEE Transactions on Visualization and Computer Graphics}, 
  title={Voronoi-Based Curvature and Feature Estimation from Point Clouds}, 
  year={2011},
  volume={17},
  number={6},
  pages={743-756},
  doi={10.1109/TVCG.2010.261}}

@book{ohser_3d_2009,
  title = {{{3D}} Images of Materials Structures: Processing and Analysis},
  author = {Ohser, Joachim and Schladitz, Katja},
  year = {2009},
  publisher = {{Wiley-VCH}},
  address = {{Weinheim}},
  comment_file = {2009\textsubscript{O}hser\textsubscript{S}chladitz{$_3$}D\textsubscript{i}mages\textsubscript{o}f\textsubscript{m}aterials\textsubscript{s}tructures.pdf:/home/mklatt/postdoc-final/references/2009\textsubscript{O}hser\textsubscript{S}chladitz{$_3$}D\textsubscript{i}mages\textsubscript{o}f\textsubscript{m}aterials\textsubscript{s}tructures.pdf:application/pdf},
  comment_isbn = {978-3-527-31203-0},
  comment_keywords = {Digital techniques, Image analysis, Image processing, Materials science, Microstructure, Three-dimensional imaging},
  comment_note = {OCLC: ocm62226348},
  comment_shorttitle = {3D images of materials structures},
  lccn = {QC173.4.M5 O37 2009}
}

@book{okabe_spatial_2000-1,
  title = {Spatial Tessellations: Concepts and Applications of {{Voronoi}} Diagrams},
  author = {Okabe, Atsuyuki and Boots, Barry  and   Sugihara, Kokichi and Chiu, Sung Nok },
  year = {2000},
  series = {Wiley Series in Probability and Statistics},
  edition = {2nd ed},
  publisher = {{Wiley}},
  address = {{Chichester ; New York}},
  langid = {english},
  lccn = {QA278.2 .O36 2000}
}

@misc{VorominkCode,
  author = {Dominik Pabst},
  title = {Voromink},
  year = {2025},
  howpublished = {\url{https://zenodo.org/records/14614277}},
  note = {DOI: 10.5281/zenodo.14614277}
}

@article{rath_strength_2008,
  ids = {Raeth:2008},
  title = {Strength through Structure: Visualization and Local Assessment of the Trabecular Bone Structure},
  author = {R{\"a}th, C. and Monetti, R. and Bauer, J. and Sidorenko, I. and M{\"u}ller, D. and Matsuura, M. and Lochm{\"u}ller, E.-M. and Zysset, P. and Eckstein, F.},
  year = {2008},
  journal = {New J. Phys.},
  volume = {10},
  pages = {125010},
  publisher = {{IOP Publishing}},
  doi = {10.1088/1367-2630/10/12/125010},
  urldate = {2019-11-21},
  langid = {english}
}

@article{rottger_contactengineeringcreate_2022,
  ids = {rottger_contactengineeringcreate_2022-1},
  title = {Contact.Engineering---{{Create}}, Analyze and Publish Digital Surface Twins from Topography Measurements across Many Scales},
  author = {R{\"o}ttger, Michael C and Sanner, Antoine and Thimons, Luke A and Junge, Till and Gujrati, Abhijeet and Monti, Joseph M and N{\"o}hring, Wolfram G and Jacobs, Tevis D B and Pastewka, Lars},
  year = {2022},
  journal = {Surf. Topogr.: Metrol. Prop.},
  volume = {10},
  pages = {035032},
  publisher = {{IOP Publishing}},
  doi = {10.1088/2051-672X/ac860a},
  urldate = {2022-12-20},
  langid = {english}
}

@article{schmalzing_minkowski_1998,
  title = {Minkowski Functionals Used in the Morphological Analysis of Cosmic Microwave Background Anisotropy Maps},
  author = {Schmalzing, Jens and G{\'o}rski, Krzysztof M.},
  year = {1998},
  journal = {Monthly Notices of the Royal Astronomical Society},
  volume = {297},
  pages = {355--365},
  doi = {10.1046/j.1365-8711.1998.01467.x},
  urldate = {2024-02-29}
}

@book{schneider_stochastic_2008,
  ids = {SchneiderWeil2008},
  title = {Stochastic and {{Integral Geometry}} ({{Probability}} and {{Its Applications}})},
  author = {Schneider, R. and Weil, W.},
  year = {2008},
  publisher = {{Springer}},
  address = {{Berlin}}
}

@article{schroder-turk_minkowski_2011,
  ids = {SchroederTurketal2010AdvMater},
  title = {Minkowski {{Tensor Shape Analysis}} of {{Cellular}}, {{Granular}} and {{Porous Structures}}},
  author = {{Schr{\"o}der-Turk}, G. E. and Mickel, W. and Kapfer, S. C. and Klatt, M. A. and Schaller, F. M. and Hoffmann, M. J. F. and Kleppmann, N. and Armstrong, P. and Inayat, A. and Hug, D. and Reichelsdorfer, M. and Peukert, W. and Schwieger, W. and Mecke, K.},
  year = {2011},
  journal = {Adv. Mater.},
  volume = {23},
  pages = {2535--2553},
  publisher = {{WILEY-VCH Verlag}},
  doi = {10.1002/adma.201100562},
  urldate = {2016-04-25},
  all_authors = {Schr{\"o}der-Turk, G. E. and Mickel, W. and Kapfer, S. C. and Klatt, M. A. and Schaller, F. M. and Hoffmann, M. J. F. and Kleppmann, N. and Armstrong, P. and Inayat, A. and Hug, D. and Reichelsdorfer, M. and Peukert, W. and Schwieger, W. and Mecke, K.},
  comment_doi = {10.1002/adma.201100562},
  comment_issn = {1521-4095},
  langid = {english}
}

@article{schroder-turk_minkowski_2013,
  ids = {SchroederTurkEtAl2013NewJPhy,SchroederTurkNJP2013},
  title = {Minkowski Tensors of Anisotropic Spatial Structure},
  author = {{Schr{\"o}der-Turk}, G. E. and Mickel, W. and Kapfer, S. C. and Schaller, F. M. and Breidenbach, B. and Hug, D. and Mecke, K.},
  year = {2013},
  journal = {New J. Phys.},
  volume = {15},
  pages = {083028},
  publisher = {{IOP Publishing}},
  doi = {10.1088/1367-2630/15/8/083028},
  urldate = {2021-05-07},
  comment_comment_url = {http://stacks.iop.org/1367-2630/15/i=8/a=083028},
  langid = {english}
}

@article{schroder-turk_tensorial_2010,
  ids = {SchroederTurketal:2010jom},
  title = {Tensorial {{Minkowski}} Functionals and Anisotropy Measures for Planar Patterns},
  author = {{Schr{\"o}der-Turk}, G.E. and Kapfer, S. and Breidenbach, B. and Beisbart, C. and Mecke, K.},
  year = {2010},
  journal = {J. Micr.},
  volume = {238},
  pages = {57--74},
  doi = {10.1111/j.1365-2818.2009.03331.x},
  urldate = {2016-06-24},
  langid = {english}
}

@article{spengler_strength_2019,
  title = {Strength of Bacterial Adhesion on Nanostructured Surfaces Quantified by Substrate Morphometry},
  author = {Spengler, Christian and Nolle, Friederike and Mischo, Johannes and Faidt, Thomas and Grandthyll, Samuel and Thewes, Nicolas and Koch, Marcus and M{\"u}ller, Frank and Bischoff, Markus and Klatt, Michael Andreas and Jacobs, Karin},
  year = {2019},
  journal = {Nanoscale},
  volume = {11},
  pages = {19713--19722},
  doi = {10.1039/C9NR04375F},
  urldate = {2019-10-10},
  langid = {english}
}

@article{stinville_multi-modal_2022,
  title = {Multi-Modal {{Dataset}} of a {{Polycrystalline Metallic Material}}: {{3D Microstructure}} and {{Deformation Fields}}},
  author = {Stinville, J. C. and Hestroffer, J. M. and Charpagne, M. A. and Polonsky, A. T. and Echlin, M. P. and Torbet, C. J. and Valle, V. and Nygren, K. E. and Miller, M. P. and Klaas, O. and Loghin, A. and Beyerlein, I. J. and Pollock, T. M.},
  year = {2022},
  journal = {Sci Data},
  volume = {9},
  pages = {460},
  publisher = {{Nature Publishing Group}},
  doi = {10.1038/s41597-022-01525-w},
  urldate = {2023-06-12},
  langid = {english}
}

@article{svane_estimation_2014,
  title = {Estimation of {{Minkowski}} Tensors from Digital Grey-Scale Images},
  author = {Svane, Anne Marie},
  year = {2014},
  journal = {Image Anal. Stereol.},
  volume = {33},
  eprint = {1401.7790},
  pages = {51},
  doi = {10.5566/ias.1124},
  urldate = {2018-07-25},
  archiveprefix = {arxiv}
}

@article{svane_estimation_2014-1,
  ids = {svane_estimation_2014},
  title = {Estimation of {{Intrinsic Volumes}} from {{Digital Grey-Scale Images}}},
  author = {Svane, Anne Marie},
  year = {2014},
  journal = {J Math Imaging Vis},
  volume = {49},
  pages = {352--376},
  doi = {10.1007/s10851-013-0469-9},
  urldate = {2019-12-03},
  arxiv = {[object Object]},
  langid = {english}
}

@book{torquato_random_2002,
  title = {Random {{Heterogeneous Materials}}},
  author = {Torquato, Salvatore},
  year = {2002},
  series = {Interdisciplinary {{Applied Mathematics}}},
  edition = {Second},
  volume = {16},
  publisher = {{Springer}},
  address = {{New York}},
  urldate = {2016-04-25}
}

@book{vanmarcke_random_2010,
  ids = {Vanmarcke2010},
  title = {Random {{Fields}}: {{Analysis}} and {{Synthesis}}},
  author = {Vanmarcke, Erik},
  year = {2010},
  publisher = {{World Scientific}},
  comment_isbn = {9789812563538},
  comment_url = {http://books.google.de/books?id=0MCxDV1bonAC},
  googlebooks = {0MCxDV1bonAC},
  langid = {english},
  lccn = {2011280738}
}

@book{zong_sphere_1999,
  title = {Sphere Packings},
  author = {Zong, Chuanming},
  year = {1999},
  series = {Universitext},
  publisher = {{Springer}},
  address = {{New York}},
  lccn = {QA166.7 .Z66 1999}
}

@article{HLW04,
  title = {A local {S}teiner-type formula for general closed sets and applications},
  author = {Hug, D. and Last, G. and Weil, W.},
  year = {2004},
  journal = {Math. Z. 246},
  pages = {237--272}
}

@article{HS22,
  title = {Curvature measures and soap bubbles beyond convexity},
  author = {Hug, D. and Santilli, M.},
  year = {2022},
  journal = {Advances in Mathematics},
  volume = {411},
  paper = {108802}
}

@incollection{HSVJK,
  author = {Hug, D. and Schneider, R.},
  year = {2017},
  publisher = {Springer},
  pages = {27--65},
  title = {Tensor Valuations and Their Local Versions},
  volume = {2177},
  series = {Lecture Notes in Mathematics},
  booktitle = {Tensor valuations and their applications in stochastic geometry and imaging}
}

@article{HW18,
  title = {Kinematic formulae for tensorial curvature measures},
  author = {Hug, D. and Weis, J.A.},
  year = {2018},
  journal = {Annali di Matematica Pura ed Applicata (1923 -)},
  volume = {197},
  pages = {1349--1384}
}

@book{Hansen2013,
    AUTHOR = {Hansen, Per Christian and Pereyra, V\'ictor and Scherer,
              Godela},
     TITLE = {Least squares data fitting with applications},
 PUBLISHER = {Johns Hopkins University Press, Baltimore, MD},
      YEAR = {2013}
}

@article{KTT19,
  title = {Expected intrinsic volumes and facet numbers of random beta-polytopes},
  author = {Kabluchko, Zakhar and Temesvari, Daniel and Thäle, Christoph},
  year = {2019},
  journal = {Mathematische Nachrichten},
  volume = {292},
  pages = {79--105}
}

@book{HugWeil2020,
    author = {Hug, Daniel and Weil, Wolfgang},
     title = {Lectures on Convex Geometry},
    series = {Graduate Texts in Mathematics},
    volume = {286},
 publisher = {Springer, Cham},
     place = {Cham},
      year = {2020} 
}

@book{B2024,
    AUTHOR = {Bj\"orck, {\AA}ke},
     TITLE = {Numerical methods for least squares problems},
   EDITION = {Second},
 PUBLISHER = {Society for Industrial and Applied Mathematics (SIAM), Philadelphia, PA},
      YEAR = {2024},
}

@book{S14,
  place={Cambridge},
  edition={Second expanded},
  title={Convex Bodies: The Brunn–Minkowski Theory},
  publisher={Cambridge University Press},
  author={Schneider, Rolf},
  year={2014}
}

@incollection{SVJK,
  author = {Schneider, R.},
  year = {2017},
  publisher = {Springer},
  pages = {1--25},
  title = {Valuations on convex bodies: the classical basic facts},
  volume = {2177},
  series = {Lecture Notes in Mathematics},
  booktitle = {Tensor valuations and their applications in stochastic geometry and imaging}
}

@article {MR2594445,
    AUTHOR = {Chazal, Fr\'ed\'eric and Cohen-Steiner, David and M\'erigot,
              Quentin},
     TITLE = {Boundary measures for geometric inference},
   JOURNAL = {Found. Comput. Math.},
  FJOURNAL = {Foundations of Computational Mathematics. The Journal of the
              Society for the Foundations of Computational Mathematics},
    VOLUME = {10},
      YEAR = {2010},
    NUMBER = {2},
     PAGES = {221--240},
      ISSN = {1615-3375,1615-3383},
   MRCLASS = {52A39 (49Q15)},
  MRNUMBER = {2594445},
       DOI = {10.1007/s10208-009-9056-2},
       URL = {https://doi.org/10.1007/s10208-009-9056-2},
}

@article{armstrong_porous_2019,
  ids = {armstrong_porous_2018},
  title = {Porous {{Media Characterization Using Minkowski Functionals}}: {{Theories}}, {{Applications}} and {{Future Directions}}},
  author = {Armstrong, Ryan T. and McClure, James E. and Robins, Vanessa and Liu, Zhishang and Arns, Christoph H. and Schl{\"u}ter, Steffen and Berg, Steffen},
  year = {2019},
  journal = {Transp Porous Med},
  volume = {130},
  pages = {305--335},
  doi = {10.1007/s11242-018-1201-4},
  urldate = {2021-06-18},
  langid = {english}
}

@article{arns_reconstructing_2003,
  title = {Reconstructing {{Complex Materials}} via {{Effective Grain Shapes}}},
  author = {Arns, C. H. and Knackstedt, M. A. and Mecke, K. R.},
  year = {2003},
  journal = {Phys. Rev. Lett.},
  volume = {91},
  pages = {215506},
  publisher = {American Physical Society},
  doi = {10.1103/PhysRevLett.91.215506},
  urldate = {2024-10-14}
}

@article{Arns2010,
  ids = {arns_3d_2010},
  title = {{{3D}} Structural Analysis: Sensitivity of {{Minkowski functionals}}},
  author = {Arns, C. H. and Knackstedt, M. A. and Mecke, K.},
  year = {2010},
  journal = {J. Microsc.},
  volume = {240},
  pages = {181},
  publisher = {Blackwell Publishing Ltd},
  doi = {10.1111/j.1365-2818.2010.03395.x},
  comment_url = {http://dx.doi.org/10.1111/j.1365-2818.2010.03395.x}
}

@article{collischon_morphometry_2024,
  title = {Morphometry on the Sphere: {{Cartesian}} and Irreducible {{Minkowski}} Tensors Explained and Implemented},
  author = {Collischon, Caroline and Klatt, Michael A. and Banday, Anthony J. and Sasaki, Manami and R{\"a}th, Christoph},
  year = {2024},
  journal = {Commun Phys},
  volume = {7},
  pages = {1--10},
  publisher = {Nature Publishing Group},
  doi = {10.1038/s42005-024-01751-1},
  urldate = {2024-07-31},
  langid = {english}
}

@article{jiang_fast_2020,
  title = {Fast {{Fourier}} Transform and Support-Shift Techniques for Pore-Scale Microstructure Classification Using Additive Morphological Measures},
  author = {Jiang, Han and Arns, Christoph H.},
  year = {2020},
  journal = {Phys. Rev. E},
  volume = {101},
  pages = {033302},
  publisher = {American Physical Society},
  doi = {10.1103/PhysRevE.101.033302},
  urldate = {2024-10-14}
}

@book{ohser_statistical_2000,
  title = {Statistical Analysis of Microstructures in Materials Science},
  author = {Ohser, Joachim and M{\"u}cklich, Frank},
  year = {2000},
  series = {Statistics in Practice},
  publisher = {John Wiley},
  address = {Chichester [England]; New York},
  lccn = {TA407 .O35 2000}
}

@article{ziegel_estimating_2015,
  title = {Estimating {{Particle Shape}} and {{Orientation Using Volume Tensors}}},
  author = {Ziegel, Johanna F. and Nyengaard, Jens R. and Vedel Jensen, Eva B.},
  year = {2015},
  journal = {Scand. J. Stat.},
  volume = {42},
  pages = {813--831},
  doi = {10.1111/sjos.12138},
  urldate = {2024-10-14},
  langid = {english}
}

@book {RZ2019,
    AUTHOR = {Rataj, Jan and Z\"ahle, Martina},
     TITLE = {Curvature measures of singular sets},
    SERIES = {Springer Monographs in Mathematics},
 PUBLISHER = {Springer, Cham},
      YEAR = {2019},
     PAGES = {xi+256},
      ISBN = {978-3-030-18182-6; 978-3-030-18183-3},
   MRCLASS = {49Q15 (28A75 28A80 53C65)},
  MRNUMBER = {3932153},
MRREVIEWER = {Lars\ Olsen},
       DOI = {10.1007/978-3-030-18183-3},
       URL = {https://doi.org/10.1007/978-3-030-18183-3},
}

@article {Rat05,
    AUTHOR = {Rataj, J.},
     TITLE = {On boundaries of unions of sets with positive reach},
   JOURNAL = {Beitr\"age Algebra Geom.},
  FJOURNAL = {Beitr\"age zur Algebra und Geometrie. Contributions to Algebra
              and Geometry},
    VOLUME = {46},
      YEAR = {2005},
    NUMBER = {2},
     PAGES = {397--404},
      ISSN = {0138-4821},
   MRCLASS = {52A30 (49Q20 52A22 52A39)},
  MRNUMBER = {2196925},
MRREVIEWER = {Giovanni\ Alberti},
}

@book {JK2017,
     TITLE = {Tensor valuations and their applications in stochastic
              geometry and imaging},
    SERIES = {Lecture Notes in Mathematics},
    VOLUME = {2177},
    EDITOR = {Jensen, Eva B. Vedel and Kiderlen, Markus},
 PUBLISHER = {Springer, Cham},
      YEAR = {2017},
     PAGES = {xiv+460},
      ISBN = {978-3-319-51950-0; 978-3-319-51951-7},
   MRCLASS = {52-06 (60-06 62-06)},
  MRNUMBER = {3726861},
}

@article {KS2021,
    AUTHOR = {Kousholt, Astrid and Schulte, Julia},
     TITLE = {Reconstruction of convex bodies from moments},
   JOURNAL = {Discrete Comput. Geom.},
  FJOURNAL = {Discrete \& Computational Geometry. An International Journal
              of Mathematics and Computer Science},
    VOLUME = {65},
      YEAR = {2021},
    NUMBER = {1},
     PAGES = {1--42},
      ISSN = {0179-5376,1432-0444},
   MRCLASS = {52A20 (47A57 68U10 94A08 94A12)},
  MRNUMBER = {4194435},
MRREVIEWER = {Ge\ Xiong},
       DOI = {10.1007/s00454-020-00225-9},
       URL = {https://doi.org/10.1007/s00454-020-00225-9},
}

@article {Kous2017,
    AUTHOR = {Kousholt, Astrid},
     TITLE = {Reconstruction of {$n$}-dimensional convex bodies from surface
              tensors},
   JOURNAL = {Adv. in Appl. Math.},
  FJOURNAL = {Advances in Applied Mathematics},
    VOLUME = {83},
      YEAR = {2017},
     PAGES = {115--144},
      ISSN = {0196-8858,1090-2074},
   MRCLASS = {52A20 (44A60 53C65 60D05)},
  MRNUMBER = {3573221},
MRREVIEWER = {V.\ K.\ Ohanyan},
       DOI = {10.1016/j.aam.2016.09.004},
       URL = {https://doi.org/10.1016/j.aam.2016.09.004},
}

@article {KK2016,
    AUTHOR = {Kousholt, Astrid and Kiderlen, Markus},
     TITLE = {Reconstruction of convex bodies from surface tensors},
   JOURNAL = {Adv. in Appl. Math.},
  FJOURNAL = {Advances in Applied Mathematics},
    VOLUME = {76},
      YEAR = {2016},
     PAGES = {1--33},
      ISSN = {0196-8858,1090-2074},
   MRCLASS = {52A20 (44A60 52A10 52A38 60D05)},
  MRNUMBER = {3482329},
MRREVIEWER = {Eugenia\ Saor\'in G\'omez},
       DOI = {10.1016/j.aam.2016.01.001},
       URL = {https://doi.org/10.1016/j.aam.2016.01.001},
}

@book{LNP+,
  title = {Geometry and Physics of Spatial Random Systems},
  editor = {Hug, Daniel and Klatt, Michael and Last, Günter and Mecke, Klaus R. and Schröder-Turk, Gerd},
  year = {2025+},
  series = {Lecture Notes in Physics},
  publisher = {{Springer}},
  address = {{Berlin ; New York}},
  lccn = {QC173.458.S75 M67 2002}
}

@article {ACV2008,
    AUTHOR = {Ambrosio, Luigi and Colesanti, Andrea and Villa, Elena},
     TITLE = {Outer {M}inkowski content for some classes of closed sets},
   JOURNAL = {Math. Ann.},
  FJOURNAL = {Mathematische Annalen},
    VOLUME = {342},
      YEAR = {2008},
    NUMBER = {4},
     PAGES = {727--748},
      ISSN = {0025-5831,1432-1807},
   MRCLASS = {28A75 (49Q15)},
  MRNUMBER = {2443761},
MRREVIEWER = {Daniel\ Hug},
       DOI = {10.1007/s00208-008-0254-z},
       URL = {https://doi.org/10.1007/s00208-008-0254-z},
}

@book{preparata_computational_1985,
  title = {Computational {{Geometry}}},
  author = {Preparata, Franco P. and Shamos, Michael Ian},
  year = {1985},
  publisher = {Springer New York},
  address = {New York, NY},
  doi = {10.1007/978-1-4612-1098-6},
  urldate = {2025-09-02}
}

@article{yang_image_1999,
  title = {Image Registration and Object Recognition Using Affine Invariants and Convex Hulls},
  author = {Yang, Zhengwei and Cohen, F.S.},
  year = {1999},
  journal = {IEEE Trans. Image Process.},
  volume = {8},
  pages = {934--946},
  doi = {10.1109/83.772236},
  urldate = {2025-09-02}
}

@article {CLMT15,
    AUTHOR = {Cuel, Louis and Lachaud, Jacques-Olivier and M\'erigot,
              Quentin and Thibert, Boris},
     TITLE = {Robust geometry estimation using the generalized {V}oronoi
              covariance measure},
   JOURNAL = {SIAM J. Imaging Sci.},
  FJOURNAL = {SIAM Journal on Imaging Sciences},
    VOLUME = {8},
      YEAR = {2015},
    NUMBER = {2},
     PAGES = {1293--1314},
      ISSN = {1936-4954},
   MRCLASS = {65D18 (94A08)},
  MRNUMBER = {3354996},
MRREVIEWER = {Domenico\ Vitulano},
       DOI = {10.1137/140977552},
       URL = {https://doi.org/10.1137/140977552},
}

@incollection {LCL17,
    AUTHOR = {Lachaud, Jacques-Olivier and Coeurjolly, David and Levallois,
              J\'er\'emy},
     TITLE = {Robust and convergent curvature and normal estimators with
              digital integral invariants},
 BOOKTITLE = {Modern approaches to discrete curvature},
    SERIES = {Lecture Notes in Math.},
    VOLUME = {2184},
     PAGES = {293--348},
 PUBLISHER = {Springer, Cham},
      YEAR = {2017},
      ISBN = {978-3-319-58001-2; 978-3-319-58002-9},
   MRCLASS = {53A05 (68U05 94A08)},
  MRNUMBER = {3727583},
MRREVIEWER = {Bert\ J\"uttler},
}

@article {McMullen1974,
    AUTHOR = {McMullen, P.},
     TITLE = {On the inner parallel body of a convex body},
   JOURNAL = {Israel J. Math.},
  FJOURNAL = {Israel Journal of Mathematics},
    VOLUME = {19},
      YEAR = {1974},
     PAGES = {217--219},
      ISSN = {0021-2172},
   MRCLASS = {52A20},
  MRNUMBER = {367810},
MRREVIEWER = {L.\ A.\ Santal\'o},
       DOI = {10.1007/BF02757715},
       URL = {https://doi.org/10.1007/BF02757715},
}

\appendix

\section{Further simulation results}\label{Sec:App}

In this section, we provide additional simulation results for the algorithm described in the paper, based on solving a least-squares problem.
The following tables can be found on the subsequent pages:
\begin{itemize}
\item Table \ref{table:Rectangles}: Simulation results for the 2-dimensional rectangle $[-\frac{3}{2},\frac{3}{2}]\times[-\frac{5}{2},\frac{5}{2}]$
\item Table \ref{table:roundVert}: Simulation results for a 2-dimensional rectangle with rounded vertices.
Specifically, we mean the parallel set of a rectangle.
The parallel set of a compact set $K$ with parameter $r_0$ is the set of all points with a distance less than $r_0$ to $K$.
\item Table \ref{table:3dim}: Simulation results for the 3-dimensional rectangle $[-\frac{1}{2},\frac{1}{2}]\times[-1,1]\times[-\frac{3}{2},\frac{3}{2}]$.
\end{itemize}

\begin{table}[H]
\centering
\captionsetup{
  labelfont = {bf},
  format = plain,
  belowskip = 1ex,
  width = \textwidth
}
\caption{Results for the 2-dimensional rectangle $[-\frac{3}{2},\frac{3}{2}]\times[-\frac{5}{2},\frac{5}{2}]$ intersected with a grid of resolution $a=0.01$.
The parameter choices were $n=50$ and $R_{n}=2$.
We took the average of 50 renditions.
All values are rounded to the fourth significant digit.}
\begin{tabular}{ccc|ccc}
  \toprule
  Tensors &  Values & Algorithm & Tensors & Values & Algorithm \\
  \midrule
  \
  \ten 0 0 0 & 1 & 0.9999 & \ten 0 0 1 & 8 & 8.000 \\ [2ex]
  \ten 0 0 2 & 15 & 15.00 & (\ten 1 0 0)$_1$ & 0 & -4.249$\cdot 10^{-5}$ \\ [2ex]
  (\ten 1 0 0)$_2$ & 0 & 5.899$\cdot 10^{-6}$ & (\ten 1 0 1)$_1$ & 0 & 1.232$\cdot 10^{-4}$ \\ [2ex]
  (\ten 1 0 1)$_2$ & 0 & 1.137$\cdot 10^{-4}$ & (\ten 1 0 2)$_1$ & 0 & -6.011$\cdot 10^{-5}$ \\ [2ex]
  (\ten 1 0 2)$_2$ & 0 & -2.237$\cdot 10^{-4}$ & (\ten 0 1 0)$_1$ & 0 & -2.822$\cdot 10^{-5}$ \\ [2ex]
  (\ten 0 1 0)$_2$ & 0 & -2.609$\cdot 10^{-5}$ & (\ten 0 1 1)$_1$ & 0 & 1.988$\cdot 10^{-4}$ \\ [2ex]
  (\ten 0 1 1)$_2$ & 0 & 1.109$\cdot 10^{-4}$ & (\ten 0 1 2)$_1$ & 0 & -3.995$\cdot 10^{-4}$ \\ [2ex]
  (\ten 0 1 2)$_2$ & 0 & 3.05$\cdot 10^{-5}$ & (\ten 0 2 0)$_{1,1}$ & 0.07958 & 0.07958 \\ [2ex]
  (\ten 0 2 0)$_{1,2}$ & 0 & 4.382$\cdot 10^{-6}$ & (\ten 0 2 0)$_{2,2}$ & 0.07958 & 0.07957 \\ [2ex]
  (\ten 0 2 1)$_{1,1}$ & 0.3979 & 0.3979 & (\ten 0 2 1)$_{1,2}$ & 0 & -1.301$\cdot 10^{-5}$ \\ [2ex]
  (\ten 0 2 1)$_{2,2}$ & 0.2387 & 0.2387 & (\ten 0 2 2)$_{1,1}$ & 0 & 5.863$\cdot 10^{-5}$ \\ [2ex]
  (\ten 0 2 2)$_{1,2}$ & 0 & 9.168$\cdot 10^{-6}$ & (\ten 0 2 2)$_{2,2}$ & 0 & 5.234$\cdot 10^{-5}$ \\ [2ex]
  (\ten 2 0 0)$_{1,1}$ & 1.125 & 1.125 & (\ten 2 0 0)$_{1,2}$ & 0 & -6.546$\cdot 10^{-7}$ \\ [2ex]
  (\ten 2 0 0)$_{2,2}$ & 3.125 & 3.125 & (\ten 2 0 1)$_{1,1}$ & 6.75 & 6.750 \\ [2ex]
  (\ten 2 0 1)$_{1,2}$ & 0 & 1.239$\cdot 10^{-5}$ & (\ten 2 0 1)$_{2,2}$ & 14.58 & 14.58 \\ [2ex]
  (\ten 2 0 2)$_{1,1}$ & 5.625 & 5.625 & (\ten 2 0 2)$_{1,2}$ & 0 & -1.561$\cdot 10^{-5}$ \\ [2ex]
  (\ten 2 0 2)$_{2,2}$ & 15.63 & 15.62 & (\ten 1 1 0)$_{1,1}$ & 0.4775 & 0.4774 \\ [2ex]
  (\ten 1 1 0)$_{1,2}$ & 0 & -7.008$\cdot 10^{-6}$ & (\ten 1 1 0)$_{2,2}$ & 0.7958 & 0.7957 \\ [2ex]
  (\ten 1 1 1)$_{1,1}$ & 2.387 & 2.388 & (\ten 1 1 1)$_{1,2}$ & 0 & 1.739$\cdot 10^{-5}$ \\ [2ex]
  (\ten 1 1 1)$_{2,2}$ & 2.387 & 2.388 & (\ten 1 1 2)$_{1,1}$ & 0 & -6.019$\cdot 10^{-4}$ \\ [2ex]
  (\ten 1 1 2)$_{1,2}$ & 0 & -3.425$\cdot 10^{-6}$ & (\ten 1 1 2)$_{2,2}$ & 0 & -5.301$\cdot 10^{-4}$ \\ [2ex]
\bottomrule
\end{tabular}
\begin{tablenotes}
The standard errors of the values were all not greater than 9.5$\cdot 10^{-4}$. The 0 entries are rounded. Their absolute values were all not greater than 6.1$\cdot 10^{-4}$.
\end{tablenotes}
\label{table:Rectangles}
\end{table}

\begin{table}[H]
\centering
\captionsetup{
  labelfont = {bf},
  format = plain,
  belowskip = 1ex,
  width = \textwidth
}
\caption{Results for a 2-dimensional rectangle with rounded vertices intersected with a grid of resolution $a=0.005$.
Here $R_{r_0}$ denotes the parallel set of the rectangle $[-\frac{3}{2},\frac{3}{2}]\times[-\frac{5}{2},\frac{5}{2}]$ with parameter $r_0>0$.
The parameter choices were $n=5$ and $R_{n}=1$.
We took the average of 10 renditions.}
\begin{tabular}{ccc|ccc}
  \toprule
  Tensors &  Values & Algorithm & Tensors & Values & Algorithm \\
  \midrule
  \
  \ten 0 0 0 $(R_{\frac{1}{4}})$ & 0.999 & 1 & \ten 0 0 1 $(R_{\frac{1}{4}})$ & 8.784 & 8.785 \\ [2ex]
  \ten 0 0 2 $(R_{\frac{1}{4}})$ & 19.194 & 19.196 & \ten 0 2 1 $(R_{\frac{1}{4}})$ & $\begin{pmatrix} 0.269 & 0 \\ 0 & 0.429 \end{pmatrix}$ & $\begin{pmatrix} 0.270 & 0 \\ 0 & 0.429 \end{pmatrix}$ \\ [5ex]
  
  \ten 0 0 0 $(R_{\frac{1}{2}})$ & 0.999 & 1 & \ten 0 0 1 $(R_{\frac{1}{2}})$ & 9.571 & 9.571 \\ [2ex]
  \ten 0 0 2 $(R_{\frac{1}{2}})$ & 23.782 & 23.785 & \ten 0 2 1 $(R_{\frac{1}{2}})$ & $\begin{pmatrix} 0.301 & 0 \\ 0 & 0.460 \end{pmatrix}$ & $\begin{pmatrix} 0.301 & 0 \\ 0 & 0.460 \end{pmatrix}$ \\ [5ex]
  
  \ten 0 0 0 $(R_1)$ & 0.999 & 1 & \ten 0 0 1 $(R_1)$ & 11.143 & 11.142 \\ [2ex]
  \ten 0 0 2 $(R_1)$ & 34.134 & 34.142 & \ten 0 2 1 $(R_1)$ & $\begin{pmatrix} 0.363 & 0 \\ 0 & 0.522 \end{pmatrix}$ & $\begin{pmatrix} 0.364 & 0 \\ 0 & 0.523 \end{pmatrix}$ \\ [2ex]
  \bottomrule
\end{tabular}
\begin{tablenotes}
\end{tablenotes}
\label{table:roundVert}
\end{table}

\begin{table}[H]
\centering
\captionsetup{
  labelfont = {bf},
  format = plain,
  belowskip = 1ex,
  width = \textwidth
}
\caption{Results for the 3-dimensional rectangle $[-\frac{1}{2},\frac{1}{2}]\times[-1,1]\times[-\frac{3}{2},\frac{3}{2}]$ intersected with a grid of resolution $a=0.01$.
The parameter choices were $n=5$ and $R_{n}=1$.
We took the average of 28 renditions.}
\begin{tabular}{ccc|ccc}
  \toprule
  Tensors &  Values & Algorithm & Tensors & Values & Algorithm \\
  \midrule
  \
  \ten 0 0 0 & 1 & 1.005 & \ten 0 0 1 & 6 & 5.986 \\ [2ex]
  \ten 0 0 2 & 11 & 11.051 & (\ten 0 2 0)$_{1,1}$ & 0.080 & 0.075 \\ [2ex]
  (\ten 0 2 0)$_{2,2}$ & 0.080 & 0.079 & (\ten 0 2 0)$_{3,3}$ & 0.080 & 0.078 \\ [2ex]
  (\ten 0 2 1)$_{1,1}$ & 0.398 & 0.406 & (\ten 0 2 1)$_{2,2}$ & 0.318 & 0.319 \\ [2ex]
  (\ten 0 2 1)$_{3,3}$ & 0.239 & 0.243 & (\ten 0 2 2)$_{1,1}$ & 0.478 & 0.472 \\ [2ex]
  (\ten 0 2 2)$_{2,2}$ & 0.239 & 0.238 & (\ten 0 2 2)$_{3,3}$ & 0.159 & 0.156 \\ [2ex]
  (\ten 2 0 0)$_{1,1}$ & 0.125 & 0.126 & (\ten 2 0 0)$_{2,2}$ & 0.5 & 0.500 \\ [2ex]
  (\ten 2 0 0)$_{3,3}$ & 1.125 & 1.126 & (\ten 2 0 1)$_{1,1}$ & 0.667 & 0.665 \\ [2ex]
  (\ten 2 0 1)$_{2,2}$ & 2.333 & 2.333 & (\ten 2 0 1)$_{3,3}$ & 4.5 & 4.497 \\ [2ex]
  (\ten 2 0 2)$_{1,1}$ & 0.958 & 0.960 & (\ten 2 0 2)$_{2,2}$ & 2.833 & 2.835 \\ [2ex]
  (\ten 2 0 2)$_{3,3}$ & 5.625 & 5.629 & (\ten 2 0 3)$_{1,1}$ & 0.25 & 0.249 \\ [2ex]
  (\ten 2 0 3)$_{2,2}$ & 1 & 0.999 & (\ten 2 0 3)$_{3,3}$ & 2.25 & 2.248 \\ [2ex]
  (\ten 1 1 0)$_{1,1}$ & 0.159 & 0.159 & (\ten 1 1 0)$_{2,2}$ & 0.318 & 0.319 \\ [2ex]
  (\ten 1 1 0)$_{3,3}$ & 0.478 & 0.475 & (\ten 1 1 1)$_{1,1}$ & 0.796 & 0.795 \\ [2ex]
  (\ten 1 1 1)$_{2,2}$ & 1.273 & 1.272 & (\ten 1 1 1)$_{3,3}$ & 1.432 & 1.437 \\ [2ex]
  (\ten 1 1 2)$_{1,1}$ & 0.955 & 0.958 & (\ten 1 1 2)$_{2,2}$ & 0.955 & 0.956 \\ [2ex]
  (\ten 1 1 2)$_{3,3}$ & 0.955 & 0.952 & & & \\ [2ex]
\bottomrule
\end{tabular}
\begin{tablenotes}
\end{tablenotes}
\label{table:3dim}
\end{table}

\end{document}